\documentclass[11pt]{article}
\usepackage{amsfonts}
\usepackage{latexsym}
\usepackage{amsmath}
\usepackage{amssymb}
\usepackage{amsthm}
\usepackage{amsmath}
\usepackage{hyperref}
\usepackage{verbatim}
\usepackage{bibspacing}

\newtheorem{thm}{Theorem}[section]
\newtheorem*{thm*}{Theorem}
\newtheorem{cor}[thm]{Corollary}

\newtheorem{example}[thm]{Example}
\newtheorem{lem}[thm]{Lemma}

\newtheorem{prop}[thm]{Proposition}
\newtheorem*{prop*}{Proposition}

\newtheorem{conj}[thm]{Conjecture}
\newtheorem*{conj*}{Conjecture}

\newtheorem*{defn*}{Definition}
\theoremstyle{definition}
\newtheorem{rem}[thm]{\textbf{Remark}}
\newtheorem*{rmk*}{Remark}
\newtheorem*{fact*}{Fact}

\theoremstyle{proof}

\newcommand{\Var}{\textrm{Var}}

\newcommand{\norm}[1]{\left\Vert#1\right\Vert}
\newcommand{\snorm}[1]{\Vert#1\Vert}
\newcommand{\abs}[1]{\left\vert#1\right\vert}
\newcommand{\set}[1]{\left\{#1\right\}}
\newcommand{\brac}[1]{\left(#1\right)}
\newcommand{\scalar}[1]{\left \langle #1 \right \rangle}

\newcommand{\Real}{\mathbb{R}}

\newcommand{\II}{\text{II}}
\newcommand{\M}{\mathcal{M}}
\renewcommand{\H}{\mathcal{H}}
\newcommand{\eps}{\epsilon}
\newcommand{\K}{\mathcal{K}}
\renewcommand{\S}{\mathcal{S}}
\newcommand{\B}{B_H}
\newcommand{\BNH}{B}
\newcommand{\D}{D}
\renewcommand{\SS}{S}
\newcommand{\sgn}{\text{sgn}}
\newcommand{\unc}{\text{uncond}}

\newlength{\defbaselineskip}
\numberwithin{equation}{section}

\begin{document}

\title{Local $L^p$-Brunn--Minkowski inequalities for $p < 1$}
\date{}

\author{Alexander V. Kolesnikov\textsuperscript{1} and Emanuel Milman\textsuperscript{2}}

\footnotetext[1]{Faculty of Mathematics, Higher School of Economics, Moscow, Russia. 
The  author was supported by RFBR project 17-01-00662 and DFG project RO 1195/12-1. 
The article was prepared within the framework of the Academic Fund Program at the National Research University Higher
School of Economics (HSE) in 2017--2018 (grant No 17-01-0102) and by the Russian Academic Excellence Project ``5-100".
Emails: akolesnikov@hse.ru, sascha77@mail.ru.}

\footnotetext[2]{Department of Mathematics, Technion - Israel
Institute of Technology, Haifa 32000, Israel. 
The research leading to these results is part of a project that has received funding from the European Research Council (ERC) under the European Union's Horizon 2020 research and innovation programme (grant agreement No 637851). Email: emilman@tx.technion.ac.il.}

\begingroup    \renewcommand{\thefootnote}{}    
\endgroup

\maketitle

\begin{abstract}
The $L^p$-Brunn--Minkowski theory for $p\geq 1$, proposed by Firey and developed by Lutwak in the 90's, replaces the Minkowski addition of convex sets by its $L^p$ counterpart, in which the support functions are added in $L^p$-norm. Recently, B\"{o}r\"{o}czky, Lutwak, Yang and Zhang have proposed to extend this theory further to encompass the range $p \in [0,1)$. In particular, they conjectured an $L^p$-Brunn--Minkowski inequality for origin-symmetric convex bodies in that range, which constitutes a strengthening of the classical Brunn-Minkowski inequality. Our main result confirms this conjecture locally for all (smooth) origin-symmetric convex bodies in $\Real^n$ and $p \in [1 - \frac{c}{n^{3/2}},1)$. In addition, we confirm the local log-Brunn--Minkowski conjecture (the case $p=0$) for small-enough $C^2$-perturbations of the unit-ball of $\ell_q^n$ for $q \geq 2$, when the dimension $n$ is sufficiently large, as well as for the cube, which we show is the conjectural extremal case. For unit-balls of $\ell_q^n$ with $q \in [1,2)$, we confirm an analogous result for $p=c \in (0,1)$, a universal constant. It turns out that the local version of these conjectures is equivalent to a minimization problem for a spectral-gap parameter associated with a certain differential operator, introduced by Hilbert (under different normalization) in his proof of the Brunn--Minkowski inequality. 
As applications, we obtain local uniqueness results in the even $L^p$-Minkowski problem, as well as improved stability estimates in the Brunn--Minkowski and anisotropic isoperimetric inequalities. 
\end{abstract}

\tableofcontents

\bigskip

\bigskip

\section{Introduction}

The celebrated Brunn--Minkowski inequality \cite{Schneider-Book,GardnerSurveyInBAMS} states that if $K_0,K_1$ are two convex sets in $\Real^n$ then:
\begin{equation} \label{eq:BM}
V((1-\lambda) K_0+\lambda K_1)^{\frac{1}{n}} \geq (1-\lambda) V(K_0)^{\frac{1}{n}} + \lambda V(K_1)^{\frac{1}{n}} \;\;\; \forall \lambda \in [0,1] . 
\end{equation}
Here $V$ denotes Lebesgue measure (volume) and
\[
 (1-\lambda) K_0 + \lambda K_1 = \set{(1-\lambda) a+ \lambda b \; ; \; a\in K_0, b \in K_1}
 \]
 denotes Minkowski addition (or interpolation). This inequality and its generalizations lie at the heart 
of the Brunn--Minkowski theory of convex sets, which is by now a classical object of study, having applications in a multitude of other fields \cite{Schneider-Book,GardnerGeometricTomography2ndEd,BuragoZalgallerBook,GardnerSurveyInBAMS}. 

\medskip
In the 60's, W.~J.~Firey \cite{Firey-Sums} proposed an $L^p$ extension ($p \in [1,\infty]$) of the Minkowski addition operation, the so-called  $L^p$-Firey--Minkowski addition, or simply $L^p$-sum. To describe it, let $h_K$ denote the support function of a convex body $K$ (see Section \ref{sec:notation} for definitions). If $K_0,K_1$ are convex compact sets containing the origin in their interior (``convex bodies"), their $L^p$ interpolation, denoted $(1-\lambda) \cdot K_0 +_p \lambda \cdot K_1$, is defined when $p \geq 1$ and $\lambda \in [0,1]$ as the convex body with support function:
\begin{equation} \label{eq:intro-supp}
h_{(1-\lambda) \cdot K_0 +_p \lambda \cdot K_1} = (1-\lambda) \cdot h_{K_0} +_p \lambda \cdot h_{K_1} := ((1-\lambda)  h_{K_0}^p + \lambda  h_{K_1}^p)^{\frac{1}{p}} . 
\end{equation}
The case $p=1$ corresponds to the usual Minkowski sum. Note that $\lambda \cdot K = \lambda^{\frac{1}{p}} K$, and the implicit dependence of $\cdot$ on $p$ is suppressed. 
 Firey established the following $L^p$-Brunn--Minkowski inequality when $p \geq 1$:
\begin{equation} \label{eq:Firey-p-BM}
V((1-\lambda) \cdot K_0+_p \lambda \cdot K_1)^{\frac{p}{n}} \geq (1-\lambda) V(K_0)^{\frac{p}{n}} + \lambda V(K_1)^{\frac{p}{n}} \;\;\; \forall \lambda \in [0,1] ,
\end{equation}
which turns out to be a consequence of the classical $p=1$ case (\ref{eq:BM}) by a simple application of Jensen's inequality. The resulting $L^p$-Brunn--Minkowski theory was extensively developed by E.~Lutwak \cite{Lutwak-Firey-Sums,Lutwak-Firey-Sums-II}, leading to a rich theory with many parallels to the classical one (see also \cite{HaberlLYZ-OrliczMinkowskiProblem} and the references therein for further extensions to more general Orlicz norms).  

\medskip

Fairly recently, K.~B\"{o}r\"{o}czky, Lutwak, D.~Yang and G.~Zhang \cite{BLYZ-logMinkowskiProblem,BLYZ-logBMInPlane} have proposed to extend the $L^p$-Brunn--Minkowski theory to the range $p \in [0,1)$. 
To describe their extension, recall that the Aleksandrov body (or Wulff shape) associated to a positive (Wulff) function $w \in C(S^{n-1})$, is defined as the following convex body:
\[
A[w] := \bigcap_{\theta \in S^{n-1}} \set{ x \in \Real^n \; ; \; \scalar{x,\theta} \leq w(\theta) } . 
\]
In other words, $A[w]$ it the largest convex body $K$ so that $h_K \leq w$. 
While the right-hand-side of (\ref{eq:intro-supp}) is no longer a support function in general when $p \in [0,1)$, B\"{o}r\"{o}czky--Lutwak--Yang--Zhang defined:
\[
(1-\lambda) \cdot K_0 +_p \lambda \cdot K_1 := A[ ((1-\lambda)  h_{K_0}^p + \lambda  h_{K_1}^p)^{\frac{1}{p}}] ,
\]
interpreting the case $p=0$ in the limiting sense as:
\[
(1-\lambda) \cdot K_0 +_0 \lambda \cdot K_1 := A[ h_{K_0}^{1-\lambda} h_{K_1}^\lambda ] .
\]
This of course coincides with Firey's definition when $p \geq 1$. 
With this notation, they proposed the following conjectural extension of (\ref{eq:Firey-p-BM}) when $p \in [0,1)$ and $n \geq 2$ (the case $n=1$ is trivial):

\begin{conj*}[$L^p$-Brunn--Minkowski Conjecture given \textbf{$p \in [0,1)$}] 
The $L^p$-Brunn--Minkowski inequality holds for all \textbf{origin-symmetric} convex bodies $K_0,K_1$ in $\Real^n$:
\begin{equation} \label{eq:p-BM}
V((1-\lambda) \cdot K_0+_p \lambda \cdot K_1) \geq \brac{(1-\lambda) V(K_0)^{\frac{p}{n}} + \lambda V(K_1)^{\frac{p}{n}}}^{\frac{n}{p}} \;\;\; \forall \lambda \in [0,1] .
\end{equation}
In the case $p=0$, called the log-Brunn--Minkowski Conjecture, (\ref{eq:p-BM}) is interpreted in the limiting sense as:
\begin{equation} \label{eq:log-BM}
V((1-\lambda) \cdot K_0+_0 \lambda \cdot K_1) \geq V(K_0)^{1-\lambda} V(K_1)^{\lambda} \;\;\; \forall \lambda \in [0,1] .
\end{equation}
\end{conj*}
The motivation for considering this question in \cite{BLYZ-logBMInPlane} was an equivalence established by B\"{o}r\"{o}czky--Lutwak--Yang--Zhang between the conjectured (\ref{eq:log-BM}) for all origin-symmetric convex bodies, and the \emph{uniqueness} question in the even log-Minkowski problem, the $p=0$ analogue of the classical Minkowski problem (for the \emph{existence} question, a novel necessary and sufficient condition was obtained in \cite{BLYZ-logMinkowskiProblem}, see Section \ref{sec:Mink} or Subsection \ref{subsec:applications} below). The restriction to origin-symmetric convex bodies is an interesting feature of the above conjectures, which we will elucidate in this work (see Subsection \ref{subsec:LK}); for now, let us just remark that the conjectures are known to be false for general convex bodies, as seen by selecting $K_1$ to be an infinitesimal translated version of an origin-symmetric convex body $K_0$. As before, it is easy to check using Jensen's inequality that the above conjectures become stronger as $p$ decreases from $1$ to $0$, with the strongest case (implying the rest) being the $p=0$ one. For general origin-symmetric convex bodies, there is no point to consider the $L^p$-Brunn--Minkowski conjecture for $p < 0$, since (\ref{eq:p-BM}) is then easily seen to be false when $K_0 = \times_{i=1}^n [-a_i,a_i]$ and $K_1 = \times_{i=1}^n [-b_i,b_i]$ are two non-homothetic origin-symmetric cubes. However, other particular bodies $K_0,K_1$ may satisfy (\ref{eq:p-BM}) with $p < 0$, explaining the convention of putting the exponent $\frac{n}{p}$ on the right-hand-side (when $p > 0$, this makes no difference).

\medskip

Being a conjectural strengthening of the ubiquitous Brunn--Minkowski inequality, establishing the validity of the $L^p$-Brunn--Minkowski conjecture, for any $p$ strictly smaller than $1$, would be of fundamental importance to the Brunn--Minkowski theory and its numerous applications. 

\subsection{Previously Known Partial Results}

The following partial results in regards to the log-Brunn--Minkowski conjecture have been previously obtained:
\begin{itemize}
\item B\"{o}r\"{o}czky--Lutwak--Yang--Zhang \cite{BLYZ-logBMInPlane} confirmed the conjecture in the plane $\Real^2$. See also Ma \cite{MaLogBMInPlane} for an alternative proof, and Xi--Leng \cite{XiLeng-DarAndLogBMInPlane} for an extension of this result which does not require origin-symmetry. 
\item C.~Saroglou \cite{Saroglou-logBM1} verified the conjecture when $K_0,K_1 \subset \Real^n$ are both simultaneously unconditional with respect to the same orthogonal basis, meaning that they are invariant under reflections with respect to the principle coordinate hyperplanes $\set{x_i=0}$. In that case, a stronger version of the conjecture follows from a multiplicative version of the Pr\'ekopa--Leindler inequality on the positive orthant (e.g. \cite[Proposition 10]{CFM-BConjecture}).
\item L.~Rotem \cite{Rotem-logBM} observed that the conjecture for complex convex bodies $K_0,K_1 \subset \mathbb{C}^n$ follows from a more general theorem of D.~Cordero--Erausquin \cite{CorderoErausquin-ComplexSantalo}.
\item A.~Colesanti, G.~Livshyts and A.~Marsiglietti \cite{CLM-LogBMForBall} verified the conjecture locally for small-enough $C^2$-perturbations of the Euclidean ball $B_2^n$ (see below for more details). 
\item The log-Brunn--Minkowski conjecture has been shown by Saroglou \cite{Saroglou-logBM1,Saroglou-logBM2} to be intimately related to the generalized $B$-conjecture. Further connections 
to a conjecture of R.~Gardner and A.~Zvavitch \cite{GardnerZvavitch-GaussianBM} on a dimensional Brunn--Minkowski inequality for even log-concave measures were observed by Livshyts, Marsiglietti, P.~Nayar and Zvavitch \cite{LMNZ-BMforMeasures}. A surprising relation to a conjecture of S.~Dar \cite{DarConjecture} was observed by Xi and Leng in \cite{XiLeng-DarAndLogBMInPlane}. 
\end{itemize}

\subsection{Main Results}

Our first main result in this work confirms the following \emph{local} version of the $L^p$-Brunn--Minkowski conjecture for a certain range of $p$'s strictly smaller than $1$. Let $\K^2_{+}$ denote the class of convex bodies with $C^2$ smooth boundary and strictly positive curvature, and let $\K^2_{+,e}$ denote its origin-symmetric members. We reserve the symbols $c,C,C'$ etc.. to denote positive universal numeric constants, independent of any other parameter.  

\begin{thm}[Local $L^p$-Brunn--Minkowski] \label{thm:intro-main}
Let $n \geq 2$ and $p \in [1-\frac{c}{n^{3/2}},1)$. Then for any $K_0,K_1 \in \K^2_{+,e}$ so that:
\begin{equation} \label{eq:intro-smooth-geodesic}
 (1-\lambda) \cdot K_0 +_p \lambda \cdot K_1 \in \K^2_{+,e} \;\;\; \forall \lambda \in [0,1]  ~,
\end{equation}
the $L^p$-Brunn--Minkowski conjecture (\ref{eq:p-BM}) for $K_0,K_1$ holds true:
\[
V((1-\lambda) \cdot K_0 +_p \lambda \cdot K_1) \geq \brac{(1-\lambda) V(K_0)^{\frac{p}{n}} + \lambda V(K_1)^{\frac{p}{n}}}^{\frac{n}{p}} \;\;\; \forall \lambda \in [0,1] . 
\]
\end{thm}

The condition (\ref{eq:intro-smooth-geodesic}) is deceptively appealing: while for $p \geq 1$ it is always automatically satisfied, this is not the case in general for $p \in [0,1)$ (see Corollary \ref{cor:Aleksandrov}). Consequently, we do not know how to conclude the validity of the (global) $L^p$-Brunn--Minkowski conjecture for all origin-symmetric $K_0,K_1$ and $p$ in the above range. 
On the other hand, for every $K \in \K^2_{+,e}$, there exists a $C^2$-neighborhood $N_{K}$ so that for all $K_0,K_1 \in N_K$, (\ref{eq:intro-smooth-geodesic}) is satisfied, so we can  confirm the $L^p$-Brunn--Minkowski conjecture locally. See also Section \ref{sec:local-global} for an equivalent local formulation in terms of the second variation $\frac{d^2}{(d\eps)^2} V(A[h_K +_p \eps \cdot f])^{\frac{p}{n}} \leq 0$ for appropriate test functions $f$. 

\medskip

Let us mention here that we conjecture the logical equivalence between the local and global formulations of the $L^p$-Brunn--Minkowski conjecture - see Conjecture \ref{conj:local-global} and the preceding discussion. This is an extremely interesting and tantalizing question, which leads to the study of the second variation of volume of an Aleksandrov body $\frac{d^2}{(d\eps)^2} V(A[h_K + \eps f])$, when no smoothness is assumed on $K$ or $f$. While the first variation has been studied by Aleksandrov himself and is well understood (see \cite[Lemma 6.5.3]{Schneider-Book}), to the best of our knowledge, nothing concrete is known regarding the second variation when $K$ is not assumed smooth, even when $f$ is the difference of two support functions. 

\medskip

Our second type of results pertain to the verification of the local log-Brunn--Minkowski conjecture (case $p=0$) for various classes of origin-symmetric convex bodies. For instance, we obtain the following result regarding $B_q^n$, the unit-ball of $\ell_q^n$. Note that $B_q^n \notin \K^2_{+,e}$ unless $q = 2$. To emphasize the invariance of the conjecture under non-degenerate linear transformations $GL_n$, we state it explicitly.

\begin{thm}[Local log-Brunn--Minkowski for $B_q^n$]\label{thm:intro-lq}
For all $q \in [2,\infty)$, there exists $n_q \geq 2$ so that for all $n \geq n_q$, there exists a neighborhood $N^C_{B_q^n}$ of $B_q^n$ in the Hausdorff topology, so that for any $K \in N^C_{B_q^n} \cap \K^2_{+,e}$, there exists a $C^2$-neighborhood $N_K$ of $K$ in $\K^2_{+,e}$ so that for all $T \in GL_n$ and $K_1,K_0 \in T(N_{K})$, the log-Brunn--Minkowski conjecture (\ref{eq:log-BM}) for $K_0,K_1$ holds true:
\[
V((1-\lambda) \cdot K_0 +_0 \lambda \cdot K_1) \geq V(K_0)^{1-\lambda} V(K_1)^{\lambda} \;\;\; \forall \lambda \in [0,1] . 
\]
\end{thm}

For more general statements, see Theorems \ref{thm:Ball}, \ref{thm:dual-BL}, \ref{thm:B-estimate}, \ref{thm:BBqn} and \ref{thm:ellq}. The two extremal cases above are of particular interest. For the Euclidean ball $B_2^n$, Theorem \ref{thm:intro-lq} in fact holds with $n_2=2$, i.e. in all dimensions $n \geq 2$. In this particular case, Theorem \ref{thm:intro-lq} strengthens \cite[Theorem 1.4]{CLM-LogBMForBall} of Colesanti--Livshyts--Marsiglietti, where it was assumed that $K_0,K_1,B_2^n$ are ``co-linear", i.e. that $K_0 = (1-\eps) \cdot K_1 +_0 \eps \cdot B_2^n$ for some $\eps \in [0,1]$; in addition, the linear invariance under $GL_n$ was not noted there (strictly speaking, the result in \cite{CLM-LogBMForBall} does not yield a $C^2$-neighborhood of $B_2^n$ for which the latter statement holds, since the allowed $C^2$-distance of $K_1$ from $B_2^n$ depended on $K_1$, but the latter caveat was very recently remedied by Colesanti and Livshyts in \cite{ColesantiLivshyts-LocalpBMUniquenessForBall}, concurrently to our work). 

\medskip

On the other extreme lies the unit-cube $B_\infty^n$. In an appropriate sense, described in the next subsection, the cube turns out to be the conjectural extremal case for the log-Brunn--Minkowski inequality. 
The local extremality can be seen as follows (see Theorems \ref{thm:B-Cube} and \ref{thm:lambda1-cube} for a more streamlined formulation):

\begin{thm}[Extremal local log-Brunn--Minkowski for $B_\infty^n$] \label{thm:intro-cube}
For any $\set{K^i}_{i \geq 1} \subset \K^2_{+,e}$ which approximate $B_\infty^n$ in the Hausdorff metric, there exist $\set{p_i}_{i \geq 1}$ converging to $0$, with the following property: for all $i \geq 1$, there exists a $C^2$-neighborhood $N_{K^i}$ of $K^i$ in $\K^2_{+,e}$, so that for all $T \in GL_n$ and $K^i_0,K^i_1 \in T(N_{K^i})$, we have:
\[
V((1-\lambda) \cdot K^i_0 +_{p_i} \lambda \cdot K^i_1) \geq \brac{(1-\lambda) V(K^i_0)^{\frac{p_i}{n}} + \lambda V(K^i_1)^{\frac{p_i}{n}}}^{\frac{n}{p_i}} \;\;\; \forall \lambda \in [0,1] . 
\]
Moreover, if $K_i = B_{q_i}^n$ with $q_i \rightarrow \infty$, it is impossible to find $\set{p_i}$ with the above property, which instead of converging to zero, satisfy $\liminf p_i < 0$. 
\end{thm}

We do not know how to handle the range $q \in [1,2)$ in Theorem \ref{thm:intro-lq}. In particular, it would be very interesting to establish analogues of Theorem \ref{thm:intro-lq} or Theorem \ref{thm:intro-cube} (without the ``moreover" part) for $B_1^n$. We can however obtain for any unconditional convex body, and in particular for $B_1^n$, a local \emph{strengthening} of Saraoglu's confirmation of the log-Brunn--Minkowski conjecture for unconditional convex bodies -- see Theorem \ref{thm:unc} and Corollary \ref{cor:lambda1-unc}. We can also show (see Theorem \ref{thm:ellq1}):

\begin{thm}[Local $L^c$-Brunn--Minkowski for $B_q^n$]\label{thm:intro-ellq1}
There exists a universal constant $c \in (0,1)$, so that for all $q \in [1,2)$, there exists a neighborhood $N^C_{B_q^n}$ of $B_q^n$ in the Hausdorff topology, so that for any $K \in N^C_{B_q^n} \cap \K^2_{+,e}$, there exists a $C^2$-neighborhood $N_K$ of $K$ in $\K^2_{+,e}$ so that for all $T \in GL_n$ and $K_1,K_0 \in T(N_{K})$, the $p$-Brunn--Minkowski conjecture (\ref{eq:p-BM}) for $K_0,K_1$ holds true with $p=c$:
\[
V((1-\lambda) \cdot K_0 +_c \lambda \cdot K_1) \geq \brac{(1-\lambda) V(K_0)^{\frac{c}{n}} + \lambda V(K_1)^{\frac{c}{n}}}^{\frac{n}{c}} \;\;\; \forall \lambda \in [0,1] . 
\]
\end{thm}

\smallskip
An application of the above results to local uniqueness in the even $L^p$-Minkowski problem is presented in Section \ref{sec:Mink}. 

\subsection{Spectral Interpretation via the Hilbert--Brunn--Minkowski operator}

An additional contribution of this work lies in revealing the connection between the local $L^p$-Brunn--Minkowski conjecture and a spectral-gap property of a certain second-order elliptic operator $L_K$ on $S^{n-1}$ associated to any $K \in \K^2_+$.
Modulo the different normalization we employ in our investigation, this operator was in fact considered by Hilbert in his proof of the classical Brunn--Minkowski inequality, 
and subsequently generalized by Aleksandrov in his second proof of the Aleksandrov--Fenchel inequality \cite[pp. 108--110]{BonnesenFenchelBook}. Consequently, we call $L_K$ the 
Hilbert--Brunn--Minkowski operator. Our normalization has several advantages over the one employed by Hilbert (see Remark \ref{rem:Hilbert}); for instance, it ensures an important equivariance property of the correspondence $K \mapsto L_K$ under linear transformations (see Theorem \ref{thm:LTK}), which to the best of our knowledge was previously unnoted.

\medskip

Let us denote by $\lambda_1(-L_K)$ the spectral-gap of $-L_K$ beyond the trivial $0$ eigenvalue, and by  $\lambda_{1,e}(-L_K)$ the \emph{even} spectral-gap beyond $0$ when restricting to \emph{even} functions. By the previous comments, both are linear invariants of $K$. It was shown by Hilbert that (with our normalization) the Brunn--Minkowski inequality (\ref{eq:BM}) is equivalent to the uniform spectral-gap estimate $\lambda_1(-L_K) \geq 1$ for all $K \in \K^2_+$; moreover, Hilbert showed that $\lambda_1(-L_K) = 1$, with the corresponding eigenspace being precisely the one spanned by (normalized) linear functions, generated by translations of $K$. A natural question is then:
\[
\text{is there a \emph{uniform} spectral-gap of $-L_K$ for all $K \in \K^2_+$ \emph{beyond} $1$?}
\]
It turns out that the local $L^p$-Brunn--Minkowski conjecture for $K \in \K^2_{+,e}$ and $p \in [0,1)$ is precisely equivalent to the conjecture that:
\begin{equation} \label{eq:intro-lambda1}
\lambda_{1,e}(-L_K) \geq \frac{n-p}{n-1} \; ( \; > 1 ). 
\end{equation}
In other words, it is a ``next eigenvalue" conjecture (similar in spirit to the B-conjecture for the Gaussian measure, established in \cite{CFM-BConjecture}). This elucidates the requirement that $K$ be origin-symmetric, since then even functions are automatically orthogonal to the odd (normalized) linear functions; moreover, this interpretation suggests a plausible extension of the conjecture to non-symmetric convex bodies (see Remark \ref{rem:extension-to-non-sym}). Theorem \ref{thm:intro-main} asserts that (\ref{eq:intro-lambda1}) holds with $p = 1 - \frac{c}{n^{3/2}}$, answering in the positive the question of whether there is a \emph{uniform even} spectral-gap for $-L_K$ beyond $1$.

\medskip

It is easy to calculate that $\lambda_{1,e}(-L_{B_2^n}) = \frac{2 n}{n-1}$ (corresponding to $p=-n$ in (\ref{eq:intro-lambda1})), yielding Theorem \ref{thm:intro-lq} for $q=2$ and $n_2 = 2$. A much greater challenge is to establish that:
\[
\liminf_{\K^2_{+,e} \ni K^i \rightarrow B_\infty^n \text{ in Hausdorff metric}} \lambda_{1,e}(-L_{K^i}) = \frac{n}{n-1} ,
\]
corresponding to $p=0$ in (\ref{eq:intro-lambda1}), and yielding Theorem \ref{thm:intro-cube}. In other words, the local log-Brunn--Minkowski conjecture is equivalent to the conjecture that the cube $B_\infty^n$ is a minimizer of the linearly invariant spectral parameter $\lambda_{1,e}(-L_K)$ over all origin-symmetric convex bodies $K\in \K^2_{+,e}$. In our opinion, this provides the most convincing and natural reason for believing the validity of the local log-Brunn--Minkowski conjecture, and emphasizes its importance. We also conjecture that the Euclidean ball $B_2^n$ is in fact a maximizer of $\lambda_{1,e}(-L_K)$. 

\subsection{Method of Proof}

Our main tool for obtaining estimates on $\lambda_{1,e}(-L_K)$ is the Reilly formula, which is a well-known formula in Riemannian geometry obtained by integrating the Bochner--Lichnerowicz--Weitzenb\"{o}ck identity. In our previous work \cite{KolesnikovEMilmanReillyPart1}, we obtained a convenient weighted version of the Reilly formula, incorporating a general density. By applying the generalized Reilly formula with an appropriate log-convex (not log-concave!) density, we obtain a sufficient condition for establishing the local log-Brunn--Minkowski conjecture -- see Theorem \ref{thm:dual-BL}. Curiously, this condition resembles a dual log-convex form of the classical Brascamp--Lieb inequality \cite{BrascampLiebPLandLambda1}. Using the known estimates on the Poincar\'e constant of $B_q^n$, we are able to verify this condition when $q \in (2,\infty)$ and $n$ is large enough, yielding Theorem \ref{thm:intro-lq}. 

\medskip

To obtain our other results, we derive a different sufficient condition for establishing the local $L^p$-Brunn--Minkowski conjecture, in terms of a boundary Poincar\'e-type inequality for harmonic functions on $K$ -- see Theorem \ref{thm:sufficient}. In Section \ref{sec:Steklov}, we interpret this sufficient condition in terms of a first order pseudo-differential operator on $\partial K$, which we call the \emph{second Steklov operator}, due to its analogy with the classical Steklov (or Dirichlet-to-Neumann) operator. We are able to precisely calculate the spectrum of this operator for $B_2^n$, and to calculate the sharp constant in the boundary Poincar\'e-type inequality for $B_\infty^n$. After establishing the continuity of this boundary Poincar\'e-type constant $\BNH(K)$ with respect to the Hausdorff metric, we are able to deduce Theorem \ref{thm:intro-cube}. 

\medskip

To handle arbitrary origin-symmetric convex $K$, we obtain a general (rough) estimate on $\BNH(K)$ in terms of the in and out radii of $K$, as well as the (usual) Poincar\'e constant of $K$. Contrary to the linear invariance of the $L^p$-Brunn--Minkowski conjecture, our sufficient condition is no longer linearly invariant, and so we are required to apply it to a suitable linear image of $K$. By using the isotropic position, and the recent estimates of Y.-T.~Lee and S.~Vempala \cite{LeeVempala-KLS} on the Poincar\'e constant of $K$ vis-\`a-vis the Kannan--Lov\'asz--Simonovits conjecture \cite{KLS}, we are able to deduce Theorem \ref{thm:intro-main}.

\subsection{Applications} \label{subsec:applications}

We conclude this work by describing a couple of additional applications of our methods. It is by now a standard argument (see \cite{Lutwak-Firey-Sums,BLYZ-logBMInPlane}) to translate our local $L^p$-Brunn--Minkowski and log-Brunn--Minkowski inequalities into local uniqueness statements for the even $L^p$-Minkowski and log-Minkowski problems, respectively. The classical Minkowski problem (see \cite{Schneider-Book,LYZ-LpMinkowskiProblem} and the references therein) asks for necessary and sufficient conditions on a finite Borel measure $\mu$ on $S^{n-1}$, to guarantee the existence and uniqueness (up to translation) of a convex body $K \in \K$ so that its surface-area measure $dS_K$ coincides with $\mu$; it was solved by Minkowski and Aleksandrov and has led to a rich theory pertaining to the regularity of the associated Monge-Amp\`ere equation. In \cite{Lutwak-Firey-Sums}, Lutwak proposed to study the analogous $L^p$-Minkowski problem, where the role of the surface-area measure $dS_K$ is replaced by the $L^p$-surface-area measure:
\[
 dS_{K,p} := h_K^{1-p} dS_K .
 \]
 The case of \emph{even} measures when $p >1$ is well understood \cite{Lutwak-Firey-Sums,LYZ-LpMinkowskiProblem}, but the case $p<1$ poses a much greater challenge for both existence and uniqueness questions, with the log-Minkowski problem (corresponding to the case $p=0$) drawing the most attention (see Section \ref{sec:Mink} for more details). For the latter problem and for \emph{even} measures $\mu$, a novel necessary and sufficient \emph{subspace concentration condition} ensuring the \emph{existence} question was obtained by B\"{o}r\"{o}czky--Lutwak--Yang--Zhang in \cite{BLYZ-logMinkowskiProblem}, and the \emph{uniqueness} question was settled in \cite{BLYZ-logBMInPlane} in dimension $n=2$; it remains open in full generality in dimension $n \geq 3$. In Theorem \ref{thm:Mink-Uniq}, we translate our results into new local uniqueness statements for the \emph{even} $L^p$-Minkowski problem for appropriate values of $p < 1$; in particular, we establish local uniqueness for $p \in [1 - \frac{c}{n^{3/2}} , 1)$ in a $C^2$-neighborhood of any $K \in \K^2_{+,e}$.
 
\smallskip
 
Our second application pertains to stability estimates in the Brunn--Minkowski and anisotropic isoperimetric inequalities. Recall that the sharp anisotropic isoperimetric inequality is the statement that:
\begin{equation} \label{eq:intro-isop}
 P_{L}(K) \geq n V(L)^{\frac{1}{n}} V(K)^{\frac{n-1}{n}} 
\end{equation}
with equality when $K = L$, where $P_L(K) := \int_{S^{n-1}} h_L dS_K$ denotes the anisotropic perimeter of $K$ with respect to the convex body $L$ (note that when $L = B_2^n$ this boils down to the usual surface area of $K$). It is well known that the anisotropic isoperimetric inequality may be obtained by differentiating the Brunn--Minkowski inequality (and in fact, the two families of inequalities for all convex bodies $K,L$ are equivalent). Since the local $L^p$-Brunn--Minkowski inequality for $p<1$ is a strengthening of the classical one (in the class of origin-symmetric convex bodies), it is natural to expect that some stability results for isoperimetric inequalities could be derived from it, where an additional deficit term is added to the right-hand-side of (\ref{eq:intro-isop}), depending on some parameter measuring the proximity of $K$ to $L$.  
 
\smallskip
In Section \ref{sec:stability}, we obtain an interesting interpretation of the Hilbert--Brunn--Minkowski operator $L_K$ as the operator controlling the deficit in Minkowski's second inequality. As a consequence, we deduce new stability estimates for the (global) Brunn--Minkowski and anisotropic isoperimetric inequalities for origin-symmetric convex bodies, which depend on the variance of $h_L / h_K$ with respect to the cone-measure $dV_K$. In addition, we derive a strengthening of the best known stability estimates in the Brunn--Minkowski and anisotropic isoperimetric inequalities for the class of all (not necessarily origin-symmetric) convex bodies with respect to the natural asymmetry parameter $\inf_{x_0 \in \Real^n} V(K \triangle (x_0 + r L))$ with $V(r L) = V(K)$, which were previously due to Figalli, Maggi and Pratelli \cite{FMP-Inventiones,FMP-StableBM} and Segal \cite{Segal-StableBM}.

\medskip

The rest of this work is organized as follows. In Section \ref{sec:notation}, we introduce some convenient notation. In Section \ref{sec:local-global}, we establish a couple of standard equivalent versions of the global $L^p$-Brunn--Minkowski conjecture, introduce the local $L^p$-Brunn--Minkowski conjecture, and discuss the relation between the global and local versions. In Section \ref{sec:equivalent} we derive various equivalent infinitesimal formulations of the local conjecture: in terms of a second $L^p$-Minkowski inequality involving mixed-volumes, in terms of a Poincar\'e-type inequality on  $S^{n-1}$ and finally in terms of an equivalent version on $\partial K$. In Section \ref{sec:LK} we obtain an equivalent formulation involving the spectral-gap of the Hilbert--Brunn--Minkowski operator $L_K$, and establish its linear equivariance. In Section \ref{sec:Reilly} we obtain a sufficient condition in terms of a boundary Poincar\'e-type inequality, establishing Theorem \ref{thm:intro-main}. In Section \ref{sec:Steklov} we introduce the second Steklov operator, use it to rewrite our sufficient condition, and calculate its spectrum for $B_2^n$. In Section \ref{sec:unc} we obtain sharp bounds on our boundary Poincar\'e-type inequality for unconditional convex bodies and calculate it precisely for the cube. In Section \ref{sec:Reilly-logconvex} we obtain another sufficient condition for establishing the log-Brunn--Minkowski conjecture, in preparation for Theorem \ref{thm:intro-lq}. In Section \ref{sec:continuity}, we establish the continuity of our boundary Poincar\'e-type inequality, and use this to deduce Theorems \ref{thm:intro-lq}, \ref{thm:intro-cube} and \ref{thm:intro-ellq1}. The application to local uniqueness in the even $L^p$-Minkowski problem is presented in Section \ref{sec:Mink}. The application pertaining to new and improved stability estimates for Minkowski's second inequality, the anisotropic isoperimetric inequality, and the Brunn--Minkowski inequality is presented in Section \ref{sec:stability}.

\bigskip

\section{Notation}  \label{sec:notation}

Given $t \in \Real$ and $p > 0$, we denote:
\[
t^p := \sgn(t) \abs{t}^p . 
\]
If $f,g : X \rightarrow \Real$, we employ the following notation:
\[
f +_p g := \brac{f^p + g^p}^{\frac{1}{p}} . 
\]
More generally, if $t , \alpha, \beta \in \Real$, we write:
\begin{equation} \label{eq:alpha-beta}
t \cdot g := t^{\frac{1}{p}} g ~ \text{ and } ~ \alpha \cdot f +_p \beta \cdot g := \brac{\alpha f^p + \beta g^{p}}^{\frac{1}{p}} , 
\end{equation}
suppressing the dependence of $\cdot$ on $p$. When $p < 0$, we will specify $\frac{1}{p} f^p : X \rightarrow \Real$, even though $f$ itself may not be defined. Given $\frac{1}{p} f^p, \frac{1}{p} g^p : X \rightarrow \Real$, we use (\ref{eq:alpha-beta}) to define $\alpha \cdot f +_p \beta \cdot g$ only for $\alpha,\beta \in \Real$ so that $\alpha f^p + \beta g^{p}$ is positive. The limiting case when $p=0$ only makes sense when $f,g : X \rightarrow (0,\infty)$, and unless otherwise stated, is interpreted throughout this work as:
\[
\frac{1}{p} f^p := \log f ~,~ \alpha \cdot f +_p \beta \cdot g :=  f^{\alpha} g^{\beta} \;\;
\text{ for $p=0$.} 
\]
For instance, note that $\lim_{p \rightarrow 0} \frac{1}{p} (f^p - 1) = \log f$, but instead of using $\frac{1}{p} (f^p-1)$ we use $\frac{1}{p} f^p$ in all statements pertaining to concavity or regularity, as the missing $-\frac{1}{p}$ makes no difference. In a few places where ambiguity may arise, we will use the full $\frac{1}{p} (f^p - 1)$ expression.

\medskip

A convex body in $\Real^n$ is a convex, compact set with non-empty interior. We denote by $\K$ the collection of convex bodies in $\Real^n$ having the origin in their interior.  
The support function $h_K : \Real^n \rightarrow (0,\infty)$ of $K \in \K$ is defined as: 
\[
h_K(y) := \max_{x \in K} \scalar{y,x} \; ~,~ y \in \Real^n . 
\]
It is easy to see that $h_K$ is continuous and convex. Clearly, it is $1$-homogeneous, so we will mostly consider its restriction to the Euclidean unit-sphere $S^{n-1}$. Conversely, a  convex $1$-homogeneous function $h : \Real^n \rightarrow (0,\infty)$ is necessarily a support function of some $K \in \K$ (which is obtained as the polar body to $\set{ h \leq 1}$). Given $f \in C(S^{n-1})$, we will denote:
\[
 K +_p \eps \cdot f := A[h_K +_p \eps \cdot f] .
\]
We will only consider $\eps \in \Real$ so that $h_K +_p \eps \cdot f > 0$, ensuring that $K +_p \eps \cdot f  \in \K$. 

\medskip

As usual, we denote by $C^k(M)$, $k = 0,1,\ldots,m$, the space of $k$-times continuously differentiable functions on a $C^m$-smooth differentiable manifold $M$, equipped with its natural $C^k$-norm. When $k=0$, we simply write $C(M)$. It is known \cite[Theorem 1.8.11]{Schneider-Book} that convergence of elements of $\K$ in the Hausdorff metric is equivalent to convergence of the corresponding support functions in the $C(S^{n-1})$ norm; we will refer to this as $C$-convergence for brevity. We denote by $C^k_{>0}(S^{n-1})$ the convex cone of positive functions in $C^k(S^{n-1})$.  

\medskip

It will be convenient to introduce the following notation for a function $h \in C^2(S^{n-1})$. Given a local orthonormal frame $e_1,\ldots,e_{n-1}$ on $S^{n-1}$, 
we use $h_i$ and $h_{i,j}$ to denote $(\nabla_{S^{n-1}})_{e_i} h$ and $(\nabla_{S^{n-1}})^2_{e_i,e_j} h$, respectively, where $\nabla_{S^{n-1}}$ is the covariant derivative on the sphere $S^{n-1}$ with its canonical Riemannian metric $\delta$. Extending $h$ to a $1$-homogeneous function on $\Real^n$ and denoting by $\nabla_{\Real^n}$ the covariant derivative on Euclidean space $\Real^n$, we define the symmetric $2$-tensor $D^2 h$ on $S^{n-1}$ as the restriction of $\nabla^2_{\Real^n} h$ onto $T S^{n-1}$; in our local orthonormal frame, this reads as:
\[
(D^2 h)_{i,j} = (\nabla_{\Real^n})^2_{e_i,e_j} h = h_{i,j} + h \delta_{i,j} ~,~ i,j = 1,\ldots,n-1 .
\]
The latter identity follows since the second fundamental form of $S^{n-1} \subset \Real^n$ satisfies $\II_{S^{n-1}} = \delta$, and since in general:
\[
\nabla^2_{S^{n-1}} f (\theta) = \nabla^2_{\Real^n} f(\theta)|_{T S^{n-1}} -  f_\theta \II_{S^{n-1}}(\theta),
\]
and $f_\theta = \scalar{\nabla_{\Real^n} f, \theta} = f$ for any $1$-homogeneous function $f$. Note that $h \in C^2_{>0}(S^{n-1})$ is a support-function of $K \in \K$ if and only if $D^2 h_K \geq 0$ as a ($n-1$ by $n-1$) positive semi-definite tensor. 

\medskip

We denote by $\K^m_+$ the subset of $\K$ of convex bodies with $C^m$ boundary and strictly positive curvature. By \cite[pp. 106-107, 111]{Schneider-Book}, for $m \geq 2$, $K \in \K^m_+$ if and only if $h_K \in C^m(S^{n-1})$ and $D^2 h_K > 0$ (as a $n-1$ by $n-1$ positive-definite tensor). 
It is well-known that $\K^2_+$ is dense in $\K$ in the $C$-topology \cite[p. 160]{Schneider-Book}.
We also denote by $C^2_h(S^{n-1}) := \set{ h_K \; ; \; K \in \K^2_+} = \set{ h \in C^2_{>0}(S^{n-1}) \; ; \; D^2 h > 0}$, the convex cone of strictly convex $C^2$ support functions. It is immediate to check that $C^2_h(S^{n-1})$ is open in $C^2(S^{n-1})$ (see e.g. \cite[pp. 38, 111]{Schneider-Book},\cite{ColesantiPoincareInequality}, or simply use that the condition $D^2 h_K > 0$ is open in $C^2(S^{n-1})$). This induces the natural $C^2$-topology on $\K^2_+$, where $K_i \rightarrow K$ in $C^2$ if and only if $h_{K_i} \rightarrow h_K$ in $C^2_h(S^{n-1})$.

\medskip

We will always use $S_e$ to denote the origin-symmetric (or even) members of a set $S$, e.g. $\K_e$ and $\K^2_{+,e}$ denote subset of origin-symmetric bodies in $\K$ and $\K^2_+$, respectively, and $C^2_{h,e}(S^{n-1})$ and $C^2_e(S^{n-1})$ denote the subset of even functions in $C^2(S^{n-1})$ and $C^2_h(S^{n-1})$, respectively. 

\medskip

Finally, we abbreviate throughout this work the phrase ``$L^p$-Brunn--Minkowski" by $p$-BM, and ``log-Brunn--Minkowski" by log-BM.

\bigskip

\section{Global vs. Local Formulations of the $L^p$-Brunn--Minkowski Conjecture}  \label{sec:local-global}

\subsection{Standard Equivalent Global Formulations}

\begin{lem}
The following are equivalent for a given dimension $n\geq 2$ and $p < 1$ (with the usual interpretation when $p=0$):
\begin{enumerate}
\item For all $K,L \in \K_e$, the global $p$-BM conjecture is valid:
\[
 V((1-\lambda) \cdot K +_p \lambda \cdot L) \geq \brac{(1-\lambda) V(K)^{\frac{p}{n}} + \lambda V(L)^{\frac{p}{n}}}^{\frac{n}{p}} \;\;\; \forall \lambda \in [0,1]  .
\]
\item  For all $K,L \in \K_e$, the following function is concave: \[
[0,1] \ni s \mapsto \frac{1}{p} V((1-s) \cdot K+_p s \cdot L)^{\frac{p}{n}} .
\]
\item For all $K,L \in \K_e$, the following function is concave: \[
\Real_+ \ni t \mapsto \frac{1}{p} V(K+_p t \cdot L)^{\frac{p}{n}} .
\]
\end{enumerate}
\end{lem}

The derivation is completely standard when $p \geq 1$, since in that case the support and Wullf functions of $K  +_{p} t \cdot L$ coincide, and hence:
\[
(1-\lambda) \cdot (K +_p t_1 \cdot L) +_p \lambda \cdot  (K+_p t_2 \cdot L) = K +_p ((1-\lambda) t_1 + \lambda t_2) \cdot L .
\]
When $p < 1$, there is in general only a one-sided containment, which fortunately goes in the right direction for us.

\begin{proof}
Clearly (2) implies (1) by checking concavity at the three points $s=0,\lambda,1$. 
To show that (1) implies (2), let $s_0,s_1 \in [0,1]$, set $s_\lambda = (1-\lambda) s_0 + \lambda s_1$. When $p < 1$, we only have:
\[
(1-\lambda) \cdot ((1-s_0) \cdot K +_p s_0 \cdot L) +_p \lambda \cdot  ((1-s_1) \cdot K+_p s_1 \cdot L) \subset (1-s_\lambda) \cdot K +_p s_\lambda\cdot L  .
\]
Indeed, the support function of $(1-s_i) \cdot K +_p s_i \cdot L$ on the left is not larger than the Wulff function $((1-s_i) h_K^p + s_i h_L^p)^{1/p}$, and hence we have the above inequality for the corresponding Wulff functions, and hence for the associated Aleksandrov bodies. Applying $\frac{1}{p} V(\cdot)^{p/n}$ to both sides and invoking (1), the desired concavity (2) is established. An identical argument also shows that (1) implies (3). 

To show that (3) implies (1), apply (3) to the bodies $K = (1-\lambda) \cdot \bar K$ and $L = \frac{\lambda}{\beta} \cdot \bar L$ at the points $t = 0,\beta,1$:
\[
V((1-\lambda) \cdot \bar K +_p \lambda \cdot \bar L)\geq \brac{(1-\beta) V((1-\lambda) \cdot \bar K)^{\frac{p}{n}} + \beta V((1-\lambda) \cdot \bar K +_p \frac{\lambda}{\beta} \cdot \bar L)^{\frac{p}{n}}}^{\frac{n}{p}} .
\]
Then letting $\beta \rightarrow 0$, it is easy to see by monotonicity and homogeneity that $\beta V((1-\lambda) \cdot \bar K +_p \frac{\lambda}{\beta} \cdot \bar L)^{p/n} \rightarrow \lambda V(\bar L)^{p/n}$, 
and we obtain in the limit:
\[
V((1-\lambda) \cdot \bar K +_p \lambda \cdot \bar L) \geq  \brac{(1-\lambda) V(\bar K)^{\frac{p}{n}} + \lambda V(\bar L)^{\frac{p}{n}}}^{\frac{n}{p}} .
\]
\end{proof}

\begin{rem} \label{rem:secant}
It is worth mentioning an alternative argument for showing that (3) implies (1): the concavity implies that the derivative at $t=0$ is larger than the secant slope as $t \rightarrow \infty$, i.e.:
\[
\frac{1}{n} \left . V(K)^{\frac{p}{n}-1} \frac{d}{dt} \right |_{t=0+} V(K +_p t \cdot L) \geq \lim_{t\rightarrow \infty} \frac{\frac{1}{p} V(K +_p t \cdot L)^{\frac{p}{n}} - \frac{1}{p}V(K)^{\frac{p}{n}}}{t} = \frac{1}{p} V(L)^{\frac{p}{n}} .
\]
Equivalently, using the $L^p$-mixed-volume $V_p(K,L)$ introduced by Lutwak \cite{Lutwak-Firey-Sums}:
\begin{equation} \label{eq:Vp}
V_p(K,L) := \frac{p}{n} \lim_{t \rightarrow 0+} \frac{V(K +_p t \cdot L) - V(K)}{t}  ,
\end{equation}
we have:
\begin{equation} \label{eq:1stMink}
\frac{1}{p} \brac{V_p(K,L) - V(K)} \geq  V(K) \frac{1}{p} \brac{\brac{\frac{V(L)}{V(K)}}^{\frac{p}{n}} -1 },
\end{equation}
with the case $p=0$ interpreted in the limiting sense. 
The latter is precisely the first $p$-Minkowski inequality, which has been shown by B\"{o}r\"{o}czky--Lutwak--Yang--Zhang in \cite{BLYZ-logBMInPlane} when $p \geq 0$ to be equivalent to the global $p$-BM inequality (1). Their proof extends to $p < 0$ with appropriate modifications. 
\end{rem}

\subsection{Global vs. Local $L^p$-Brunn--Minkowski}

We now introduce the following local version of the global $p$-BM conjecture.

\begin{conj}[Local $L^p$-Brunn--Minkowski conjecture]
Let $n \geq 2$ and $p \in [0,1)$. For all $K \in \K^2_{+,e}$:
\begin{equation} \label{eq:local-p-BM}
\forall \; \frac{1}{p} f^p \in C^{2}_{e}(S^{n-1}) \;\;\; \left . \frac{d^2}{(d\eps)^2} \right |_{\eps=0} \frac{1}{p} V(K +_p \eps \cdot f)^{\frac{p}{n}} \leq 0 .
\end{equation}
Recall that when $p=0$, this is interpreted as:
\begin{equation} \label{eq:local-log-BM}
\forall \text{ positive } f \in C^{2}_{e}(S^{n-1}) \;\;\; \left . \frac{d^2}{(d\eps)^2} \right |_{\eps=0} \log V(K +_0 \eps \cdot f) \leq 0 . 
\end{equation}
\end{conj}
\noindent
Whenever referring to (\ref{eq:local-p-BM}) with $p=0$, we will always interpret this as (\ref{eq:local-log-BM}). It follows from the results of this work (see Theorem \ref{thm:lambda1-cube}) that for a fixed $p < 0$, (\ref{eq:local-p-BM}) cannot hold for all $K \in \K^2_{+,e}$, but nevertheless for a particular $K$, (\ref{eq:local-p-BM}) may hold with $p = p_K < 0$. 

\medskip

The following is standard:
\begin{lem}[Global $p$-BM implies Local $p$-BM for given $K \in \K^2_{+,e}$] \label{lem:global2local}
Fix $p \in \Real$ and $K \in \K^2_{+,e}$. Assume that the global $p$-BM conjecture holds for $K$, namely:
\[
[0,1] \ni \lambda \mapsto \frac{1}{p} V((1-\lambda) \cdot K+_p \lambda \cdot L)^{\frac{p}{n}} \text{ is concave }  \;\; \forall L \in N_K ,
\]
for some $C^2$-neighborhood $N_K$ of $K$ in $\K^2_{+,e}$. Then the local $p$-BM conjecture (\ref{eq:local-p-BM}) holds for $K$. 
\end{lem}
\begin{proof}
We know that $h_K \in C^2_{h,e}(S^{n-1})$ since $K \in \K^2_{+,e}$. 
Let $\frac{1}{p} f^p \in C^2_{e}(S^{n-1})$, and consider the Wulff function $w_\eps := h_K +_p \eps \cdot f = (h_{K}^p + \eps f^p)^{1/p}$, which is in $C^2_e(S^{n-1})$ for small enough $\abs{\eps}$ since $h_K$ is strictly positive. As $w_\eps$ is a $C^2$ perturbation of $h_K$ and since $C^2_{h,e}(S^{n-1})$ is open in $C^2_e(S^{n-1})$ and locally convex, it follows
that  for small enough $\eps_0 > 0$,  $w_\eps$ is the support function of a convex body $K_\eps = K+_p  \eps \cdot f \in N_K$, for all $\eps \in [0,\eps_0]$.

Since $w_{\lambda \eps_0} = (1-\lambda) \cdot h_K +_p \lambda \cdot w_{\eps_0}$ for all $\lambda \in [0,1]$, it follows that $K_{\lambda \eps_0} = (1-\lambda) \cdot K +_p \lambda \cdot K_{\eps_0}$. Our global assumption (with $L = K_{\eps_0}$) therefore implies that $[0,1] \ni \lambda \mapsto \frac{1}{p} V(K_{\lambda \eps_0})^{p/n}$ is concave, and in particular, its second derivative at $\lambda=0$ is non-positive, yielding (\ref{eq:local-p-BM}). Note that the function $[-\eps_0,\eps_0] \ni \eps \mapsto V(K +_p \eps \cdot f)$ is indeed in $C^2$, as witnessed by writing an explicit differential formula for the volume in terms of the support function (as we shall do in the next section). 
\end{proof}

We also have the following \emph{conditional} converse:

\begin{lem}[Local $p$-BM implies Global $p$-BM assuming geodesic in $\K^2_{+,e}$] \label{lem:loc-to-global}
Fix $p \in \Real$ and $K_0,K_1 \in \K^2_{+,e}$. Assume that:
\begin{equation} \label{eq:smooth-geodesic}
\forall t \in [0,1] \;\;\; K_t := (1-t) \cdot K_0 +_p t \cdot K_1 \text{ is in } \K^2_{+,e} 
\end{equation}
If the local $p$-BM conjecture (\ref{eq:local-p-BM}) holds for $K_t$ for all $t \in [0,1]$, then the global $p$-BM conjecture holds between $K_{t_0}$ and $K_{t_1}$ for all $t_0,t_1 \in [0,1]$:
\begin{equation} \label{eq:global-concavity}
V((1-\lambda) \cdot K_{t_0} + \lambda \cdot K_{t_1}) \geq \brac{(1-\lambda) V(K_{t_0})^{\frac{p}{n}} + \lambda V(K_{t_1})^{\frac{p}{n}}}^{\frac{n}{p}} \;\;\; \forall \lambda \in [0,1] .
\end{equation}
\end{lem}

This is not as obvious as it may seem. The proof crucially relies on a classical observation of Aleksandrov \cite[Lemma 6.5.1]{Schneider-Book}:
\begin{lem}[Aleksandrov]
If $K = A[w]$ for some $w \in C_{>0}(S^{n-1})$, then for any $\theta \in S^{n-1}$ so that $h_K(\theta) < w(\theta)$, $\partial K$ has at least two normals at any contact point $x_0 \in \partial K \cap \set{x \in \Real^n \; ; \;\scalar{x,\theta} = h_K(\theta)}$; in particular, $\partial K$ is not $C^1$ smooth at $x_0$. 
\end{lem}
The usual application of this lemma is to deduce that the set of contact points corresponding to bad directions $\theta$ as above has zero $(n-1)$-dimensional Hausdorff measure $\H^{n-1}$, since $\partial K$ is (twice) differentiable $\H^{n-1}$ almost-everywhere. We will only require the following:

\begin{cor} \label{cor:Aleksandrov}
Let $K_0,K_1 \in \K^2_+$, $p \in \Real$ and $t \in [0,1]$. Denote $w_t := (1-t) \cdot h_{K_0} +_p t \cdot h_{K_1} \in C^2(S^{n-1})$, and set $K_t := A[w_t]$. Then:
\[
K_t \in \K^2_+  \;\; \Leftrightarrow \;\; w_t = h_{K_t} \;\; \Leftrightarrow \;\; \text{$w_t$ is a support function} .
\]
\end{cor}
\begin{proof}
If $w_t$ is a support function then it must coincide with $h_{K_t}$, and therefore $h_{K_t} \in C^2(S^{n-1})$, i.e. $K_t \in \K^2_+$. Conversely, if $K_t \in \K^2_+$, Aleksandrov's lemma implies that $w_t = h_{K_t}$ and so $w_t$ is a support function. 
\end{proof}

\begin{proof}[Proof of Lemma \ref{lem:loc-to-global}]
By the previous corollary, our assumption (\ref{eq:smooth-geodesic}) is equivalent to:
\begin{equation} \label{eq:interpolation}
h_{K_t} = (1-t) \cdot h_{K_0} +_p t \cdot h_{K_1} \;\;\; \forall t \in [0,1].
\end{equation}
The global concavity of $[0,1] \ni \lambda \mapsto \frac{1}{p} V((1-\lambda) \cdot K_{0} + \lambda \cdot K_{1})^{\frac{p}{n}}$ now follows by testing its second derivative at a given $\lambda \in [0,1]$, which must be non-positive by the local $p$-BM assumption (\ref{eq:local-p-BM}) for the body $K = K_{\lambda}$ and the function $\frac{1}{p} f^p := \frac{1}{p} h_{K_{1}}^p - \frac{1}{p} h_{K_{0}}^p \in C^2_{e}(S^{n-1})$. The latter concavity is equivalent to the desired (\ref{eq:global-concavity}). 
\end{proof}

Note that when $p \geq 1$, (\ref{eq:interpolation}) holds automatically since the Wulff function $(1-t) \cdot h_{K_0} +_p t \cdot h_{K_1}$ is a support function (being an $L^p$-combination of two support functions). Consequently, assumption (\ref{eq:smooth-geodesic}) on the smoothness of the entire geodesic $t \mapsto K_t$ is satisfied when $p \geq 1$, 
and employing a standard approximation argument for general (non-smooth) end points $K_0,K_1$, it is immediate to deduce the global $p$-BM formulation (\ref{eq:p-BM}) from the local one (\ref{eq:local-p-BM}). 

However, this is definitely not the case in general when $p \in [0,1)$. First, the semi-group property (\ref{eq:interpolation}) will not hold in general when $p\in [0,1)$, and one can only ensure that the left-hand-side is a subset of the right-hand one. This time, the inclusion goes in the unfavorable direction for us: 
\[
K_{\lambda,\eps} := K_{\lambda} +_p \eps \cdot f \subset K_{\lambda + \eps} ,
\]
and so knowing that $\eps \mapsto \frac{1}{p} V(K_{\lambda,\eps})^{p/n}$ is concave at $\eps=0$ for every $\lambda \in [0,1]$ will not tell us much about the (local) concavity of $[0,1] \ni \lambda \mapsto \frac{1}{p} V(K_{\lambda})^{p/n}$. While it is possible to bypass this point, the main obstacle for deducing the global $p$-BM conjecture  (\ref{eq:p-BM}) from the local one (\ref{eq:local-p-BM}) when $p \in [0,1)$ is the violation of the smoothness assumption (\ref{eq:smooth-geodesic}). By Aleksandrov's lemma, $K_t$ will necessarily have a singular boundary whenever the Wulff function $(1-t) \cdot h_{K_0} +_p t \cdot h_{K_1}$ is no longer a support function. For such $t$'s, we will need to approximate $K_t$ by bodies $K_t^i \in \K^2_{+,e}$, and establish a relation between:
\[
 \lim_{i \rightarrow \infty} \left . \frac{d^2}{(d\eps)^2} \right |_{\eps=0} V(K^i_t +_p \eps \cdot f)  \; \text{  and } \left . \frac{d^2}{(d\eps)^2} \right |_{\eps=0} V(K_t +_p \eps \cdot f) .
\]
This turns out to be an extremely difficult and tantalizing question, which boils down to the study of the second variation of the volume of the Aleksandrov body for non-smooth $K \in \K$:
\[
 \left . \frac{d^2}{(d\eps)^2} \right |_{\eps=0} V(A[h_K + \eps f]) ;
\]
here $f$ may be assumed to be the difference of two support functions. While the first variation of volume was studied by Aleksandrov himself and is well understood (see e.g. \cite[Lemma 6.5.3]{Schneider-Book}), to the best of our knowledge, the second variation is \emph{terra incognita}. 

\medskip

Being unable to establish the equivalence between the global and local formulations, we state this as:
\begin{conj} \label{conj:local-global}
Given $p \in [0,1)$, the validity of the local $p$-BM conjecture (\ref{eq:local-p-BM}) for all $K \in \K^2_{+,e}$ is logically equivalent to the validity of the global $p$-BM conjecture (\ref{eq:p-BM}) for all $K_0,K_1 \in \K_e$. 
\end{conj}

For future reference, we record the following: 
\begin{prop} \label{prop:local-equiv}
Let $p \in \Real$ and $K \in \K^2_{+,e}$. Then the following statements are equivalent:
\begin{enumerate}
\item There exists a $C^2$-neighborhood $N_K$ of $K$ in $\K^2_{+,e}$ so that the local $p$-BM conjecture (\ref{eq:local-p-BM}) holds for all $K' \in N_K$. 
\item There exists a $C^2$-neighborhood $N'_K$ of $K$ in $\K^2_{+,e}$ so that for all $K_0,K_1 \in N'_K$ and $t \in [0,1]$, $K_t := (1-t) \cdot K_0 +_p t \cdot K_1 \in N'_K$, and the global $p$-BM conjecture (\ref{eq:p-BM}) holds between $K_{t_0},K_{t_1}$ for all $t_0,t_1 \in [0,1]$.
\end{enumerate}
Furthermore, given $p_0 \in \Real$, the following statements are equivalent:
\begin{enumerate}
\item[(1')] For every $p > p_0$, (1) or (2) above hold. 
\item[(2')] The local $p_0$-BM conjecture (\ref{eq:local-p-BM}) holds for $K$. 
 \end{enumerate}
\end{prop}
The equivalence between (1') and (2') follows from the results of the next sections, but at the risk of forward-referencing, we include this result in the present section as it fits more naturally here; the reader may wish to skip its proof in the first reading. 
\begin{proof}
Lemma \ref{lem:global2local} verifies that (2) implies (1) with $N_K = N'_K$. To show that (1) implies (2), assume for simplicity that $p \neq 0$; the case $p=0$ follows analogously. Denoting by $N_{h_K} = \set{h_{K'} \;  ; \; K' \in N_K}$ the corresponding neighborhood of $h_K$ in $C^2_{h,e}(S^{n-1})$, set $N_{h_K}^p := \set{ h^p \; ; \; h \in N_{h_K}}$, which is an open subset of $C^2_e(S^{n-1})$. As $C^2_e(S^{n-1})$ is locally convex, we may find a convex neighborhood $N'^{p}_{h_K}$ of $h_K^p$ in $N_{h_K}^p$. Setting $N'_{h_K} := \{h^{1/p} \; ; \; h \in N'^{p}_{h_K}\}$, the latter is an open subset of $N_{h_K} \subset C^2_{h,e}(S^{n-1})$ containing $h_K$ which is convex with respect to the $+_p$ operation:  if $h_{K_0},h_{K_1} \in N'_{h_K}$ then $w_\lambda := (1-\lambda) \cdot h_{K_0} +_p \lambda \cdot h_{K_1} \in N'_{h_K} \subset C^2_{h,e}(S^{n-1})$ for all $\lambda \in [0,1]$. By Corollary \ref{cor:Aleksandrov}, since $w_\lambda$ is a support function, it must be that of $K_\lambda := (1-\lambda) \cdot K_0 +_p \lambda \cdot K_1$, and therefore $K_\lambda \in \K^2_{+,e}$. Consequently, defining $N'_K = \{ K' \in \K^2_{+,e} \; ; \; h_{K'} \in N'_{h_K} \}$, the latter is an open subset of $N_K \subset \K^2_{+,e}$ containing $K$  which is convex with respect to the $L^p$-Minkowski operation. The smoothness assumption (\ref{eq:smooth-geodesic}) is therefore satisfied between any $K_0,K_1 \in N'_K$, and so the assertion follows from Lemma \ref{lem:loc-to-global}.

To show that (1') implies (2'), apply the local $p$-BM conjecture (\ref{eq:local-p-BM}) to $K$ in any of the equivalent forms given in the next sections (e.g. Propositions \ref{prop:2nd-p-Minkowski}, \ref{prop:p-BM-Sphere}, \ref{prop:p-BM-K} or \ref{prop:LK-SG}), and take the limit as $p \rightarrow p_0$. The implication that (2') implies (1') follows from the equivalent spectral characterization of the local $p$-BM conjecture given in Corollary \ref{cor:LK-SG}, and the continuity of the spectrum of the Hilbert--Brunn--Minkowski operator $-L_K$ under $C^2$ perturbations asserted in Theorem \ref{thm:LK} (4).
\end{proof}

\bigskip

\section{Local $L^p$-Brunn--Minkowski Conjecture -- Infinitesimal Formulation} \label{sec:equivalent}

In this section we fix $K \in \K^2_{+,e}$, and derive equivalent infinitesimal formulations to the local $p$-Brunn--Minkowski conjecture for $K$.

\subsection{Mixed Surface Area and Volume of $C^2$ functions}

It was shown by Minkowski (e.g. \cite{Schneider-Book,BonnesenFenchelBook}) that when $\set{K_i}_{i=1}^m$ are convex bodies in $\Real^n$, then the volume of their Minkowski sum is a polynomial in the scaling coefficients:
\[
V(\sum_{i=1}^m t_i K_i) = \sum_{1 \leq i_1,\ldots,i_n \leq m} t_{i_1} \cdot \ldots \cdot t_{i_n} V(K_{i_1},\ldots,K_{i_n}) \;\;\; \forall t_i \geq 0 .
\]
The coefficient $V(K_{i_1},\ldots,K_{i_n})$ is called the mixed volume of the $n$-tuple $(K_{i_1},\ldots,K_{i_n})$; it is uniquely defined by requiring in addition that it be invariant under permutation of its arguments. In this subsection, we extend the definition of mixed volume $V(h_1,\ldots,h_n)$ to a $n$-tuple of functions in $C^2(S^{n-1})$, in a manner ensuring that:
\[
V(h_{K_1},\ldots,h_{K_n}) = V(K_1,\ldots,K_n) \;\;\; \forall \set{K_i}_{i=1}^n \subset \K^2_+ .
\]

\medskip

Recall our notation given a local orthonormal frame $e_1,\ldots,e_{n-1}$ on $S^{n-1}$:
\[
(D^2 h)_{i,j} = (\nabla_{\Real^n})^2_{e_i,e_j} h = h_{i,j} + h \delta_{i,j} ~,~ i,j = 1,\ldots,n-1 .
\]
Let $D_m(A^1,\ldots,A^m)$ denote the mixed discriminant (or mixed determinant) of an $m$-tuple $(A^1,\ldots,A^m)$ with $A^i \in \M_m$, the set of $m$ by $m$ matrices (over $\Real$), namely:
\[
D_m(A^1,\ldots,A^m) := \frac{1}{m!} \sum_{\sigma, \tau \in S_m} (-1)^{\sgn(\sigma) + \sgn(\tau)} A^1_{\sigma(1),\tau(1)} \cdot \ldots \cdot A^m_{\sigma(m),\tau(m)} ,
\]
where $S_m$ denotes the permutation group on $\set{1,\ldots,m}$.
Recall that the mixed discriminant is simply the multi-linear polarization of the usual determinant functional $\det$ on $\M_m$, and so in particular is invariant under permutation of its arguments and satisfies $D_m(A,\ldots,A) = \det A$. 

\begin{defn*} Given a tuple $(h_1,\ldots,h_{n-1})$ of functions in $C^2(S^{n-1})$, define their ``mixed surface area function" $\SS(h_1,\ldots,h_{n-1}) \in C(S^{n-1})$ by:
\begin{align*}
\SS(h_1,\ldots,h_{n-1})(\theta) & := D_{n-1}(D^2 h_1(\theta) , \ldots , D^2 h_{n-1}(\theta)) \\
&= D_{n}(\nabla^2_{\Real^n} h_1(\theta),\ldots,\nabla^2_{\Real^n} h_{n-1}(\theta),\theta \otimes \theta) \\
& =D_{n}(\nabla^2_{\Real^n} h_1(\theta),\ldots,\nabla^2_{\Real^n} h_{n-1}(\theta),\text{Id})  
\end{align*}
\end{defn*}
It is easy to see that the former expression does not depend on the choice of local orthonormal frame. 
The latter two equalities follow easily by expanding $D_n$ according to the last entry, since for any $1$ homogeneous function $h$, $\theta$ is an eigenvector of the symmetric $\nabla^2_{\Real^n} h$ with eigenvalue zero, and hence $\nabla^2_{\Real^n} h = P_{\theta^{\perp}} \nabla^2_{\Real^n} h P_{\theta^{\perp}}$ where $P_{\theta^{\perp}}$ denotes orthogonal projection perpendicular to $\theta$. 

\medskip
The surface-area measure of $K$, denoted $dS_K$, is the Borel measure on $S^{n-1}$ obtained by pushing forward the $(n-1)$-dimensional Hausdorff measure $\H^{n-1}$ on $\partial K$ via the Gauss map $\nu_{\partial K} : \partial K \rightarrow S^{n-1}$ \cite[p. 115,207]{Schneider-Book}. Here $\nu_{\partial K}(x)$ is the unit outer-normal to $\partial K$ at $x$, which by convexity is well-defined and unique for $\H^{n-1}$-a.e. $x \in \partial K$. 

Note that when $K \in \K^2_+$, $\nu_{\partial K} : \partial K \rightarrow S^{n-1}$ is in fact a $C^1$ diffeomorphism. 
Identifying between the tangent spaces $T_x \partial K$ and $T_{\nu_{\partial K}(x)} S^{n-1}$, we have \cite[p. 107]{Schneider-Book}:
\[
d \nu_{\partial K}(x) = \II_{\partial K}(x) \;\;\; \forall x \in \partial K,
\]
where $\II_{\partial K}(x)$ denotes the second fundamental form of $\partial K$ at $x$. The inverse of the Gauss map is the Weingarten map $\nabla_{\Real^n} h_K : S^{n-1} \rightarrow \partial K$, and therefore \cite[p. 108]{Schneider-Book}:
\[
D^2 h_K(\nu_{\partial K}(x)) = \II_{\partial K}^{-1}(x) \;\;\; \forall x \in \partial K .
\]
If follows by the change-of-variables formula that:
\begin{equation} \label{eq:dSK}
dS_K(\theta) = \text{det}(D^2 h_K)(\theta) d\theta  = \SS(h_K,\ldots,h_K)(\theta) d\theta ;
\end{equation}
this explains the name ``mixed surface-area function" for $\SS(h_1,\ldots,h_n)$.

\medskip

Recall that when $K_1,\ldots,K_{n} \in \K^2_+$ then $h_{K_1},\ldots,h_{K_n} \in C^2(S^{n-1})$, and so their mixed volume may be expressed as \cite[p. 64]{BonnesenFenchelBook}, \cite[p. 115,275]{Schneider-Book}:
\begin{equation} \label{eq:mixed-vol-bodies}
V(K_1,\ldots,K_n) = \frac{1}{n} \int_{S^{n-1}} h_{K_n} \SS(h_{K_1},\ldots,h_{K_{n-1}}) d\theta ;
\end{equation}
this is just a multi-linear polarization of the usual formula:
\[
V(K,\ldots,K) = V(K) = \frac{1}{n} \int_{S^{n-1}} h_K dS_K = \frac{1}{n} \int_{S^{n-1}} h_K \SS(h_K,\ldots,h_K)(\theta) d\theta .
\]
The fact that the expression on the right-hand-side of (\ref{eq:mixed-vol-bodies}) is invariant under permutation of $K_1,\ldots,K_n$ is a nice exercise in integration by parts, which we shall reproduce below. Consequently, it is natural to give the following:
\begin{defn*} Given $h_1,\ldots,h_n \in C^2(S^{n-1})$, define their ``mixed-volume" as:
\[
V(h_1,\ldots,h_n) :=  \frac{1}{n} \int_{S^{n-1}} h_{n} \SS(h_{1},\ldots,h_{n-1}) d\theta .
\]
\end{defn*}
\noindent Note that the mixed-volume is indeed multi-linear in its arguments.

\subsection{Properties of Mixed Surface Area and Volume}

Next, if $A^1,\ldots,A^{n-1} \in \M_{n-1}$, a direct computation verifies:
 \begin{align}
\label{eq:DviaQ} & D_{n-1}(A^1,\ldots,A^{n-1}) =  \sum_{i,j} A^1_{i,j} Q^{i,j}(A^2,\ldots,A^{n-1}) ~,~ \\
\nonumber &  Q^{i,j}(A^2,\ldots,A^{n-1}) := \frac{(-1)^{i+j}}{n-1} D_{n-2}(M^{i,j}(A^2),\dots,M^{i,j}(A^{n-1})) ,
 \end{align}
where $M^{i,j}(A)$ is the minor resulting after removing the $i$-th row and the $j$-th column from $A \in \M_{n-1}$. Consequently, when $A^2 = \ldots = A^{n-1} = A \in GL_{n-1}$, we see that:
\begin{align}
\nonumber Q^{i,j}(A) & := Q^{i,j}(A,\ldots,A) = \frac{(-1)^{i+j}}{n-1} \text{det}(M^{i,j}(A)) \\
\label{eq:Q-formula}  & = \frac{1}{n-1} \text{adj}(A)^{i,j} = \frac{1}{n-1} \text{det}(A) (A^{-1})^{i,j} . 
\end{align}
 Clearly $Q^{i,j} = Q^{j,i}$ is symmetric and multi-linear in its arguments. Furthermore, the first of the following properties \cite[Lemma 2-12]{Andrews-EntropyOfFlows} will be constantly used (the second is mentioned for completeness):
\begin{itemize}
 \item For any $h_1,\ldots,h_{n-1} \in C^3(S^{n-1})$ and $i=1,\ldots,n-1$, the (local) $C^1$ vector field on $S^{n-1}$
\[
\sum_{j=1}^{n-1} Q^{i,j}(D^2 h_{1},\ldots,D^2 h_{n-1}) e_j 
\]
is divergence free. 
\item For any $K_1,\ldots,K_{n-1} \in \K^2_+$, $\sum_{i,j} Q^{i,j}(D^2 h_{K_1},\ldots,D^2 h_{K_{n-1}}) e_i \otimes e_j$ is a positive definite (local) $2$-tensor on $S^{n-1}$.
\end{itemize}
 
It is easy to see that the above two properties do not depend on the choice of local orthonormal frame. 
We will henceforth freely employ Einstein summation convention, summing over repeated indices. Our choice of using local orthonormal frames instead of local coordinates is in order to simplify notation, dispensing with the need to keep track of the covariance / contravariance of our various tensors. In local coordinates, we would apply the mixed discriminant $D_m$ to 1-covariant 1-contravariant tensors $A^i_j$, 
and as suggested by the present notation, $Q^{i,j}$ would be a 2-contravariant tensor. 
 
\begin{lem} \label{lem:mixed-vol-symmetric}
For any $h_1,\ldots,h_n \in C^2(S^{n-1})$, the mixed volume $V(h_1,\ldots,h_n)$ is invariant under permutation of its arguments. 
\end{lem}
\begin{proof}
By approximation, we may assume $h_3,\ldots,h_n \in C^3(S^{n-1})$; it is then enough to show that:
\[
V(f,g,h_3\ldots,h_n) = V(g,f,h_3\ldots,h_n) \;\;\; \forall f,g \in C^2(S^{n-1}),
\]
since the last $n-1$ arguments are invariant under permutation by definition of $V$ and $D$. 
Abbreviating $Q^{i,j} = Q^{i,j}(D^2 h_3,\ldots,D^2 h_n)$, integrating by parts, and using the divergence free property of $Q^{i,\cdot}$, we have:
\begin{align*}
& V(f,g,h_3\ldots,h_n) =  \frac{1}{n} \int_{S^{n-1}} f \SS(g,h_3,\ldots,h_n) d\theta = \frac{1}{n} \int_{S^{n-1}} f (g_{i,j} + g \delta_{i,j}) Q^{i,j} d\theta \\
&=  \frac{1}{n} \int_{S^{n-1}} \brac{Q^{i}_{i} f g  - (f Q^{i,j})_j g_i }  d\theta = \frac{1}{n} \int_{S^{n-1}}  \brac{Q^{i}_{i} f g - Q^{i,j} f_j g_i } d\theta .
\end{align*}
The latter expression is symmetric in $f,g$ (by symmetry of $Q^{i,j}$), and so the assertion is established. 
\end{proof}

\medskip

\subsection{Second $L^p$-Minkowski Inequality} \label{subsec:2nd-p-Minkowski}

We now fix $K \in \K^2_+$ and $p < 1$. Given $\frac{1}{p} f^p \in C^2(S^{n-1})$, the strict convexity of $K$ and the positivity of $h_K$ imply that for small enough $\abs{\eps}$, the Wulff function $h_K +_p \eps \cdot f$ is a $C^2$-smooth support function, and consequently:
\[
h_{K +_p \eps \cdot f} = h_{A[h_K +_p \eps \cdot f]} = h_K +_p \eps \cdot f  \;\;\; \forall \abs{\eps} \ll 1 . 
\]
Denoting:
\begin{equation} \label{eq:z}
z := \begin{cases} \frac{1}{h_K^p} \frac{f^p}{p}  & p \neq 0 \\ \log f & p=0 \end{cases} \in C^2(S^{n-1}) \; , 
\end{equation}
a second-order Taylor expansion immediately yields:
\begin{equation} \label{eq:2ndorder}
h_K +_p \eps \cdot f = h_K + \eps z h_K + \frac{\eps^2}{2} (1-p) z^2 h_K + R(\eps) ~,~ \lim_{\eps \rightarrow 0} \frac{R_h(\eps)}{\eps^2} = 0 \text{ in $C^2(S^{n-1})$.}
\end{equation}

Now denote:
\[
J_p(\eps) := V(K +_p \eps \cdot f) = V(K+_p \eps \cdot f, \ldots, K+_p \eps \cdot f) ,
\]
apply the differential formula for mixed-volume (\ref{eq:mixed-vol-bodies}), and expand using the above Taylor expansion. Inspecting the coefficients of $\eps$ and $\frac{\eps^2}{2}$ in the expansion of $J_p(\eps)$, and noting that the remainder term $R_{J_p}(\eps)$ satisfies:
\[
\abs{R_{J_p}(\eps)} \leq C_{h_K,z,p,n} \norm{R_h(\eps)}_{C^2(S^{n-1})} \text{ and hence } \lim_{\eps \rightarrow 0} \frac{R_{J_p}(\eps)}{\eps^2} = 0 ,
\]
it follows by multi-linearity of mixed-volume and invariance under permutation of its arguments that:
\begin{equation} \label{eq:J-derivs}
J_p'(0) = n V(z h_K ; 1) ~,~ J_p''(0) = n V((1-p) z^2 h_K ; 1) + {n \choose 2} 2 V(z h_K; 2) ,
\end{equation}
where we employ the abbreviation:
\[
V(f; m) = V(\underbrace{f,\ldots,f}_{\text{$m$ times}},\underbrace{h_K,\ldots,h_K}_{\text{$n-m$ times}}). 
\]

\begin{prop}[Second $L^p$-Minkowski Inequality] \label{prop:2nd-p-Minkowski}
Given $K \in \K^{2}_{+,e}$ and $p <1 $, the local $p$-BM conjecture (\ref{eq:local-p-BM}) for $K$  is equivalent to the assertion that:
\begin{equation} \label{eq:p-BM-mixed-vols}
\forall z \in C^2_e(S^{n-1}) \;\;\; \frac{1}{V(K)} V(z h_K;1)^2 \geq \frac{n-1}{n-p} V(z h_K ; 2) + \frac{1-p}{n-p} V(z^2 h_K;1) .
\end{equation}
\end{prop}
\begin{proof}
The local $p$-BM conjecture is the assertion that $(\frac{d}{d\eps})^2|_{\eps=0} \frac{1}{p} J_p(\eps)^{p/n} \leq 0$ for all $\frac{1}{p} f^p \in C^{2}_{e}(S^{n-1})$, or equivalently, that:
\[
\frac{n-p}{n} J_p'(0)^2 \geq J_p(0) J_p''(0) .
\]
Plugging in the expressions for $J_p^{(m)}(0)$ obtained in (\ref{eq:J-derivs}), the equivalence with (\ref{eq:p-BM-mixed-vols}) immediately follows after noting that (\ref{eq:z}) gives a bijection between $\frac{1}{p} f^p \in C^2_e(S^{n-1})$ and $z \in C^2_e(S^{n-1})$. \end{proof}

\begin{rem} \label{rem:scale-invariance}
Note that the validity of (\ref{eq:p-BM-mixed-vols}) remains invariant under $K \mapsto \lambda K$ (as all terms are $n$-homogeneous in $K$), $z \mapsto \lambda z$ (as all terms are quadratic in $z$) and $z \mapsto z + \lambda$ (this requires a quick check of both the linear term in $\lambda$ and the quadratic one). Consequently, (\ref{eq:p-BM-mixed-vols}) is equivalent to:
\[
\forall z \in C^2_e(S^{n-1})  \;\; V(z h_K;1) = 0 \;\; \Rightarrow \;\; -V(z h_K ; 2) \geq \frac{1-p}{n-1} V(z^2 h_K;1) . 
\]
\end{rem}

\subsection{Comparison with classical $p=1$ case}

Before proceeding, let us compare (\ref{eq:p-BM-mixed-vols}) with the classical case $p=1$. Plugging in $p=1$ in (\ref{eq:p-BM-mixed-vols}), the quadratic term in $z^2$ disappears, and denoting $w = z h_K$, we obtain:
\begin{equation} \label{eq:Minkowski-2nd}
\forall w \in C^2_e(S^{n-1}) \;\;\; V(w;1)^2 \geq V(w ; 2) V(w;0) ;
\end{equation}
this is the classical Minkowski's second inequality, valid without any evenness assumption on $K$ or $w$, which indeed is well-known to be equivalent to the local concavity of $\eps \mapsto V(K + \eps w)^{\frac{1}{n}}$, and hence to the global Brunn--Minkowski inequality. 

\medskip
In addition, let us check that the conjectured (\ref{eq:p-BM-mixed-vols}) for $p \in [0,1)$ is indeed stronger than the classical case $p=1$. To see this, first note that by (\ref{eq:dSK}), for all $w \in C^2(S^{n-1})$ (in fact, $C(S^{n-1})$ is enough for the first mixed volume):
\begin{equation} \label{eq:1st-mixed-vol}
V(w;1) = V(w,h_K,\ldots,h_K) = \frac{1}{n} \int_{S^{n-1}} w \SS(h_K,\ldots,h_K) d\theta = \frac{1}{n} \int_{S^{n-1}} w dS_K . 
\end{equation}
In particular $V(K) = \frac{1}{n}\int_{S^{n-1}} h_K dS_K$. 
Applying Cauchy--Schwarz, it follows that:
\begin{equation} \label{eq:CS}
V(z^2 h_K;1) = \frac{1}{n} \int_{S^{n-1}} z^2 h_K dS_K \geq \frac{\brac{\frac{1}{n} \int_{S^{n-1}} z h_K dS_K}^2}{\frac{1}{n} \int_{S^{n-1}} h_K dS_K} =
\frac{V(z h_K;1)^2}{V(K)} .
\end{equation}
This means that when $p \in [0,1)$, since $\frac{1-p}{n-p} > 0$, we can always make the inequality (\ref{eq:p-BM-mixed-vols}) weaker by replacing the last term by the one on the left-hand-side, and after rearranging terms and setting $w = z h_K$, we obtain (\ref{eq:Minkowski-2nd}) corresponding to the classical case $p=1$.

\subsection{Infinitesimal Formulation On $S^{n-1}$}

Let us now obtain an explicit expression for the second mixed volume appearing in (\ref{eq:p-BM-mixed-vols}). 
Recall by (\ref{eq:Q-formula}) that:
\begin{equation} \label{eq:Q-formula2}
 Q_K^{i,j} := Q^{i,j}(D^2 h_K,\ldots,D^2 h_K) = \frac{1}{n-1} \text{det}(D^2 h_K) ((D^2 h_K)^{-1})^{i,j} .
\end{equation}
Plugging this below assuming $K \in \K^3_+$, after integrating by parts in $j$ and using the divergence-free property of $Q_K^{i,\cdot}$, and finally recalling (\ref{eq:dSK}), we obtain for any $w \in C^2(S^{n-1})$:
\begin{align}
\nonumber & V(w;2) = V(w,w,h_K,\ldots,h_K) \\
\nonumber & = \frac{1}{n} \int_{S^{n-1}} w \SS(w,h_K,\ldots,h_K) d\theta = \frac{1}{n} \int_{S^{n-1}} w (w_{i,j} + w \delta_{i,j}) Q_K^{i,j} d\theta \\
\nonumber & = \frac{1}{n} \brac{\int_{S^{n-1}} (Q_K)^{i}_{i} w^2  d\theta - \int_{S^{n-1}}  (Q_K^{i,j} w)_j  w_i  d\theta} \\
\nonumber  & = \frac{1}{n} \brac{\int_{S^{n-1}} (Q_K)^{i}_{i} w^2  d\theta - \int_{S^{n-1}}  Q_K^{i,j} w_j w_i  d\theta} \\
\label{eq:2nd-mixed-vol} & = \frac{1}{n (n-1)} \brac{\int_{S^{n-1}} ((D^2 h_K)^{-1})_i^i \; w^2 dS_K -  \int_{S^{n-1}} ((D^2 h_K)^{-1})^{i,j} w_i w_j dS_K } .
\end{align}

\begin{prop}[Infinitesimal $p$-BM on $S^{n-1}$] \label{prop:p-BM-Sphere}
Given $K \in \K^{2}_{+,e}$ and $p < 1$, the local $p$-BM conjecture (\ref{eq:local-p-BM}) for $K$ is equivalent to the assertion that:
\begin{align}
\label{eq:p-BM-Sphere} & \forall w \in C^1_e(S^{n-1}) \;\;\;\; \int_{S^{n-1}} \scalar{(D^2 h_K)^{-1} \nabla_{S^{n-1}} w, \nabla_{S^{n-1}} w} dS_K \geq \\
\nonumber & \int_{S^{n-1}} \text{tr}((D^2 h_K)^{-1}) w^2 dS_K + (1-p) \int_{S^{n-1}} \frac{w^2}{h_K} dS_K  - \frac{n-p}{n} \frac{1}{V(K)} \brac{\int_{S^{n-1}} w dS_K}^2 .
\end{align}
\end{prop}
\begin{proof}
The equivalence for $w \in C^2_e(S^{n-1})$ and $K \in \K^3_+$ is immediate by setting $w = z h_K$ in (\ref{eq:p-BM-mixed-vols}) and using the expressions for the first and second mixed-volume derived in (\ref{eq:1st-mixed-vol}) and (\ref{eq:2nd-mixed-vol}). When $w \in C^1_e(S^{n-1})$ and $K \in \K^2_+$, the equivalence follows by a standard approximation argument, utlizing the fact that only first derivatives of $w$ and second derivatives of $h_K$ appear in (\ref{eq:2nd-mixed-vol}).
\end{proof}

\begin{rem}
The case $p=0$ (local log-BM) of Proposition \ref{prop:p-BM-Sphere} was previously derived by Colesanti--Livshyts--Marsiglietti in \cite{CLM-LogBMForBall}. 
\end{rem}

\subsection{Infinitesimal Formulation On $\partial K$} \label{subsec:FromSToPartialK}

By using the Weingarten map:
\[
\nu_{\partial K}^{-1} : S^{n-1} \ni \theta \mapsto \nabla_{\Real^n} h_K(\theta) \in\partial K ,
\]
we may transfer the infinitesimal formulation of the local $p$-BM conjecture obtained in the previous subsection from $S^{n-1}$ to $\partial K$. Indeed, recall that $d\nu_{\partial K}(x) = \II_{\partial K}(x)$, $d \nu_{\partial K}^{-1}(\theta) = D^2 h_K (\theta)$ and
$D^2 h_K(\nu_{\partial K}(x)) = \II_{\partial K}^{-1}(x)$ (with the usual identification of tangent spaces), and that $\text{Jac}( \nu_{\partial K}^{-1})(\theta) = \text{det}(D^2 h_K)(\theta) = \frac{dS_K(\theta)}{d \theta}$. 

Denoting $\Psi(x) = w(\nu_{\partial K}(x))$, we see that:
\[
\int_{S^{n-1}} w(\theta) dS_K(\theta) = \int_{\partial K} \Psi(x) dx ,
\]
\[
\int_{S^{n-1}} \text{tr}((D^2 h_K)^{-1}) w^2(\theta) dS_K(\theta)  = \int_{\partial K} \text{tr}(\II_{\partial K}) \Psi^2(x) dx ,
\]
and as $h_K(\nu_{\partial K}(x)) = \scalar{x,\nu_{\partial K}(x)}$, 
\[
\int_{S^{n-1}} \frac{w^2(\theta)}{h_K(\theta)} dS_K(\theta) = \int_{\partial K} \frac{\Psi^2(x)}{\scalar{x,\nu_{\partial K}(x)}} dx .
\]
Lastly, as $w(\theta) = \Psi(\nu_{\partial K}^{-1}(\theta))$ we have, setting $x = \nu_{\partial K}^{-1}(\theta)$: \[
\nabla_{S^{n-1}} w(\theta) = d\nu_{\partial K}^{-1}(\theta) \nabla_{\partial K} \Psi(x) = \II_{\partial K}^{-1}(x) \nabla_{\partial K}\Psi(x) ,
\]
and therefore:
\begin{align*}
\int_{S^{n-1}} \scalar{(D^2 h_K)^{-1} \nabla_{S^{n-1}} w , \nabla_{S^{n-1}} w} dS_K & = \int_{\partial K} \scalar{\II_{\partial K} \II_{\partial K}^{-1} \nabla_{\partial K} \Psi, \II_{\partial K}^{-1} \nabla_{\partial K} \Psi} dx \\
& = \int_{\partial K}\ \scalar{\II_{\partial K}^{-1} \nabla_{\partial K} \Psi, \nabla_{\partial K} \Psi} dx .
\end{align*}

Plugging the above identities into Proposition \ref{prop:p-BM-Sphere}, and using the fact that the Weingarten map is a $C^1$ diffeomorphishm when $K \in \K^2_+$, we immediately obtain:

\begin{prop}[Infinitesimal $p$-BM on $\partial K$] \label{prop:p-BM-K}
Given $K \in \K^{2}_{+,e}$ and $p < 1$, the local $p$-BM conjecture (\ref{eq:local-p-BM}) for $K$ is equivalent to the assertion that:
\begin{align}
\label{eq:pBMPsi} & \forall \Psi \in C^1_e(\partial K) \;\;\; \int_{\partial K}\ \scalar{\II_{\partial K}^{-1} \nabla_{\partial K} \Psi, \nabla_{\partial K} \Psi} dx  \geq\\
\nonumber &  \int_{\partial K} H_{\partial K}(x) \Psi^2(x) dx + (1-p) \int_{\partial K} \frac{\Psi^2(x)}{\scalar{x,\nu_{\partial K}(x)}} dx - \frac{n-p}{n} \frac{1}{V(K)} \brac{\int_{\partial K} \Psi(x) dx}^2 ,
\end{align}
where $H_{\partial K}(x) = \text{tr}(\II_{\partial K})(x)$ denotes the mean-curvature of $\partial K$ at $x \in \partial K$.  
\end{prop}

\begin{rem} \label{rem:p-BM-K}
Remark \ref{rem:scale-invariance} translates into the fact that the validity of (\ref{eq:p-BM-Sphere}) is invariant under $w \mapsto w + \lambda h_K$, or equivalently, that
(\ref{eq:pBMPsi}) is invariant under $\Psi \mapsto \Psi + \lambda \scalar{x , \nu_{\partial K}(x)}$. Consequently, (\ref{eq:pBMPsi}) is equivalent to:
\begin{align}
\label{eq:pBMPsiNor} & \forall \Psi \in C^1_e(\partial K) \;\;\; \int_{\partial K} \Psi(x) dx = 0 \;\; \Rightarrow \\
\nonumber &  \int_{\partial K}\ \scalar{\II_{\partial K}^{-1} \nabla_{\partial K} \Psi, \nabla_{\partial K} \Psi} dx  \geq \int_{\partial K} H_{\partial K}(x) \Psi^2(x) dx + (1-p) \int_{\partial K} \frac{\Psi^2(x)}{\scalar{x,\nu_{\partial K}(x)}} dx .
\end{align}
\end{rem}

\begin{rem}
The classical case $p=1$ of Proposition \ref{prop:p-BM-K} was previously derived by Colesanti in \cite{ColesantiPoincareInequality}. See also \cite{KolesnikovEMilmanReillyPart2} for extensions of the case $p=1$ to the setting of weighted Riemannian manifolds satisfying the Curvature-Dimension condition $\text{CD}(0,N)$, \cite{KolesnikovEMilman-GaussianPoincare} for an analogous statement involving Ehrhard's inequality for the Gaussian measure, and \cite{ColesantiEugenia-PoincareFromAF} for a version involving other intrinsic volumes. 
\end{rem}

\bigskip

\section{Relation to Hilbert--Brunn--Minkowski Operator and Linear Equivariance} \label{sec:LK}

It is easy to show that the validity of the local $p$-BM conjecture (\ref{eq:local-p-BM}) as a function of $K \in \K^2_{+,e}$ is invariant under linear transformations. More precisely, 
given $T \in GL_n$ and $\frac{1}{p} f^p \in C^2_e(S^{n-1})$, define $f_T : S^{n-1} \rightarrow \Real$ by extending $f$ as a $1$-homogeneous function to $\Real^n$ and pushing it forward via $T^{-t}$, namely:
\[
f_T(\theta) = T^{-t}_* f (\theta) = f(T^t \theta) = f \brac{\frac{T^t \theta}{\abs{T^t \theta}}} \abs{T^t \theta} .
\]
It is easy to see that $p f_T^p \in C^2_e(S^{n-1})$ and that:
\begin{equation} \label{eq:Lin0}
T(K) +_p \eps \cdot f_T = T(K +_p \eps \cdot f) .
\end{equation}
Indeed, by definition, we have:
\begin{equation} \label{eq:h-contra}
h_{T(K)}(\theta) = h_K(T^t \theta) ,
\end{equation}
ans so for small enough $\abs{\eps}\ll 1$ and all $\theta \in S^{n-1}$: \begin{align*}
&h_{T(K) +_p \eps \cdot f_T}(\theta) = \brac{h^p_{T(K)}(\theta) + \eps f^p_T(\theta)}^{\frac{1}{p}} = \brac{h^p_K\brac{\frac{T^t \theta}{\abs{T^t \theta}}} + \eps f^p \brac{\frac{T^t \theta}{\abs{T^t \theta}}}}^{\frac{1}{p}} \abs{T^t \theta} \\
& = \brac{h_{K}^p + \eps f^p}^{\frac{1}{p}} \brac{\frac{T^t \theta}{\abs{T^t \theta}}}  \abs{T^t \theta} = h_{K +_p \eps \cdot f}  \brac{\frac{T^t \theta}{\abs{T^t \theta}}}  \abs{T^t \theta} = h_{K +_p \eps \cdot f} (T^t \theta) = h_{T(K +_p \eps \cdot f)}(\theta) .
\end{align*}
In fact, it is equally easy to check that (\ref{eq:Lin0}) remains valid for all $\eps \in \Real$ as equality between Aleksandrov bodies, even when the support function of $K+_p \eps \cdot f$ does not coincide with the corresponding defining Wulff function. 

In any case, since $\frac{1}{p} f^p \mapsto \frac{1}{p} f^p_T$ is clearly a bijection on $C^2_e(S^{n-1})$, the invariance under linear transformations $K \mapsto T(K)$ of the validity of (\ref{eq:local-p-BM}) immediately follows. Consequently, the invariance under linear transformations of the validity of the equivalent infinitesimal versions (\ref{eq:p-BM-Sphere}) on $S^{n-1}$ and (\ref{eq:pBMPsi}) on $\partial K$ follows as well. This is not surprising, since this is just an infinitesimal manifestation of the (easily verifiable) fact that:
\[
T(K +_p L) = T(K) +_p T(L) \;\;\; \forall T \in GL_n ,
\]
which implies that the validity of the $p$-BM conjecture (\ref{eq:p-BM}) for $K,L$ is equivalent to that for $T(K),T(L)$.

\medskip

However, one of our goals in this section is to establish a somewhat deeper linear equivariance of a certain second-order linear differential operator $L_K$ associated to every $K \in \K^2_+$, which extends the above elementary observation. Modulo our different normalization, this operator was already been considered by Hilbert in his proof of the Brunn--Minkowski inequality, and generalized by Aleksandrov in his second proof of the Aleksandrov--Fenchel inequality (see \cite[pp. 108--110]{BonnesenFenchelBook}). Consequently, we call $L_K$ the Hilbert--Brunn--Minkowski operator.

\subsection{Hilbert--Brunn--Minkowski operator} \label{subsec:LK}

\begin{defn*}[Hilbert--Brunn--Minkowski operator]
Given $K \in \K^2_+$, the associated Hilbert--Brunn--Minkowski operator, denoted $L_K$, is the second-order linear differential operator on $C^2(S^{n-1})$ defined by:
\[
L_K := \tilde{L}_K - Id ~,~ \tilde{L}_K z := \frac{\SS(z h_K , \overbrace{h_K, \ldots,h_K}^{\text{$n-2$ times}})}{\SS(\underbrace{h_K,\ldots,h_K}_{\text{$n-1$ times}})} .
\]
\end{defn*}
\begin{rem}
Abbreviating $Q^{i,j} = Q^{i,j}(D^2 h_K)$ and $h = h_K$, observe that:
\begin{align} 
\nonumber & \SS(z h,h,\ldots,h) - z \SS(h,h,\ldots,h) \\
\nonumber & = Q^{i,j} \brac{((z h)_{i,j} + z h \delta_{i,j})  - z (h_{i,j} + h \delta_{i,j})} \\
\label{eq:Dzh-zDh} & = Q^{i,j} (z_i h_j + h_i z_j + h z_{i,j}) = Q^{i,j} \frac{1}{h} (h^2 z_i)_j  ,
\end{align}
where the last transition follows by the symmetry of $Q^{i,j}$. 
Recalling (\ref{eq:Q-formula2}), we obtain:
\begin{equation} \label{eq:LK-explicit}
L_K z = \frac{1}{n-1} ((D^2 h)^{-1})^{i,j} (z_i h_j + h_i z_j + h z_{i,j}) = \frac{1}{n-1}  \frac{((D^2 h)^{-1})^{i,j}}{h} (h^2 z_i)_j  .
\end{equation}
In particular, we see that $L_K$ has no zeroth order term. 
\end{rem}

\begin{defn*}[Cone Measure]
The cone measure of $K$, denoted $dV_K$, is the Borel measure on $S^{n-1}$ defined by:
\[
dV_K := \frac{1}{n} h_K dS_{K} .
\]
\end{defn*}

It is easy to check that for any Borel set $A \subset S^{n-1}$, $V_K(A)$ equals the volume of the cone in $K$ generated by $\nu_K^{-1}(A)$ with vertex at the origin. In particular, $V_K(S^{n-1}) = V(K)$. When $K \in \K^2_+$, we have by (\ref{eq:dSK}):
\[
dV_K = \frac{1}{n} h_K \SS(h_K,\ldots,h_K) d\theta . 
\]

Fixing $K \in \K^2_+$, and recalling the terms appearing in the second $p$-Minkowski inequality (Proposition \ref{prop:2nd-p-Minkowski}), we rewrite:
\begin{align}
\nonumber V(z h_K ; 1) &= \frac{1}{n} \int_{S^{n-1}} z h_K dS_K = \int_{S^{n-1}} z dV_K \\
\nonumber V(z^2 h_K ; 1) & = \frac{1}{n} \int_{S^{n-1}} z^2 h_K dS_K = \int_{S^{n-1}} z^2 dV_K \\  
\label{eq:tilde-LK-as-mixedvol}
V(z h_K ; 2) & = \frac{1}{n} \int_{S^{n-1}} z h_K \SS(z h_K, h_K,\ldots,h_K) d\theta = \int_{S^{n-1}} (\tilde{L}_K z) z dV_K . 
\end{align}

Plugging these expressions into Proposition \ref{prop:2nd-p-Minkowski}, and applying Remark \ref{rem:scale-invariance}, we obtain:
\begin{prop} \label{prop:LK-SG}
Given $K \in \K^{2}_{+,e}$ and $p < 1$, the local $p$-BM conjecture (\ref{eq:local-p-BM}) for $K$ is equivalent to the assertion that:
\[
\forall z \in C^2_e(S^{n-1}) \;\;\; \int_{S^{n-1}} z dV_K = 0 \;\; \Rightarrow \;\; \int_{S^{n-1}} (-L_K z) z dV_K \geq \frac{n-p}{n-1} \int_{S^{n-1}} z^2 dV_K . 
\]
\end{prop}

The latter formulation has a clear spectral flavor. Let us make  this more precise. 

\begin{thm} \label{thm:LK}
 Let $K \in \K^2_+$. 
\begin{enumerate}
\item The operator $L_K : C^2(S^{n-1}) \rightarrow C(S^{n-1})$ is symmetric on $L^2(dV_K)$. 
\item The operator $L_K$ is elliptic and hence admits a unique self-adjoint extension in $L^2(dV_K)$ with domain $\text{Dom}(L_K) = H^2(S^{n-1})$, which we continue to denote by $L_K$. Its spectrum  $\sigma(L_K) \subset \Real$ is discrete, consisting of a countable sequence of eigenvalues of finite multiplicity tending (in absolute value) to $\infty$. 
\item The Dirichlet form associated to $L_K$ is given by:
\begin{equation} \label{eq:Dirichlet}
\int_{S^{n-1}} (-L_K z) z dV_K = \frac{1}{n-1} \int_{S^{n-1}} h_K ((D^2 h_K)^{-1})^{i,j} z_i z_j dV_K .
\end{equation}
In particular,  $-L_K$ is positive semi-definite and therefore $\sigma(-L_K) \subset \Real_+$. We denote its eigenvalues (arranged in non-decreasing order and repeated according to multiplicity) by $\set{\lambda_m}_{m \geq 0}$. 
\item If $\set{K_i} \subset \K^2_+$ and $K_i \rightarrow K$ in $C^2$ then $\lim_{i \rightarrow \infty} \lambda_m(-L_{K_i}) = \lambda_m(-L_K)$ for all $m \geq 0$. 
\item $0$ is an eigenvalue of $-L_K$ with multiplicity one corresponding to the one-dimensional subspace of constant functions $E_0 := \text{span}(\mathbf{1})$. 
\item $-L_K|_{\mathbf{1}^{\perp}} \geq Id|_{\mathbf{1}^{\perp}}$ (as positive semi-definite operators on $L^2(dV_K)$). \item $1$ is an eigenvalue of $-L_K$ with multiplicity precisely $n$ corresponding to the $n$-dimensional subspace $E_1^K$ spanned by the (renormalized) linear functions:
\[
 \ell^K_v(\theta) = \frac{\scalar{\theta,v}}{h_K(\theta)} ~,~ v \in \Real^n .
 \]
\end{enumerate}
\end{thm}
\begin{proof}
\begin{enumerate}
\item If $z_1,z_2 \in C^2(S^{n-1})$ then:
\begin{align*} \int_{S^{n-1}} (L_K z_1) z_2 dV_K & = \frac{1}{n} \int_{S^{n-1}} z_1 h_K \SS(z_2 h_K,h_K,\ldots,h_K) d\theta \\
& = V(z_1 h_K, z_2 h_K, h_K,\ldots,h_K) .
\end{align*}
By Lemma \ref{lem:mixed-vol-symmetric}, the mixed volume is invariant under permutations, so the right-hand-side is symmetric in $z_1,z_2$, and hence so is the left-hand-side, as asserted. 
\item The ellipticity of $L_K$ follows since by (\ref{eq:LK-explicit}), 
its (leading) second-order term is given by:
\[
\frac{h_K}{n-1} ((D^2 h_K)^{-1})^{i,j} z_{i,j} .
\]
Since $K \in \K^2_+$, there exist $a,b > 0$ so that $a \delta \leq D^2 h_K \leq b \delta$ on $S^{n-1}$, where $\delta$ denotes the standard metric on $S^{n-1}$, and the (uniform) ellipticity follows. 

As for essential self-adjointness, it is well known using elliptic regularity theory \cite[Section 8.2]{Taylor-PDE-II-Book}, \cite{StrichartzLaplacianOnManifold} that a symmetric second-order elliptic operator with continuous coefficients on a compact closed manifold $M$ has a unique self-adjoint extension from $C^2(M)$ to the Sobolev space $H^2(M)$. Its resolvent is necessarily compact, and hence its spectrum is discrete.

\item By density, it is enough to verify the representation (\ref{eq:Dirichlet}) for $z \in C^2(S^{n-1})$. Abbreviating as usual $Q^{i,j} = Q^{i,j}(D^2 h_K)$ and $h = h_K$, it follows by (\ref{eq:Dzh-zDh}) and the divergence free property of $Q^{i,\cdot}$ when $K \in \K^3_+$ that:
\[
\int_{S^{n-1}} (-L_K z) z dV_K = - \frac{1}{n} \int_{S^{n-1}} z Q^{i,j} (h^2 z_i)_j d\theta = \frac{1}{n} \int_{S^{n-1}} Q^{i,j} h^2 z_i z_j d\theta. 
\]
The case of a general $K \in \K^2_+$ follows by approximation.
Recalling (\ref{eq:Q-formula2}), the representation (\ref{eq:Dirichlet}) immediately follows. 
The positive semi-definiteness of $-L_K$ follows since $D^2 h$ is positive definite; note that this also follows by (5) and (6) as well.  

\item The continuity of the eigenvalues as a function of the coefficients of a family of uniformly elliptic operators on a compact manifold is classical (e.g. \cite[Theorem 2.3.3]{HenrotBook}). 
Indeed, the $C^2$ convergence of $K_i$ to $K$ ensures that the coefficients of $-L_{K_i}$ converge in $C(S^{n-1})$ to those of $-L_K$; using the uniform ellipticity (as $D^2 h_{K_i} \geq \frac{1}{2} D^2 h_K \geq c \delta$ for some $c > 0$ and large enough $i$), one shows pointwise convergence of the corresponding compact resolvent operators, from whence norm convergence of the resolvent operators is deduced, yielding the convergence of eigenvalues.

\item Clearly $L_K \mathbf{1} = 0$ as it has no zeroth order term (or since clearly $\tilde{L}_K \mathbf{1} = \mathbf{1}$). The standard fact that the multiplicity of the $0$ eigenvalue is precisely one follows since $L_K$ is second-order elliptic with no zeroth order term and since the sphere is a connected compact manifold. Note that this also follows from (6). 

\item The spectral-gap estimate $-L_K|_{\mathbf{1}^{\perp}} \geq Id|_{\mathbf{1}^{\perp}}$ is a deep fact which is equivalent to the (local, and hence global) Brunn--Minkowski inequality. Under a different normalization, this equivalence was first noted by Hilbert (see \cite[pp. 108--109]{BonnesenFenchelBook} and Remark \ref{rem:Hilbert}), who obtained a direct proof of the former spectral-gap estimate by employing the method of continuity, thereby obtaining a novel proof of the Brunn--Minkowski inequality. To see the equivalence, simply apply Proposition \ref{prop:LK-SG} in the classical case $p=1$, and note that the Brunn--Minkowski inequality holds without any symmetry assumptions on $K,L$, so that the evenness assumption on the test function $z$ is unnecessary in this case.

\item It is immediate to check that $\ell^K_v$ is indeed an eigenfunction of $-L_K$ with eigenvalue $1$, since:
\[
\nabla^2_{\Real^n} \scalar{\theta,v} = 0 \;\; \text{and hence} \;\; \tilde{L}_K \ell^K_v = \frac{\SS(\scalar{\theta,v},h_K,\ldots,h_K)}{\SS(h_K,h_K,\ldots,h_K)} = 0 .
\]
Consequently, the multiplicity of the eigenvalue $1$ is at least $n$ (the dimension of linear functionals on $\Real^n$). The fact that there are no other eigenfunctions with eigenvalue $1$, and hence that the corresponding multiplicity is \emph{precisely} $n$, was established by Hilbert (see \cite[p. 110]{BonnesenFenchelBook} for an alternative argument) in his proof of the spectral-gap estimate (6), and in fact constitutes the crux of Hilbert's argument. 
\end{enumerate}
\end{proof}

Theorem \ref{thm:LK}, which modulo our different normalization is essentially due to Hilbert (see Remark \ref{rem:Hilbert} below), interprets the Brunn--Minkowski inequality as a uniform spectral-gap statement (beyond the trivial $\lambda_0 = 0$ eigenvalue corresponding to $E_0 = \text{span}(\mathbf{1})$):
\[
\lambda_1(-L_K) := \min \sigma(-L_K |_{\mathbf{1}^{\perp}}) \geq 1 \;\;\; \forall K \in \K^2_+ . 
\]
Moreover, it provides the \emph{additional} information that $\lambda_1(-L_K) = 1$ with corresponding eigenspace $E^K_1$ of dimension precisely $n$. Consequently, the next eigenvalue $\lambda_{n+1}(-L_K)$, which is obtained by restricting $-L_K$ to the (invariant) subspace perpendicular to $E^K_1 + E_0$, satisfies:
\[
\lambda_{n+1}(-L_K) := \min \sigma(-L_K |_{(E^K_1)^{\perp} \cap \mathbf{1}^{\perp} }) > 1 \;\;\; \forall K \in \K^2_+ . 
\]

\medskip
A naturally arising question, which perhaps could have been asked by Hilbert himself (had he been using our normalization), is whether the above \emph{next} eigenvalue gap \emph{beyond} $1$ is actually \emph{uniform} over all $K \in \K^2_+$. The most convenient way to obtain a necessary condition for this to hold, is to assume that $K$ is origin-symmetric, and so the $\ell^K_v$ eigenfunctions will all be odd, as the ratio of linear (odd) functions and an even one. If we only consider test-functions $z \in H^2(S^{n-1})$ which arise from perturbations of $K$ by another origin-symmetric body $L$, they will always be even, and hence constitute an invariant subspace $E_{\text{even}}$ for $-L_K$, which is in addition automatically perpendicular to $E^K_1$, and hence:
\[
\lambda_{1,e}(-L_K) := \min \sigma(-L_K |_{E_{\text{even}} \cap \mathbf{1}^{\perp}}) \geq \min \sigma(-L_K |_{(E^K_1)^{\perp} \cap \mathbf{1}^{\perp}}) = \lambda_{n+1}(-L_K) .
\]
Proposition \ref{prop:LK-SG} thus translates into an interpretation of the local $p$-BM conjecture as a question on the \emph{even} spectral-gap of $-L_K$ beyond $1$:
\begin{cor} \label{cor:LK-SG}
Given $K \in \K^{2}_{+,e}$ and $p < 1$, the local $p$-BM conjecture (\ref{eq:local-p-BM}) for $K$ is equivalent to the following even spectral-gap estimate for $-L_K$ beyond $1$:
\[
\lambda_{1,e}(-L_K) \geq \frac{n-p}{n-1} .
\]
\end{cor}
\noindent
This gives a concrete spectral reason for the restriction to origin-symmetric convex bodies in the $p$-BM conjecture when $p \in [0,1)$, and explains why the conjecture fails for general convex bodies $K$ -- without being perpendicular to $E^K_{1}$, the spectral-gap beyond $0$ is precisely $1$ and never better (as seen by the test functions $\ell^K_v$, corresponding to translations of $K$).

\begin{rem} \label{rem:extension-to-non-sym}
The above discussion suggests a plausible extension of the local $p$-BM conjecture which does not require that $K \in \K^2_+$ be origin-symmetric. In spectral terms it naturally reads as:
\[
\lambda_{n+1}(-L_K) \geq \frac{n-p}{n-1} ,
\]
or equivalently:
\begin{align}
\label{eq:non-symmetric-z} \forall z \in C^2(S^{n-1})  \;\; & \;\; \int_{S^{n-1}} z h_K dS_K = 0 \; \text{ and } \int_{S^{n-1}} \vec{\theta} \; z(\theta) dS_K(\theta) = \vec{0}  \;\; \Rightarrow \\
\nonumber & \int_{S^{n-1}} (-L_K z) z dV_K \geq \frac{n-p}{n-1} \int_{S^{n-1}} z^2 dV_K .
\end{align}
In terms of local concavity as in (\ref{eq:local-p-BM}), recalling the derivation in Subsection \ref{subsec:2nd-p-Minkowski}, this reads as:
\[
\left . \frac{d^2}{(d\eps)^2} \right |_{\eps=0} \frac{1}{p} V(A[h_K (1 + \eps z)^{\frac{1}{p}}])^{\frac{p}{n}} \leq 0 \;\;\; \text{for all $z \in C^2(S^{n-1})$ satisfying (\ref{eq:non-symmetric-z})}.
\]
However, we do not know how to translate the local condition (\ref{eq:non-symmetric-z}) on $z = \frac{1}{p} \frac{h_L^p - h_K^p}{h_K^p}$ into a global requirement on $K,L$ which would guarantee the local condition along the $p$-Minkowski interpolation between $K$ and $L$. In this regard, we note that it was shown in \cite{XiLeng-DarAndLogBMInPlane} that for all convex bodies $K,L$ in the plane, there exist translations of $K,L$ (``dilation positions") for which the log-BM conjectured inequality holds true. 
\end{rem}

\begin{rem} \label{rem:Hilbert}
Hilbert originally considered \cite[pp. 108-109]{BonnesenFenchelBook} the operator:
\[
H_K w := \SS(w , h_K,\ldots,h_K) ,
\]
which is elliptic and essentially self adjoint on $L^2(d\theta)$ with respect to the Lebesgue measure $d\theta$ on $S^{n-1}$. However, this operator is not well suited for our investigation. Indeed, setting as usual $w = z h_K$ in Remark \ref{rem:scale-invariance}, the local $p$-BM conjecture for $K \in \K^2_{+,e}$ is then equivalent to:
\[
\forall w \in C^2_e(S^{n-1}) \;\;\; \int w dS_K = 0 \;\; \Rightarrow \int_{S^{n-1}} (-H_K w) w d\theta \geq \frac{1-p}{n-1} \int_{S^{n-1}}\frac{w^2}{h_K} dS_K ,
\]
which does not have a nice spectral interpretation when $p \neq 1$. Furthermore, the correspondence $K \mapsto H_{K}$ does not posses the useful equivariance property under linear transformations we shall establish for $L_{K}$ in the next subsection, and it is not even invariant under homothety, so there is no chance of obtaining uniform estimates for $H_K$ valid for all convex bodies $K$. 

The normalization we employ in our definition of $L_K$ may be uniquely characterized (up to scaling) as defining a second-order differential operator with no zeroth order term, which is essentially self-adjoint  on $L^2(\mu_K)$, and so that the conjectured second $p$-Minkowski inequality  (\ref{eq:p-BM-mixed-vols}) may be equivalently rewritten as a spectral-gap inequality on the subspace of $L^2(\mu_K)$ of even functions perpendicular to the constant ones. 
The unique (up to scaling) measure $\mu_K$ satisfying these requirements turns out to be the cone measure $dV_K$ (for all values of $p$, not just $p=0$!), amounting further evidence to the intimate relation between the $p$-BM conjecture and the cone measure.
\end{rem}

\subsection{Linear equivariance of the Hilbert--Brunn--Minkowski operator}

In this subsection, we establish an important equivariance property of the Hilbert-Brunn--Minkowski operators $L_{T(K)}$ under linear transformations $T \in GL_n$. 

Denote by $T^{(0)}$ the following ``$0$-homogeneous" linear change of variables:
\[
T^{(0)} : S^{n-1} \ni \theta \mapsto \frac{T^{-t} \theta}{\abs{T^{-t} \theta}} \in S^{n-1}.
\]
We denote by $T^{(0)}_*z$ the push-forward of a (Lebesgue) measurable function $z \in \mathcal{L}(S^{n-1})$ via $T^{(0)}$, i.e. an application of a linear change of variables when treating $z$ as a $0$-homogeneous function on $\Real^n$:
\[
T^{(0)}_* : \mathcal{L}(S^{n-1}) \ni z(\theta) \mapsto z((T^{(0)})^{-1} \theta) = z(T^t \theta) \in \mathcal{L}(S^{n-1})  .
\]

We now state the two main results of this subsection:

\begin{lem} \label{lem:ST-Iso}
$T^{(0)}$ pushes forward $dV_K$ onto $\frac{1}{\abs{\det(T)}} dV_{T(K)}$. In particular, $T^{(0)}_*$ is an isometry from $L^2(d\tilde{V}_K)$ to $L^2(d\tilde{V}_{T(K)})$, where $d\tilde{V}_Q := dV_Q / V(Q)$ is the normalized cone probability measure. 
\end{lem}

\begin{thm} \label{thm:LTK}
For any $K \in \K^2_+$ and $T \in GL_n$, the following diagram commutes:
\begin{align*}
L_K:   L^2(dV_K) \supset & H^2(S^{n-1}) \rightarrow  L^2(dV_K) \\
 & T^{(0)}_* \downarrow \hspace{40pt}  T^{(0)}_*  \downarrow \\
 L_{T(K)} :  L^2(dV_{T(K)}) \supset &  H^2(S^{n-1}) \rightarrow L^2(dV_{T(K)}) .
\end{align*}
In particular, it follows by the previous lemma that $L_K$ and $L_{T(K)}$ are conjugates via an isometry of Hilbert spaces, and therefore have the same spectrum: 
\[
\sigma(-L_K) = \sigma(-L_{T(K)}). \]
\end{thm}

The proof involves several calculations. We will constantly use the linear contravariance of the support function (\ref{eq:h-contra}).

\begin{lem} \label{lem:JacST}
\[
\abs{\text{Jac} \; T^{(0)}(\theta)} = \frac{1}{\abs{\det(T)} \abs{T^{-t} \theta}^n} .
\]
\end{lem}
\begin{proof}
Complete $\theta$ to an orthonormal basis $\theta,e_1,\ldots,e_{n-1}$ of $\Real^n$. Denote by $V(\set{v_i}_{i=1}^{n-1})$ the volume of the $n-1$-dimensional parallelepiped spanned by $v_1,\ldots,v_{n-1}$. Now calculate:
\[
\abs{\text{Jac} \; T^{(0)}(\theta)} =  V\brac{ \set{\frac{P_{(T^{-t} \theta)^{\perp}} T^{-t} e_i}{\abs{T^{-t} \theta}}}_{i=1}^{n-1}} = \frac{1}{\abs{T^{-t} \theta}^{n-1}} V\brac{\set{P_{(T^{-t} \theta)^{\perp}} T^{-t} e_i}_{i=1}^{n-1} }.
 \]
 On the other hand, by expanding the determinant of $T^{-t}$ (volume of the parallelepiped spanned by the vectors $T^{-t} \theta,T^{-t} e_1,\ldots,T^{-t} e_{n-1}$):
 \[
 \abs{\det(T^{-t})} = V\brac{ \set{P_{(T^{-t} \theta)^{\perp}} T^{-t} e_i}_{i=1}^{n-1} } \abs{T^{-t} \theta} ,
 \]
 and so the assertion follows. 
\end{proof}

\begin{lem} \label{lem:DST}
Let $h_1,\ldots,h_{n-1}$ denote $1$-homogeneous $C^2$ functions on $\Real^n$. Then:
\[
\SS(h_1 \circ T^t,\ldots, h_{n-1} \circ T^t)(T^{(0)} \theta) = \det(T)^2 \abs{T^{-t} \theta}^{n+1} \SS(h_1, \ldots,h_{n-1})(\theta) .
\]
\end{lem}
\begin{proof}
Let us abbreviate $\nabla^2 = \nabla^2_{\Real^n}$. 
Since $\nabla^2 (g \circ T^t) (x) = T (\nabla^2 g)(T^t x) T^t$, we calculate, using the fact that $\nabla^2 g$ is $-1$-homogeneous if $g$ is $1$-homogeneous and the multi-linearity of the mixed discriminant $D_n$:
\begin{align*}
& \SS(h_1 \circ T^t,\ldots, h_{n-1} \circ T^t)(T^{(0)} \theta) \\
= & D_{n}(\nabla^2_{\Real^n} (h_1 \circ T^t)(T^{(0)} \theta), \ldots, \nabla^2_{\Real^n} (h_{n-1} \circ T^t)(T^{(0)} \theta),T^{(0)} \theta \otimes T^{(0)} \theta) \\
= & \abs{T^{-t} \theta}^{n-3} D_{n}(T (\nabla^2 h_1)(\theta) T^t ,  \ldots, T (\nabla^2 h_{n-1})(\theta) T^t, T^{-t} \theta \otimes T^{-t} \theta) \\
= & \det(T)^2 \abs{T^{-t} \theta}^{n-3} D_{n}(\nabla^2 h_1(\theta),  \ldots, \nabla^2 h_{n-1}(\theta), T^{-2t} \theta \otimes T^{-2t} \theta) .
\end{align*}
Recalling that $\nabla^2 h_1(\theta) = P_{\theta^{\perp}} \nabla^2 h_1(\theta) P_{\theta^{\perp}}$, we proceed by expanding the mixed discriminant in the last entry:
\begin{align*}
= & \det(T)^2 \abs{T^{-t} \theta}^{n-3} \scalar{T^{-2t} \theta, \theta}^2 D_{n-1}(P_{\theta^{\perp}} \nabla^2 h_1(\theta) P_{\theta^{\perp}}, \ldots, P_{\theta^{\perp}} \nabla^2 h_{n-1}(\theta) P_{\theta^{\perp}}) \\
= & \det(T)^2 \abs{T^{-t} \theta}^{n+1} \SS(h_1,\ldots,h_{n-1})(\theta) .
\end{align*}
\end{proof}

\begin{proof}[Proof of Lemma \ref{lem:ST-Iso}]
Recall that the surface area measure $dS_K$ and the cone measure $dV_K$ are defined as the following measures on $S^{n-1}$:
\[
dS_K = \SS(h_K,\ldots,h_K) d\theta ~,~ dV_K = \frac{1}{n} h_K dS_K .
\]
Lemma \ref{lem:DST} implies that:
\begin{align*} dS_{T(K)}(T^{(0)} \theta) & = \SS(h_K \circ T^t,\ldots, h_K \circ T^t)(T^{(0)} \theta) d\theta \\
& = \det(T)^2 \abs{T^{-t} \theta}^{n+1}  \SS(h_K,\ldots,h_K)(\theta) d\theta ,
\end{align*} 
and so together with Lemma \ref{lem:JacST}:
\begin{align*} & dV_{T(K)}(T^{(0)} \theta) = \frac{1}{n} h_{T(K)}(T^{(0)} \theta) dS_{T(K)}(T^{(0)} \theta) \\
& = \det(T)^2 \abs{T^{-t} \theta}^{n+1} \frac{1}{n}  h_K(T^{t} T^{(0)} \theta) \SS(h_K,\ldots,h_K)(\theta) d\theta  \\
& = \det(T)^2 \abs{T^{-t} \theta}^{n} \frac{1}{n}  h_K(\theta) \SS(h_K,\ldots,h_K)(\theta) d\theta = \frac{1}{\abs{\text{Jac} \; T^{(0)}(\theta)}} \abs{\det(T)} dV_K(\theta)  ,
\end{align*}
confirming that $T^{(0)}$ pushes forward $dV_K$ onto $\frac{1}{\abs{\det (T)}} dV_{T(K)}$. 
\end{proof}

\begin{proof}[Proof of Theorem \ref{thm:LTK}]
We would like to prove that:
\[
(T^{(0)}_*)^{-1} \circ L_{T(K)} \circ T^{(0)}_* = L_K   ,
\]
or equivalently, that for all $z \in  H^2(S^{n-1})$:
\begin{equation} \label{eq:diagram-commutes0}
L_{T(K)}(T^{(0)}_* z)(T^{(0)} \theta) = (L_K z)(\theta) .
\end{equation}
By density, it is enough to establish this for $z \in C^2(S^{n-1})$. 
Recall that in this case:
\[
\tilde{L}_K(z) = \frac{\SS(z h_K, h_K, \ldots,h_K)}{\SS(h_K,\ldots,h_K)} ,
\]
so that:
\[
\tilde{L}_{T(K)}(T^{(0)}_* z)(T^{(0)} \theta) = \frac{\SS(T^{(0)}_* z \cdot h_{T(K)}, h_{T(K)}, \ldots,h_{T(K)})(T^{(0)} \theta)}{\SS(h_{T(K)},\ldots,h_{T(K)})(T^{(0)} \theta)} .
\]
We think of $z$ as $0$-homogeneous ($T^{(0)}_* z (x) = z(T^{t} x)$) and of course $h_K$ is $1$-homogeneous and satisfies $h_{T(K)}(x) = h_K(T^t x)$, and so $w := z h_K$ is $1$-homogeneous and satisfies:
\[
T^{(0)}_* z \cdot h_{T(K)}(x) = z(T^t x) h(T^t x) = w(T^t x) . 
\]
By Lemma \ref{lem:DST}:
\begin{align*}
& \SS(T^{(0)}_* z \cdot h_{T(K)} , h_{T(K)} , \ldots, h_{T(K)})(T^{(0)} \theta) \\
& = \SS(w \circ T^t,h_K \circ T^t,\ldots, h_K \circ T^t)(T^{(0)} \theta) \\
& = \det(T)^2 \abs{T^{-t} \theta}^{n+1}  \SS(w,h_K,\ldots,h_K)(\theta)  ,
\end{align*}
and in particular:
\[
\SS(h_{T(K)} , \ldots, h_{T(K)})(T^{(0)}\theta) = \det(T)^2 \abs{T^{-t} \theta}^{n+1}  \SS(h_K,h_K,\ldots,h_K)(\theta).
\]
Taking the quotient of the latter two expressions, we verify:
\[
\tilde{L}_{T(K)}(T^{(0)}_* z)(T^{(0)} \theta) = (\tilde{L}_K z)(\theta) .
\]
Since $L = \tilde{L} - \text{Id}$, (\ref{eq:diagram-commutes0}) is established. 
\end{proof}

\subsection{Spectral Minimization Problem and Potential Extremizers}  \label{subsec:extremizers}

We now restrict our discussion to origin-symmetric $K \in \K^2_{+,e}$. Recall that:
\[
\lambda_{1,e}(-L_K) := \min \sigma(-L_K|_{E_{\text{even}} \cap \mathbf{1}^{\perp,L^2(dV_K)}}) ,
\]
and that Theorem \ref{thm:LK} implies that:
\begin{equation} \label{eq:individual}
\lambda_{1,e}(-L_K) > 1 \;\;\; \forall K \in \K^2_{+,e}.
\end{equation}
In addition, since the isometry $T^{(0)}_* : L^2(d\tilde{V}_K) \rightarrow L^2(d\tilde{V}_{T(K)})$ clearly maps $E_{\text{even}} \cap \mathbf{1}^{\perp,L^2(dV_K)}$ onto $E_{\text{even}} \cap \mathbf{1}^{\perp,L^2(dV_{T(K)})}$, it follows by Theorem \ref{thm:LTK} that:
\[
\lambda_{1,e}(-L_{T(K)}) = \lambda_{1,e}(-L_K) \;\;\; \forall K \in K^2_{+,e} \;\; \forall T \in GL_n .
\]
Corollary \ref{cor:LK-SG} therefore translates into:
\begin{cor} \label{cor:LK-SG2}
Given $p < 1$, the validity of the local $p$-BM conjecture (\ref{eq:local-p-BM}) for all $K \in \K^2_{+,e}$ is equivalent to the validity of the following lower bound for the minimization problem over the linearly invariant spectral parameter $\lambda_{1,e}(-L_K)$:
\begin{equation} \label{eq:LK-SG2}
\inf_{K \in \K^{2}_{+,e} / GL_n } \lambda_{1,e}(-L_K) \geq \frac{n-p}{n-1} . 
\end{equation}
\end{cor}

Observe that by F.~John's Theorem \cite[Theorem 4.2.12]{GardnerGeometricTomography2ndEd}, $\K_e / GL_n$ is a compact set of equivalence classes of origin-symmetric convex bodies (with respect to the natural Hausdorff topology $C$), the so-called Banach-Mazur compactum. Also note that by Theorem \ref{thm:LK} (4), $\K^2_{+,e} \ni K \mapsto \lambda_{1,e}(-L_K)$ is continuous in the $C^2$ topology. Unfortunately, $\K^{2}_{+,e} / GL_n$ isn't closed in either of these topologies, and is only 
 a dense subset of the Banach-Mazur compactum. In particular, we do not know how to show from general functional-analytic arguments that the infimum in (\ref{eq:LK-SG2}) is strictly greater than $1$, even though we have the individual estimate (\ref{eq:individual}). However, we will see in the next section that we can verify the validity of (\ref{eq:pBMPsi}) without resorting to compactness arguments for a concrete range of $p < 1$, which translates into the following:
\begin{thm} \label{thm:main-spectral}
There exists a constant $c > 0$ so that:
\[
\inf_{K \in \K^{2}_{+,e} / GL_n } \lambda_{1,e}(-L_K) \geq \frac{n-p_n}{n-1}  > 1 \;\; \text{ where } \;\; p_n := 1 - \frac{c}{n^{3/2}} .
\]
\end{thm}

This provides a positive answer to the qualitative question of whether there is a uniform even spectral-gap for $-L_K$ beyond $1$, and so the only remaining question is the quantitative one -- how big is it? By Corollary \ref{cor:LK-SG2}, the (local) log-BM conjecture ($p=0$ case) predicts it should be $\frac{n}{n-1}$, an a-priori mysterious quantity. Better insight is gained by inspecting several natural candidates $K$ for being a minimizer in (\ref{eq:LK-SG2}). As with essentially all minimization problems over linearly invariant parameters in Convexity Theory, there are three immediate suspects:
\begin{itemize}
\item $K = B_2^n$, the Euclidean unit-ball. Recalling (\ref{eq:LK-explicit}), we immediately see that $L_{B_2^n} = \frac{1}{n-1} \Delta_{S^{n-1}}$, where $\Delta_{S^{n-1}}$ is the Laplace-Beltrami operator on $S^{n-1}$. The spectral decomposition of $\Delta_{S^{n-1}}$ is classical \cite{VilenkinClassicBook,ChavelEigenvalues}, with $k$-th distinct eigenvalue ($k \geq 0$) equal to $k (k+n-2)$, corresponding to the 
eigenspace of spherical harmonics of degree $k$. As expected, spherical harmonics of degree $0$ are constant functions, of degree $1$ are linear functions $\ell_v$, and of degree $2$ are homogeneous quadratic harmonic polynomials (which are in particular even). It follows that for $-L_{B_2^n}$:
\[
\lambda_0 = 0 ~,~ \lambda_1 = \ldots = \lambda_n = \frac{n-1}{n-1} = 1 ~,~ \lambda_{1,e} =\lambda_{n+1} = \frac{2n}{n-1} ,
\]
and we see that we get a much better even spectral-gap (corresponding to $p=-n$ in (\ref{eq:LK-SG2})) than the conjectured lower bound $\frac{n}{n-1}$. So $B_2^n$ is not a minimizer for (\ref{eq:LK-SG2}). 

Applying Proposition \ref{prop:local-equiv} (with $p_0=-n$ and $p=0$) and the invariance under linear transformations, we obtain: 
\begin{thm} \label{thm:Ball}
$\lambda_{1,e}(-L_{B_2^n}) = \frac{2n}{n-1}$; equivalently, the local $(-n)$-BM inequality (\ref{eq:local-p-BM}) holds for $B_2^n$. In particular, there exists a $C^2$-neighborhood $N_{B_2^n}$ of $B_2^n$ in $\K^2_{+,e}$ so that for all $T \in GL_n$,  for all $K_1,K_0 \in T(N_{B_2^n})$, the local log-BM conjecture (\ref{eq:local-log-BM}) holds for $K_0$ and 
\[
V((1-\lambda) \cdot K_0 +_0 \lambda \cdot K_1) \geq V(K_0)^{1-\lambda} V(K_1)^{\lambda} \;\;\; \forall \lambda \in [0,1] . 
\]
\end{thm}
This confirms Theorem \ref{thm:intro-lq} for $q=2$ for all $n \geq 2$. 

\item $K = B_\infty^n = [-1,1]^n$, the unit cube. Note that $B_\infty^n$ is not smooth, so $L_{B_\infty^n}$ is not well-defined. However, defining for any $K \in \K_e$:
\begin{equation} \label{eq:lambda-nonsmooth}
\lambda_{1,e}(K) := \liminf_{\K^2_{+,e} \ni K_i \rightarrow K \text{ in $C$}}  \lambda_{1,e}(-L_{K_i}) ,
\end{equation}
we obtain a lower semi-continuous function on the Banach--Mazur compactum $\K_e / GL_n$, which must attain a minimum. By Theorem \ref{thm:main-spectral}, this minimum is strictly greater than $1$. We will verify in Theorem \ref{thm:lambda1-cube} that:
\[
\lambda_{1,e}(B_\infty^n) = \frac{n}{n-1} .
\]
Consequently, we have the following natural interpretation:
\begin{prop} The validity of the local log-BM conjecture  (\ref{eq:local-log-BM}) for all $K \in \K^2_{+,e}$ is equivalent to the validity of the conjecture that the cube $K = B_\infty^n$ is a minimizer of the linearly invariant even spectral-gap $\lambda_{1,e}(K)$:
\[
\min_{K \in \K_e / GL_n } \lambda_{1,e}(K) = \lambda_{1,e}(B_\infty^n) .
\]
\end{prop}
In our opinion, the latter conjecture is extremely natural, and constitutes the best justification for believing that the local log-BM conjecture is true. 
\item $K = B_1^n$, the unit-ball of $\ell_1^n$. This might be the only natural potential counter-example to the log-BM conjecture, and we presently do not know how to verify the conjecture for it. As before, $B_1^n$ is not smooth so $L_{B_1^n}$ is not well-defined. However, we will verify in Corollary \ref{cor:lambda1-unc} that:
\[
\lambda_{1,\unc}(K) \geq \frac{n}{n-1} \;\;\; \forall K \in \K_{\unc} .
\]
Here $\K_{\unc}$, $\K^2_{+,\unc}$ and $E_{\unc}$ denote the unconditional elements of $\K$, $\K^2_{+}$ and $H^2(S^{n-1})$, respectively, meaning that they are invariant under reflections with respect to the coordinate hyperplanes; for $K \in \K^2_{+,\unc}$ we define:
\[
\lambda_{1,\unc}(-L_K) := \min \sigma(-L_{K}|_{E_{\unc} \cap \mathbf{1}^{\perp}}) ,
\]
and for $K \in \K_{\unc}$ we set:
\begin{equation} \label{eq:lambda-unc}
\lambda_{1,\unc}(K) := \liminf_{\K^2_{+,\unc} \ni K_i \rightarrow K \text{ in $C$}} \lambda_{1,\unc}(-L_{K_i}) .
\end{equation}
In particular, we have $\lambda_{1,\unc}(B_1^n) \geq \frac{n}{n-1}$, which is a good sign. 
\end{itemize}

While the $p$-BM conjecture pertains to the minimization problem (\ref{eq:LK-SG2}), it also makes sense to consider the corresponding \emph{maximization} problem. In view of the above examples, we make the following:

\begin{conj}
\[
\max_{K \in \K^{2}_{+,e} / GL_n } \lambda_{1,e}(-L_K) = \frac{2n}{n-1} ,
\]
with equality for origin-symmetric ellipsoids $K = T(B_2^n)$. 
\end{conj}

\bigskip

\section{Obtaining Estimates via the Reilly Formula} \label{sec:Reilly}

We are finally ready to prove our main results in this work. These are based on an integral formula obtained by twice integrating-by-parts the Bochner--Lichnerowicz--Weitzenb\"{o}ck identity, which in the Riemannian setting is due to Reilly \cite{ReillyOriginalFormula}.
In the description below, we specialize the Reilly formula to our Euclidean setting (see \cite[Theorem 1.1]{KolesnikovEMilmanReillyPart1} for a proof of a more general version, which holds on weighted Riemannian manifolds, and involves an additional curvature term).

We denote by $\nabla$ the Euclidean connection, and by $\Delta$ the Euclidean Laplacian. We denote by $\norm{\nabla^2 u}$ the Hilbert-Schmidt norm of the Euclidean Hessian $\nabla^2 u$. If $\Omega \subset \Real^n$ is a compact set with Lipschitz boundary, we denote by $\S_0(\Omega)$ the class of functions $u$ on $\Omega$ which are in $C^2(\text{int}(\Omega)) \cap C^1(\Omega)$. We shall henceforth assume that $\partial \Omega$ is $C^2$ smooth with outer normal $\nu = \nu_{\partial \Omega}$, and denote by $\S_N(\Omega)$ the elements in $\S_0(\Omega)$ which satisfy $u_\nu  := \scalar{\nabla u,\nu}\in C^1(\partial \Omega)$. Let $\mu = \exp(-V(x)) dx$ denote a measure on $\Omega$ with $V \in C^2(\Omega)$, and denote $\mu_{\partial \Omega} = \exp(-V(x)) d\H^{n-1}|_{\partial \Omega}(x)$. Introduce the following weighted Laplacian, defined by:
\[
L_\mu u := \Delta u - \scalar{\nabla V , \nabla u} \;\;\ \forall u \in C^2(\text{int}(\Omega)) ;
\]
it satisfies the following weighted integration-by-parts property (see \cite[Remark 2.2]{KolesnikovEMilmanReillyPart1}):
\begin{equation} \label{eq:parts}
\int_{\Omega} L_\mu u \; d\mu = \int_{\partial \Omega} u_\nu d\mu_{\partial \Omega} \;\; \forall u \in S_0(\Omega) . 
\end{equation}
As usual, we denote by $\II_{\partial \Omega}$ the second fundamental form of $\partial \Omega \subset \Real^n$, and define its generalized mean curvature at $x \in \partial \Omega$ as:
\[
H_{\partial \Omega,\mu} := tr(\II_{\partial \Omega}) - \scalar{\nabla V , \nu} . 
\]
Finally, $\nabla_{\partial \Omega}$ denotes the induced connection on $\partial \Omega$.

\begin{thm*}[Generalized Reilly Formula] 
For any function $u \in \S_N(\Omega)$: 
\begin{multline}
\label{eq:Reilly}
\int_\Omega (L_\mu u)^2 d\mu = \int_\Omega \norm{\nabla^2 u}^2 d\mu + \int_\Omega \scalar{ \nabla^2 V \; \nabla u, \nabla u} d\mu + \\
\int_{\partial \Omega} H_{\partial \Omega,\mu} u_\nu^2 d\mu_{\partial \Omega} + \int_{\partial \Omega} \scalar{\II_{\partial \Omega}  \;\nabla_{\partial \Omega} u,\nabla_{\partial \Omega} u} d\mu_{\partial \Omega} - 2 \int_{\partial \Omega} \scalar{\nabla_{\partial \Omega} u_\nu, \nabla_{\partial \Omega} u} d\mu_{\partial \Omega} ~.
\end{multline}
\end{thm*}

We will also use the following classical existence and regularity results for linear elliptic PDEs (e.g. \cite[Chapter 8]{GilbargTrudinger}, \cite[Chapter 5]{LiebermanObliqueBook}, \cite[Chapter 3]{LadyEllipticBook}):

\begin{thm*}
Let $f \in C^{\alpha}(\text{int}(\Omega))$ for some $\alpha \in (0,1)$, let $\Psi \in C^{1}(\partial \Omega)$, and assume that:
\begin{equation} \label{eq:compat}
 \int_{\Omega} f d\mu = \int_{\partial \Omega} \Psi d\mu_{\partial \Omega} ~.
\end{equation}
Then there exists a function $u \in C^{2,\alpha}_{loc}(\text{int}(\Omega)) \cap C^{1,\beta}(\Omega)$ for all $\beta \in (0,1)$, which solves the following Poisson equation with Neumann boundary conditions: \begin{equation} \label{eq:Poisson} 
L_\mu u = f ~ \text{in $\text{int}(\Omega)$} ~,~ u_\nu = \Psi ~ \text{on } \partial \Omega . 
\end{equation}
Moreover, $u$ is unique up to an additive constant. 
\end{thm*}
Note that in particular, the function $u$ above is in $\S_N(\Omega)$. By (\ref{eq:parts}), the compatibility condition (\ref{eq:compat}) is also a necessary condition for solving (\ref{eq:Poisson}). 

\subsection{A sufficient condition for confirming the local $p$-BM inequality}

We now derive a sufficient condition for confirming the local $p$-BM inequality (\ref{eq:local-p-BM}) in its equivalent infinitesimal form (\ref{eq:pBMPsiNor}) on $\partial K$. 
Our motivation comes from our previous work \cite{KolesnikovEMilmanReillyPart2}, where we obtained a new proof of the (local, and hence global) Brunn--Minkowski inequality ($p=1$ case), by verifying  (\ref{eq:pBMPsiNor}) directly (for all test functions $\Psi \in C^1(\partial K)$, without any evenness assumption). In fact, our proof in \cite{KolesnikovEMilmanReillyPart2} applies to a general weighted Riemannian manifold satisfying the Curvature-Dimension condition $\text{CD}(0,N)$, yielding a novel interpretation of Minkowski addition in the Riemannian setting. 

Given $K \in \K^2_{+,e}$, let $\mu$ denote the Lebesgue measure $dx$ on $\Omega = K$ (corresponding to $V=0$, $L_\mu = \Delta$ and $H_{\partial K,\mu} = H_{\partial K}$ above). Given $\Psi \in C^1_e(\partial K)$ with $\int_{\partial K} \Psi dx = 0$, the classical $L^2$-method consists of solving for $u \in \S_N(K)$ the Laplace equation:
\[
\Delta u = 0 ~ \text{in $\text{int}(K)$} ~,~ u_\nu = \Psi ~ \text{on } \partial K  ,
\]
(which clearly satisfies the necessary and sufficient compatibility condition (\ref{eq:compat})). The origin-symmetry of $K$ and evenness of $\Psi$ guarantee that $u(-x)$ is also a solution, and so by uniqueness of the solution it follows that $u$ is necessarily even; we denote by $\S_{N,e}(K)$ the even elements of $\S_N(K)$ (and similarly for $\S_{0,e}(K)$). We see that the above procedure yields a bijection between $\Psi \in C^1_e(\partial K)$ and harmonic $u \in \S_{N,e}(K)$, characterized by the property that $u_\nu = \Psi$. 

Now, applying the Reilly formula (\ref{eq:Reilly}) to $u$, using that $\Delta u = 0$, and plugging in the resulting expression for $\int_{\partial K} H_{\partial K} u_\nu^2 dx$ into (\ref{eq:pBMPsiNor}), we obtain:

\begin{thm} \label{thm:sufficient}
Given $K \in \K^2_{+,e}$ and $p < 1$, the local $p$-BM conjecture (\ref{eq:local-p-BM}) for $K$ is equivalent to the assertion that:
\begin{align*} \forall u \in \S_{N,e}(K) \;\;\; &  \Delta u = 0 \text{ in int($K$)} \;\; \Rightarrow \;\; \\
& \int_K \norm{\nabla^2 u}^2 dx \geq (1-p) \int_{\partial K} \frac{u_\nu^2(x)}{\scalar{x,\nu_{\partial K}(x)}} dx - R_K(u) ,
\end{align*}
where:
\begin{align*} R_K(u)  & := \int_{\partial K} \scalar{\II_{\partial K} \nabla_{\partial K} u,  \nabla_{\partial K} u} dx \\
& + \int_{\partial K}\ \scalar{\II_{\partial K}^{-1} \nabla_{\partial K} u_\nu, \nabla_{\partial K} u_\nu} dx  - 2 \int_{\partial K} \scalar{\nabla_{\partial K} u_\nu, \nabla_{\partial K} u} dx  .
\end{align*}
In particular, as $R_K(u) \geq 0$ by Cauchy--Schwarz (since $\II_{\partial K} > 0$), a sufficient condition for the local $p$-BM conjecture to hold for $K$ is that:
\[
\forall u \in \S_{0,e}(K) \;\;\; \Delta u = 0 \text{ in int($K$)} \;\; \Rightarrow \;\; \int_K \norm{\nabla^2 u}^2 dx \geq (1-p) \int_{\partial K} \frac{u_\nu^2(x)}{\scalar{x,\nu_{\partial K}(x)}} dx .
\]
\end{thm}
\begin{rem} \label{rem:RK}
For future reference, we mention the following alternative expression for $R_K(u)$ when $u \in C^2(K)$:
\[
R_K(u) = \int_{\partial K} \scalar{ \II^{-1}_{\partial K} P_{T_{\partial K}} \bigl[ \nabla^2 u \cdot \nu \bigr], P_{T_{\partial K}} \bigl[\nabla^2 u \cdot \nu \bigr]}  dx ,
\]
where $P_{T_{\partial K}}$ denotes projection to the tangent space to $\partial K$. Indeed, this follows by plugging above: 
\[
\nabla_{\partial K} u_\nu = \nabla_{\partial K}  \scalar{\nabla u,\nu} = \II_{\partial K} \nabla_{\partial K} u  +  P_{T_{\partial K}} \bigl[ \nabla^2 u \cdot \nu \bigr] .
\]
\end{rem}

The sufficient condition of Theorem \ref{thm:sufficient} naturally leads us to the following:
\begin{defn*}[$\B(K)$ and $\BNH(K)$]
Given $K \in \K_e$, let $\B(K)$ denote the best constant $B$ in the following \emph{boundary Poincar\'e-type inequality for harmonic functions}:
\[
\forall u \in \S_{0,e}(K) \;\;\; \Delta u = 0 \text{ in int($K$)} \;\; \Rightarrow \int_{\partial K} \frac{u_\nu^2(x)}{\scalar{x,\nu_{\partial K}(x)}} dx  \leq B \int_K \norm{\nabla^2 u}^2 dx .
\]
Without the requirement that $\Delta u = 0$, the above inequality is called a \emph{boundary Poincar\'e-type inequality}, and the best constant $B$ above is denoted by $\BNH(K)$. 
\end{defn*}
\noindent
Note that all expressions above are well-defined without any smoothness or strict convexity assumptions on $\partial K$, since $\nu_{\partial K}(x)$ exists for $\H^{n-1}$-a.e. $x \in \partial K$. We also take this opportunity to introduce:
\begin{defn*}[$\D(K)$]
Given $K \in \K$, let $\D(K)$ denote the best constant in the following \emph{absolute boundary Poincar\'e-type inequality}:
\[
\forall u \in \S_{0}(K) \;\;\; \int_K \vec{\nabla} u\; dx = \vec{0} \;\; \Rightarrow \;\; \int_{\partial K} \frac{\abs{\nabla u(x)}^2}{\scalar{x,\nu_{\partial K}(x)}} dx  \leq \D(K) \int_K \norm{\nabla^2 u}^2 dx .
\]
\end{defn*}
\noindent
Note that the evenness assumption on $u$ from the former definitions has been replaced by a balancing condition in the latter one, and that the $u_\nu^2$ term has been replaced by (the possibly larger) $\abs{\nabla u}^2$ one.
Since $\nabla u$ is odd for any even function $u$, and hence integrates to zero on any origin-symmetric $K$, it immediately follows that:
\[
\B(K) \leq \BNH(K) \leq \D(K) \;\;\; \forall K \in \K_e . 
\]

We will soon see that $\D(K) < \infty$ for all $K \in \K$, so the above definitions are non-trivial.  
It is easy to see that the constants $\B(K)$, $\BNH(K)$ and $\D(K)$ are invariant under homothety $K \mapsto \lambda K$. However, they are no longer invariant under general linear transformations as in Section \ref{sec:LK}. This can be seen from the following:
\begin{example}
For a $2$-dimensional cube $K = [-a,a] \times [-b,b]$, $\B(K) \geq \frac{1}{6} \brac{\frac{b^2}{a^2} + \frac{a^2}{b^2}}$, as witnessed by the even harmonic function $u = x_1 x_2$. Consequently, $\B(K),\BNH(K),\D(K) \rightarrow \infty$ as the aspect-ratio of $K$ grows to infinity, demonstrating the absence of invariance under $GL_n$. 
\end{example}

Consequently, while the validity of the first condition of Theorem \ref{thm:sufficient} is invariant under $GL_n$ (being equivalent to the local $p$-BM conjecture), the validity of the second sufficient condition is not, and requires putting $K$ in a ``good position", i.e. a suitable linear image. This is due to application of the Cauchy--Schwarz inequality when transitioning from the first formulation to the second, which identifies between tangent and cotangent spaces, and thus destroys the natural covariance--contravariance enjoyed by the first formulation. We summarize all of the relevant information we have obtained thus far in the following:

\begin{thm} \label{thm:main-tool}
Given $K \in \K^2_{+,e}$, assume that $\B(T_0(K)) \leq \frac{1}{1-p}$ for some $p < 1$ and $T_0 \in GL_n$. Then the local $p$-BM conjecture (\ref{eq:local-p-BM}) holds for $T(K)$ for all $T \in GL_n$. 
\end{thm}

In the notation of Section \ref{sec:LK}, we equivalently have:
\begin{thm} \label{thm:B2LK}
For all $K \in \K^2_{+,e}$, $\lambda_{1,e}(-L_K) \geq 1 + \sup_{T \in GL_n} \frac{1}{(n-1) \B(T(K))}$. 
\end{thm}
\begin{proof}
This follows immediate from Corollary \ref{cor:LK-SG}, which asserts that $\lambda_{1,e}(-L_K) \geq \frac{n-p}{n-1}$ if and only if the local $p$-BM conjecture (\ref{eq:local-p-BM}) holds for $K$, and the sufficient condition for its validity for $T(K)$ given by Theorem \ref{thm:main-tool}. 
\end{proof}

\subsection{General Estimate on $\D(K)$} \label{subsec:general-DK}

Unfortunately, getting a handle on $\B(K)$ or $\BNH(K)$ directly is quite a formidable task, and it is easier to upper bound the larger $\D(K)$ constant. 
Recall that the Poincar\'e constant $C_{Poin}(K)$ is defined as the best constant in the following inequality:
\[
\forall f \in C^1(\text{int}(K)) \;\;\; \int_K f(x) dx = 0 \;\; \Rightarrow \;\; \int_K f^2(x) dx \leq C_{Poin}^2(K) \int_K \abs{\nabla f}^2 dx . 
\]
Equivalently, $1 / C_{Poin}^2(K)$ is the first positive eigenvalue of the Neumann Laplacian on $K$. 

\begin{thm} \label{thm:DK}
Let $K \in \K$ and assume that $r B_2^n \subset K \subset R B_2^n$. Then:
\[
\D(K) \leq \frac{1}{r^2} \brac{C_{Poin}^2(K)  n + 2 C_{Poin}(K) R} .
\]
\end{thm}
\begin{proof}
Let $u \in \S_0(K)$ be such that $\int_K \vec{\nabla} u\; dx = \vec{0}$. 
Applying Cauchy-Schwarz, we have for any $\lambda > 0$:
\[
\text{div}(\abs{\nabla u}^2 x) = n \abs{\nabla u}^2 + 2 \scalar{\nabla^2 u \cdot \nabla u , x} \leq n \abs{\nabla u}^2 + \lambda \norm{\nabla^2 u}^2 + \frac{1}{\lambda} \abs{\nabla u}^2 \abs{x}^2 . 
\]
Using the assumption that $\scalar{x,\nu_{\partial K}(x)} = h_K(\nu_{\partial K}(x)) \geq r$, integrating by parts, and finally that $K \subset R B_2^n$, we obtain:
\begin{align*}
\int_{\partial K} \frac{\abs{\nabla u}^2}{\scalar{x,\nu_{\partial K}(x)}} dx  & \leq \frac{1}{r^2}  \int_{\partial K} \abs{\nabla u}^2 \scalar{x,\nu_{\partial K}(x)} dx = \frac{1}{r^2} \int_K \text{div}(\abs{\nabla u}^2 x) dx \\
& \leq \brac{\frac{n}{r^2} + \frac{R^2}{r^2} \frac{1}{\lambda}} \int_K \abs{\nabla u}^2 dx + \frac{\lambda}{r^2} \int_K \norm{\nabla^2 u}^2 dx . 
\end{align*}
Since $\int_K u_i dx = 0$ for all $i=1,\ldots,n$, by applying the Poincar\'e inequality to $u_i$ and summing the resulting inequalities, we obtain:
\[
\int_K \abs{\nabla u}^2 dx \leq C_{Poin}^2(K) \int_K \norm{\nabla^2 u}^2 dx .
\]
 It follows that for all $\lambda > 0$:
  \[
\int_{\partial K} \frac{\abs{\nabla u}^2}{\scalar{x,\nu_{\partial K}(x)}} dx \leq  \brac{\brac{\frac{n}{r^2} + \frac{R^2}{r^2} \frac{1}{\lambda}} C^2_{Poin}(K) + \frac{\lambda}{r^2}} \int_K \norm{\nabla^2 u}^2_{HS} dx .
\]
Using the optimal $\lambda = C_{Poin}(K) R$ and recalling the definition of $\D(K)$, the assertion is established. 
\end{proof}

\begin{cor}
For all $K \in \K$, $\D(K) < \infty$. 
\end{cor}
\begin{proof}
Any $K \in \K$ satisfies $r B_2 \subset K \subset R B_2^n$ with some $r,R>0$. By a well-known theorem of Payne--Weinberger \cite{PayneWeinberger}, for any $K \in \K$ we have $C_{Poin}(K) \leq \frac{D}{\pi}$ where $D$ is the diameter of $K$. As $D \leq 2R$, the assertion follows by Theorem \ref{thm:DK}. 
\end{proof}

In order to apply Theorem \ref{thm:main-tool}, we would like to apply our estimate for $\D(K)$ in a good position of $K \in K$. Recall that the isotropic position is defined as an affine image of $K \in \K$ having barycenter at the origin and for which:
\[
\int_K \scalar{x,\theta}^2 dx = \abs{\theta}^2 \;\;\; \forall \theta \in S^{n-1} . 
\]
It is well known \cite{Milman-Pajor-LK} that such a position always exists and is unique up to orthogonal transformations. In this position, we have the following (sharp) estimates on the in and out radii of $K$ \cite{KLS}:
\[
\sqrt{\frac{n+2}{n}} B_2^n \subset K \subset \sqrt{(n+2)n} B_2^n . 
\]
As for the Poincar\'e constant in isotropic position, a bold conjecture of Kannan, Lov\'asz and Simonovits \cite{KLS} predicts that $C_{Poin}(K) \leq C$ for some universal numeric constant $C>1$, independent of the dimension $n$. The current best known estimate on the KLS conjecture has recently been improved by Lee and Vempala \cite{LeeVempala-KLS}, who showed that:
\begin{equation} \label{eq:LV}
C_{Poin}(K) \leq C' \sqrt[4]{n} 
\end{equation}
for all isotropic $K \in \K$. We immediately deduce from Theorem \ref{thm:DK} the following:
\begin{cor} \label{cor:DK}
There exists a universal numeric constant $C > 1$, independent of the dimension $n$, so that for all $K \in \K$ in isotropic position, $\D(K) \leq C n^{3/2}$. Assuming a positive answer to the KLS conjecture, the latter estimate may be improved to $\D(K) \leq C n$. 
\end{cor}

Since any $K \in \K_e$ has a linear isotropic image and satisfies $\B(K) \leq \D(K)$, Theorem \ref{thm:main-tool} in conjunction with Corollary \ref{cor:DK} and Lemma \ref{lem:loc-to-global} immediately yield Theorem \ref{thm:intro-main}, which we re-state as follows:
\begin{thm} \label{thm:main}
The local $p$-BM conjecture (\ref{eq:local-p-BM}) holds for all $K \in \K^2_{+,e}$ and $p \in [1 - \frac{1}{C n^{3/2}},1)$. Equivalently, for all $p$ in this range, 
for all $K_0,K_1 \in \K^2_{+,e}$ so that $(1-\lambda) \cdot K_0 +_p \lambda \cdot K_1 \in \K^2_{+,e}$ for all $\lambda \in [0,1]$, we have:
\[
V((1-\lambda) \cdot K_0 +_p \lambda \cdot K_1) \geq \brac{(1-\lambda) V(K_0)^{\frac{p}{n}} + \lambda V(K_1)^{\frac{p}{n}}}^{\frac{n}{p}} \;\;\; \forall \lambda \in [0,1] . 
\]
\end{thm}

\subsection{Examples}

Of course, using the known estimates on $C_{Poin}(K)$ which improve over the general (\ref{eq:LV}) for various classes of convex bodies $K$ (see e.g. \cite{KlartagUnconditionalVariance,BartheCorderoVariance,HuetSphericallySymmetric,KolesnikovEMilman-HardyKLS, KolesnikovEMilman-GeneralizedOrlicz}), one may obtain an improved estimate for $p$ above. It will be instructive in this work to concentrate on the unit-balls of $\ell_q^n$, denoted $B_q^n$.

\begin{thm}[Sodin, Lata{\l}a--Wojtaszczyk] \label{thm:Poincare}
For all $q \in [1,\infty]$, $C_{Poin}(B_q^n)$ is of the order of $n^{-\frac{1}{q}}$. 
\end{thm}
\begin{proof}
It was shown by S.~Sodin \cite{SodinLpIsoperimetry} for $q \in [1,2]$ and by R.~Lata{\l}a and J.~Wojtaszczyk \cite{LatalaJacobInfConvolution} for $q \in [2,\infty]$ that  if $\lambda_q^n B_q^n$ has volume $1$ then  $\lambda_q^n C_{Poin}(B_q^n) = C_{Poin}(\lambda_q^n B_q^n)$ is of the order of $1$. An easy and well-known computation verifies that $\lambda_q^n$ is of the order of $n^{1/q}$, yielding the claim. 
\end{proof}

\begin{lem} \label{lem:B-lower}
For any $K \in \K_e$, $\D(K) \geq \BNH(K) \geq 1$. 
\end{lem}
\begin{proof}
Testing the even function $u(x) = \frac{\abs{x}^2}{2}$, note that:
\[
\int_{\partial K} \frac{u_\nu^2}{\scalar{x,\nu_{\partial K}(x)}} dx   = \int_{\partial K} \scalar{x,\nu_{\partial K}(x)} dx = \int_K \text{div}(x) dx = n V(K) = \int_K \norm{\nabla^2 u}^2 dx .
\]
It follows by definition that $\D(K) \geq \BNH(K) \geq 1$.
\end{proof}

\begin{lem} \label{lem:D-Bqn}
For any $q \in [1,2]$, $1 \leq \D(B_q^n) \leq C$. For any $q \in [2,\infty]$, $\D(B_q^n) \leq C n^{1-\frac{2}{q}}$. 
\end{lem}
\begin{proof}
The lower estimate is given by the previous lemma. The upper bound follows from the general estimate of Theorem \ref{thm:DK} in combination with Theorem \ref{thm:Poincare} and the obvious estimates $r B_2^n \subset B_q^n \subset R B_2^n$ with $r = n^{\min(0,\frac{1}{2} - \frac{1}{q})}$ and $R = n^{\max(0,\frac{1}{2} - \frac{1}{q})}$.
\end{proof}

For the cube $B_\infty^n$, we can obtain rather tightly matching lower and upper estimates.

\begin{lem} \label{lem:D-Cube}
$\frac{1}{3} n \leq \D(B_\infty^n) \leq \frac{4}{\pi^2} n + \frac{4}{\pi} \sqrt{n}$. 
\end{lem}
\begin{proof}
The upper bound follows from the general estimate of Theorem \ref{thm:DK}, using $B_2^n \subset B_\infty^n \subset \sqrt{n} B_2^n$ and $C_{Poin}^2(B_\infty^n) = \frac{4}{\pi^2}$ \cite{Ledoux-Book}. For the lower bound, consider the function $u(x) = x_1^2/2 \in \S_{0,e}(B_\infty^n)$, for which $u_1(x) = x_1$ and $u_i(x) = 0$ for all $i=2,\ldots,n$. Calculating the contribution on each boundary facet, we have:
\[
\int_{\partial B_\infty^n} \frac{\abs{\nabla u}^2}{\scalar{x,\nu}} dx = 2 \cdot 2^{n-1} + 2 (n-1) 2^{n-2} \int_{-1}^1 x_1^2 dx_1 ,
\]
and clearly:
\[
\int_{B_\infty^n} \norm{\nabla^2 u}^2 dx = 2^n . 
\]
Taking the quotient of these two expressions, we see that $\D(B_\infty^n) \geq 1 + (n-1) /3$. 

\end{proof}

\begin{rem}
Is not hard to improve the constant $\frac{1}{3}$ in Lemma \ref{lem:D-Cube} to $\frac{3}{8}$ by using a function $u(x) = u(x_1)$ so that $u_1(x_1) = x_1 / \eps \vee -1 \wedge +1$ for an appropriate $\eps > 0$. In addition, we see that the conjectural estimate $\D(K) \leq C n$ for isotropic $K$, which by Corollary \ref{cor:DK} would follow from a positive answer to the KLS conjecture, is best possible (up to the value of the constant $C$).
\end{rem}

The above examples demonstrate that the general estimate given by Theorem \ref{thm:DK} is in fact fairly accurate in a variety of situations, and so in order to make further progress on the local $p$-BM conjecture, it is best to work with the $\B(K)$ or $\BNH(K)$ constants. Indeed, we will see in the next sections that $\B(B_2^n) = \frac{2}{n+2} < 1$ and $\B(B_\infty^n) = 1$, which are better by an order of $n$ from the corresponding values of $\D(B_2^n), \D(B_\infty^n)$ estimated above. Furthermore, we will see that when $q \in (2,\infty)$, for $n \geq n_q$ large enough, $\B(B_q^n) < 1$. In view of these results and examples, we make the following:
\begin{conj}
For all $K \in \K_e$, there exists $T_0 \in GL_n$ so that $\B(T_0(K)) \leq 1$. 
\end{conj}
\noindent
By Theorem \ref{thm:main-tool}, a positive answer to the latter conjecture will imply a positive answer to the local log-BM conjecture.

\bigskip

\section{The second Steklov operator and $\B(B_2^n)$} \label{sec:Steklov}

In this section, we obtain an operator-theoretic interpretation of the inequality:
\begin{equation} \label{eq:B-Steklov}
\forall u \in \S_{0,e}(K) \;\;\; \Delta u = 0 \text{ in int($K$)} \;\; \Rightarrow  \int_{\partial K} \frac{u_\nu^2}{\scalar{x,\nu_{\partial K}(x)}} dx \leq \B(K) \int_K \norm{\nabla^2 u}^2 dx ,
\end{equation}
which we will use for calculating $\B(B_2^n)$. It is related to the classical Steklov (or Dirichlet-to-Neumann) 1st order elliptic pseudo-differential operator $S$ \cite{SteklovInequalities,SteklovSurvey}.

\subsection{Second Steklov operator}

Let us assume for simplicity that $\partial K$ is $C^{\infty}$ smooth, and denote by $C^{\infty}_0(\partial K)$ the subspace of smooth functions integrating to zero on $\partial K$. 
The Neumann-to-Dirichlet operator $D$, which is the inverse of $S$ on $C^{\infty}_0(\partial K)$, is the linear operator defined by:
\[
D : C^\infty_0(\partial K) \ni \Psi \mapsto u|_{\partial K} \in C^\infty(\partial K) ,
\]
where $u = u_{\Psi}\in C^{\infty}(K)$ solves:
\[
\Delta u = 0 ~ \text{in $K$} ~,~ u_\nu = \Psi ~ \text{on } \partial K  .
\]
In fact, $D$ may be extended to a compact operator  \cite{SteklovInequalities} on:
\[
L^2_0(\partial K) := \set{ \Psi \in L^2(dx|_{\partial K}) ; \int_{\partial K} \Psi dx = 0 } ,
\]
(and moreover to the Sobolev space $H^{-1/2}_0(\partial K)$, but we will not require this here). 

Note that $D$ is self-adjoint and positive semi-definite on $L^2_0(\partial K)$, since for all $\Psi,\Phi \in C^\infty_0(\partial K)$, denoting $v = u_{\Psi}$ and $w = u_{\Phi}$,  we have (integrating by parts and using that $\Delta w =0$):
\[
\int_{\partial K} (D \Psi) \Phi dx = \int_{\partial K} v w_\nu dx = \int_K \text{div}(v \nabla w) dx = \int_K \scalar{\nabla v,\nabla w} dx .
\]
By analogy, we introduce the second Steklov operator $S_2$, by requiring that:
\begin{equation} \label{eq:D2-require}
\int_{\partial K} (S_2 \Psi) \Phi dx = \int_K \scalar{\nabla^2 v,\nabla^2 w} dx .
\end{equation}
Indeed, on $C^\infty_0(\partial K)$, $S_2 \Psi$ has the following explicit description:
\[
S_2 \Psi := -\Delta_{\partial K}(D \Psi) - D (\Delta_{\partial K} \Psi) - H_{\partial K} \Psi + D \nabla_{\partial K} \cdot \II_{\partial K} \nabla_{\partial K} D \Psi ,
\]
where of course $\nabla_{\partial K} \cdot$ denotes the divergence operator on $\partial K$. To see this, 
denote $v = u_{\Psi}$ (so that $v_\nu = \Psi$), integrate by parts on $\partial K$, use the self-adjointness of $D$, and finally apply the Reilly formula (\ref{eq:Reilly}), to obtain:
\begin{align*} & \int_{\partial K} (S_2 \Psi) \Psi dx  =  - \int_{\partial K} \scalar{\Delta_{\partial K} v , v_\nu} dx - \int_{\partial K} \scalar{D \Delta_{\partial K} v_\nu, v_\nu} 
- \int_{\partial K} H_{\partial K} v_\nu^2 dx \\
& + \int_{\partial K} (D \nabla_{\partial K} \cdot \II_{\partial K} \nabla_{\partial K} v) v_\nu dx = 2 \int_{\partial K}  \scalar{\nabla_{\partial K} v, \nabla_{\partial K} v_\nu} - \int_{\partial K} H_{\partial K} v_\nu^2 dx \\
& - \int_{\partial K} \scalar{\II_{\partial K} \nabla_{\partial K} v, \nabla_{\partial K} v} dx =  \int_K \norm{\nabla^2 v}^2 dx ,
\end{align*}
and so (\ref{eq:D2-require}) follows by polarization. In particular, (\ref{eq:D2-require}) implies that $S_2$ is symmetric and positive semi-definite on $L^2_0(\partial K)$, and hence admits a Friedrichs self-adjoint extension. Note that as $S$ is of order 1, $D$ is of order $-1$, and hence $S_2$ is also of order $1$, like $S$, explaining our nomenclature. 

\medskip

Recalling (\ref{eq:B-Steklov}), we see that $\B(K)$ for $K \in \K^\infty_e$ is the best constant in the following inequality for the second Steklov operator:
\[
\forall \Psi \in C^{\infty}_{0,e}(\partial K) \;\;\; \int_{\partial K} (S_2 \Psi) \Psi dx \geq \frac{1}{\B(K)} \int_{\partial K} \frac{\Psi^2}{\scalar{x,\nu_{\partial K}(x)}} dx .
\]
In this sense, we can think of the sufficient condition of Theorem \ref{thm:sufficient} as a 1st order relaxation (via the second Steklov operator) of the original 2nd order spectral problem (for the Hilbert--Brunn--Minkowski operator). 

\subsection{Computing $\B(B_2^n)$}

When $K = B_2^n$, $\B(B_2^n)$ is the best constant in:
\[
\forall \Psi \in C^{\infty}_{0,e}(S^{n-1}) \;\;\;  \int_{S^{n-1}} (S_2 \Psi) \Psi d\theta \geq \frac{1}{\B(B_2^n)} \int_{S^{n-1}} \Psi^2 d\theta ,
\]
and so $\B(B_2^n)$ is the reciprocal of the first eigenvalue of $S_2$ corresponding to an even eigenfunction in $C_0^\infty(S^{n-1})$. As $H_{S^{n-1}} \equiv n-1$ and $\II_{S^{n-1}} = \delta_{S^{n-1}}$, we see that:
\begin{equation} \label{eq:S2-Ball-rep}
S_2 = -\Delta_{S^{n-1}} D - D \Delta_{S^{n-1}} - (n-1) \text{Id} + D \Delta_{S^{n-1}} D .
\end{equation}
As both operators $\Delta_{S^{n-1}}$ and $D$ clearly intertwine the natural action of $SO(n)$ on $L^2_0(\partial K)$, so does $S_2$. It follows by Schur's lemma \cite{VilenkinClassicBook} that the eigenspaces of $S_2$ are given by $H_k$, the subspace of degree $k$ spherical harmonics on $S^{n-1}$, for $k \geq 1$ ($k=0$ is excluded since we are in $C_0^\infty(S^{n-1})$). For $h \in H_k$ it is well known \cite{VilenkinClassicBook,ChavelEigenvalues} that $-\Delta_{S^{n-1}} h = k (k + n -2) h$. In addition, $h$ is already the restriction to $S^{n-1}$ of a harmonic homogeneous polynomial of degree $k$ on $B_2^n$, which we continue to denote by $h$; it follows by Euler's identity that $h_\nu = k h$, and so by definition $D h = \frac{1}{k} h$. Consequently, (\ref{eq:S2-Ball-rep}) yields a complete description of the spectral decomposition of $S_2$:
\[
S_2|_{H_k} = \brac{2 \frac{k (k+n-2)}{k} - (n-1) - \frac{k (k+n-2)}{k^2}} \text{Id}|_{H_k} .
\]
It follows that the first even eigenfunction of $S_2$ lies in $H_2$ (quadratic harmonic polynomials), with corresponding eigenvalue $1 + \frac{n}{2}$. We thus obtain:
\begin{thm}
$\B(B_2^n) = \frac{2}{n+2} < 1$. 
\end{thm}
By Theorem \ref{thm:main-tool}, this corresponds to sufficient condition for confirming the local $p$-BM conjecture with $p=-\frac{n}{2}$. Note that this is worse by a factor of $2$ than the equivalent characterization from Subsection \ref{subsec:extremizers} using $\lambda_{1,e}(-L_{B_2^n}) = \frac{2n}{n-1}$, which corresponds to $p=-n$. This means that the Cauchy-Schwarz inequality we have employed in Theorem \ref{thm:sufficient} is indeed wasteful for $B_2^n$, but still we obtain a good enough condition to reaffirm the local log-BM conjecture (case $p=0$) for $B_2^n$ (and its $C^2$-perturbations), as $\B(B_2^n) < 1$.

\bigskip

\section{Unconditional Convex Bodies and the Cube} \label{sec:unc}

It is also a challenging task to compute $\BNH(K)$ even for some concrete convex bodies. In this section, we precisely compute the variant $\BNH_{\text{uncond}}(K)$, when only testing unconditional functions on an unconditional convex body $K$. In the case of the cube $B_\infty^n$, we also manage to precisely compute $\BNH(B_\infty^n)$ and consequently $\B(B_\infty^n)$, in precise agreement with the worst-possible predicted value by the local log-BM conjecture. 

\subsection{Unconditional Convex Bodies}

Let $\K_{\unc}$ denote the class of unconditional convex bodies, namely convex bodies which are invariant under reflections with respect to the coordinate hyperplanes $\set{x_i=0}$, $i=1,\ldots,n$. We denote by $u \in \S_{0,\unc}(K)$ the elements of $\S_0(K)$ which are invariant under the aforementioned reflections. 
\begin{defn*}[$\BNH_{\unc}(K)$]
Given $K \in \K_{\unc}$, let $\BNH_{\unc}(K)$ denote the best constant in the following \emph{boundary Poincar\'e-type inequality for unconditional functions}:
\[
\forall u \in \S_{0,\unc}(K) \;\;\; \int_{\partial K} \frac{u_\nu^2(x)}{\scalar{x,\nu_{\partial K}(x)}} dx  \leq \BNH_{\unc}(K) \int_K \norm{\nabla^2 u}^2 dx .
\]
\end{defn*}

Observe that $\BNH_{\unc}(K) \geq 1$, by testing the unconditional function $u(x) = \frac{\abs{x}^2}{2}$ as in Lemma \ref{lem:B-lower}. 
Note that when $K \in \K_{\unc}$, it is easy to see that $\nu_i x_i \geq 0$ for all $i=1,\ldots,n$ and $\H^{n-1}$-a.e. $x \in \partial K$. The Cauchy--Schwarz inequality immediately yields:

\begin{lem}
For all $K \in \K_{\unc}$ and $u \in C^1(K)$, we have for $\H^{n-1}$-a.e. $x \in \partial K$:
\[
\frac{u_\nu^2}{\scalar{x,\nu}} = \frac{(\sum_{i=1}^n u_i \nu_i)^2}{\sum_{i=1}^n x_i \nu_i} \leq \sum_{i=1}^n \frac{u_i^2}{x_i} \nu_i  .
\]
\end{lem}

We also have the following lemma, inspired by the method in our previous work \cite{KolesnikovEMilman-HardyKLS}:
\begin{lem} \label{lem:on-diagonal}
For any $K \in \K$, $u \in \S_{0}(K)$ and $i=1,\ldots,n$ so that $u_i \equiv 0$ on $K \cap \set{x_i = 0}$, we have:
\[
\int_{\partial K} \frac{u_i^2}{x_i} \nu_i  = \int_{K} \brac{2 \frac{u_i}{x_i} u_{ii} - \frac{u_i^2}{x_i^2}}  dx \leq \int_{K} u_{ii}^2 dx .
\]
\end{lem}
\begin{proof}
The first identity follows by integration-by-parts on $K_+ := K \cap \set{x_i \geq 0}$ and $K_{-} := K \cap \set{x_0 \leq 0}$ separately. The assumption that $u_i \equiv 0$ on $K \cap \set{x_i = 0}$ and $u \in \S_0(K)$ are crucial here, to ensure that $\lim_{x \rightarrow x^0} \frac{u_i(x)}{x_i} = u_{ii}(x^0)$ if $x^0 \in \text{int}(K) \cap \set{x_i = 0}$. Defining:
\[
\xi(x) := \begin{cases} \frac{u_i^2}{x_i} e_i & K \cap \set{x_i > 0} \\ 0 & K \cap \set{x_i = 0} \end{cases} ,
\]
it follows that the vector field $\xi$ is in $C^1(\text{int}(K_+)) \cap C(K_+)$, and hence integrating by parts:
\[
\int_{\partial K \cap \set{x_i \geq 0}} \scalar{\xi,\nu_{\partial K}} dx =  \int_{\partial K_+} \scalar{\xi,\nu_{\partial K_+}} dx = \int_{K_+} \text{div}(\xi) dx .
\]
Repeating the argument for $K_-$ and summing, the first equality follows. 
The second inequality follows by applying the Cauchy--Schwarz (Geometric--Arithmetic mean) inequality $2ab \leq a^2 + b^2$. 
\end{proof}

Applying the previous two lemmas and summing over $i=1,\ldots,n$, we deduce:
\begin{thm} \label{thm:unc}
Let $K \in \K_{\unc}$. Then for all $u \in \S_0(K)$ such that $u_i \equiv 0$ on $K \cap \set{x_i = 0}$ for all $i=1,\ldots,n$, we have:
\begin{equation} \label{eq:unc}
\int_{\partial K} \frac{u_\nu^2}{\scalar{x,\nu}} dx \leq \int_K \sum_{i=1}^n u_{ii}^2 dx \leq \int_K \norm{\nabla^2 u}^2 dx .
\end{equation}
In particular, this holds for all $u \in \S_{0,\unc}(K)$, and therefore $\BNH_{\unc}(K) = 1$. 
\end{thm}

\begin{rem}
Note that the latter theorem does not follow from Saroglou's global affirmation of the log-BM conjecture for unconditional convex bodies \cite{Saroglou-logBM1}. When $K,L \in \K_{\unc}$,  all relevant test functions $\Psi$ on $\partial K$ (and thus the harmonic $u$ on $K$) for the local log-BM conjecture will indeed be unconditional. However, Theorem \ref{thm:unc} confirms the \emph{stronger} sufficient condition given in Theorem \ref{thm:main-tool}, and moreover, \emph{without the requirement that $u$ be harmonic}. 
\end{rem}

\subsection{The Cube}

In the case of the cube, we can use the off-diagonal elements of $D^2 u$ to control the non-unconditional part of a general test function $u$: 

\begin{thm} \label{thm:B-Cube}
$\B(B_\infty^n) = \BNH(B_\infty^n) = 1$.
\end{thm}
\begin{proof}
To show that $\BNH(B_\infty^n) \leq 1$, we need to show for all $u \in \S_{0,e}(B_\infty^n)$ that:
\[
\int_{\partial B_\infty^n} u_\nu^2 dx \leq \int_{B_\infty^n} \norm{\nabla^2 u}^2 dx .
\]
To this end, it is enough to establish for all $i=1,\ldots,n$ that:
\begin{equation} \label{eq:off-diagonal}
\int_{\partial B_\infty^n} u_i^2 \abs{\nu_i} dx \leq \int_{B_\infty^n} \abs{\nabla u_i}^2 dx .
\end{equation}
Without loss of generality, we may assume that $i=1$. For $x = (x_1,\ldots,x_n) \in \Real^n$, set $y = (x_2,\ldots,x_n)$, and write $u = u^+ + u^-$, where $u^+$ is even w.r.t. both $x_1$ and $y$ ($u^+(-x_1,y) = u^+(x_1,-y) = u^+(x_1,y)$) and $u^-$ is odd w.r.t. both $x_1$ and $y$ ($u^-(-x_1,y) = u^-(x_1,-y) = -u^-(x_1,y)$, namely:
\[
u^+(x) := \frac{1}{2} (u(x_1,y) + u(x_1,-y)) ~,~  u^-(x) := \frac{1}{2} (u(x_1,y) - u(x_1,-y)) 
\]
(recall that $u$ was assumed even). 
It is enough to verify (\ref{eq:off-diagonal}) for $u^+$ and $u^-$ separately, since it is easy to see that the behavior under reflections and the unconditionality of the cube guarantee that $\int_{\partial B_\infty^n} u^+_1 u^-_1 \abs{\nu_1} dx = 0$ and $\int_{B_\infty^n} \scalar{\nabla u^+_1, \nabla u^-_1} dx = 0$. 

Note that $u^+_1$ is odd w.r.t. $x_1$ and hence $u^+_1 \equiv 0$ on $B_\infty^n \cap \set{x_1 = 0}$. It follows by Lemma \ref{lem:on-diagonal} that:
\[
\int_{\partial B_\infty^n} (u^+_1)^2 \abs{\nu_1} dx = \int_{\partial B_\infty^n} \frac{(u^+_1)^2}{x_1} \nu_1 dx \leq \int_{B_\infty^n} (u^+_{11})^2 dx \leq \int_{B_\infty^n} \abs{\nabla u^+_1}^2 dx .
\]
As for $u^-_1$, which is even w.r.t. $x_1$, write:
\[
(u^-_1)^2(1,y) = \int_0^{1} \frac{\partial}{\partial x_1} (x_1 (u^-_1)^2(x_1,y) ) dx_1 = \int_0^1 (2 x_1 u^-_1(x_1,y) u^-_{11}(x_1,y) +  (u^-_1)^2(x_1,y)) dx_1 . 
\]
Using the evenness of the above integrand in $x_1$, we obtain:
\begin{align}
\nonumber \int_{\partial B_\infty^n} (u^-_1)^2 \abs{\nu_1} dx & = \int_{B_\infty^{n-1}} \int_{-1}^1  (2 x_1 u^-_1(x_1,y) u^-_{11}(x_1,y) +  (u^-_1)^2(x_1,y)) dx_1 dy \\
\label{eq:cube-terms} & \leq \int_{B_\infty^n} (x_1^2 (u^{-}_{11})^2(x) + 2 (u^-_1)^2(x)) dx ,
\end{align}
where the last inequality follows by completing the square. The first term on the right is trivially controlled by:
\[
\int_{B_\infty^n} x_1^2 (u^{-}_{11})^2(x) dx \leq \int_{B_\infty^n} (u^{-}_{11})^2(x) dx  .
\]
For the second term, note the $u_1^-(x_1,y)$ is odd w.r.t. $y$, and hence integrates to zero on each $(n-1)$-dimensional slice $B_t := B_\infty^n \cap \set{x_1 = t}$. Applying the Poincar\'e inequality on $B_t$, and recalling the well known fact \cite{Ledoux-Book} that $C^2_{Poin}(B_\infty^k) = \frac{4}{\pi^2}$ for any $k \geq 1$, it follows that:
\[
2 \int_{B_\infty^n} (u^-_1)^2(x) dx = 2 \int_{-1}^1 \int_{B_{x_1}} (u^-_1)^2(x_1,y) dy dx_1 \leq \frac{8}{\pi^2} \int_{-1}^1 \int_{B_{x_1}} \abs{\nabla_y u^-_{1}}^2 dy dx_1 . 
\]
Since $\frac{8}{\pi^2} <  1$, combining the contributions of the above two terms to (\ref{eq:cube-terms}), we obtain:
\[
\int_{\partial B_\infty^n} (u^-_1)^2 \abs{\nu_1} dx \leq \int_{B_\infty^n} \abs{\nabla u^{-}_1}^2 dx ,
\]
as required. 

This concludes the proof that $\BNH(B_\infty^n) \leq 1$. Consequently, we also have $\B(B_\infty^n) \leq 1$. It remains to note that the constant $1$ is sharp in both cases, as witnessed by the even harmonic function $u(x) = \frac{x_1^2}{2} - \frac{x_2^2}{2}$, and therefore $\B(B_\infty^n) = \BNH(B_\infty^n) = 1$. 
\end{proof}

\begin{rem}
Note that the value $\B(B_\infty^n) = 1$ is in precise accordance with the threshold required in Theorem \ref{thm:main-tool} for confirming the local log-BM conjecture in the case of smooth bodies in $\K^2_{+,e}$. In this formal sense, the cube can be thought as satisfying the local log-BM conjecture. We will give this a more rigorous sense in Section \ref{sec:continuity}. 
\end{rem}

\bigskip

\section{Local log-Brunn--Minkowski via the Reilly Formula} \label{sec:Reilly-logconvex}

In this section, we apply the generalized Reilly formula to a measure on $K$ with log-convex (not log-concave!) density, specifically constructed for verifying the local log-BM conjecture for certain classes of convex bodies. 

\subsection{Sufficient condition for verifying local log-Brunn--Minkowski} 

Recall by Proposition \ref{prop:p-BM-K} and Remark \ref{rem:p-BM-K} that the validity of the local log-BM conjecture (\ref{eq:local-log-BM}) for $K \in \K^2_{+,e}$ is equivalent to the validity of the following assertion:
\begin{align}
\nonumber & \forall \Psi \in C^1_e(\partial K) \;\;\; \int_{\partial K} \Psi(x) dx = 0 \;\; \Rightarrow \\
 \label{eq:log-desired} &  \int_{\partial K}\ \scalar{\II_{\partial K}^{-1} \nabla_{\partial K} \Psi, \nabla_{\partial K} \Psi} dx  \geq \int_{\partial K} H_{\partial K}(x) \Psi^2(x) dx + \int_{\partial K} \frac{\Psi^2(x)}{\scalar{x,\nu_{\partial K}(x)}} dx .
\end{align}

Given $K \in \K_e$, the associated norm $\norm{\cdot}_K$ is defined by $\norm{x}_K = \min \set{t > 0 ; x \in t K}$. 
Let  $w : [0,1] \rightarrow \Real$ denote a $C^2$ function with $w'(0) = 0$ and $w'(1) = 1$.
\[
W(x) := w(\norm{x}_K) ~,~ \mu := \exp(W(x)) dx|_K .
\]
Note that $w$ cannot be concave, and typically will be chosen to be convex, so that $\mu$ is log-convex (and not log-concave). 
Assuming $K \in \K^2_{+,e}$ and abbreviating $\norm{x} = \norm{x}_K$ and $\nu = \nu_{\partial K}$, observe that on $\partial K$:
\[
\scalar{\nabla W , \nu} = w'(1) \scalar{\nabla \norm{x},\nu} = \abs{\nabla \norm{x}} = \frac{\norm{x}}{\scalar{x,\nu}} = \frac{1}{\scalar{x,\nu}} ,
\]
and hence:
\[
H_{\partial K,\mu} = H_{\partial K} + \scalar{\nabla W,\nu} = H_{\partial K} + \frac{1}{\scalar{x,\nu}} . 
\]
Also note that:
\[
 d\mu_{\partial K} = e^{w(1)} d\H^{n-1}|_{\partial K} . 
\]

Given $\Psi \in C^1_e(\partial K)$ with $\int_{\partial K} \Psi dx = 0$, since also $\int_{\partial K} \Psi d\mu_{\partial K} = 0$, 
we may solve for $u \in \S_{N,e}(K)$ the Laplace equation:
\begin{equation} \label{eq:Laplacemu}
L_\mu u = 0 ~ \text{in $\text{int}(K)$} ~,~ u_\nu = \Psi ~ \text{on } \partial K . 
\end{equation}
Applying the Reilly formula to $u$ on $K$ equipped with the measure $\mu$, we have:
\begin{align*}
 0 = & \int_K (L_\mu u)^2 d\mu = \int_K \norm{\nabla^2 u}^2 d\mu - \int_K \scalar{ \nabla^2 W \; \nabla u, \nabla u} d\mu + \\
 & e^{w(1)} \left ( \int_{\partial K} \brac{H_{\partial K} + \frac{1}{\scalar{x,\nu}}} u_\nu^2 dx + \int_{\partial K} \scalar{\II_{\partial K}  \;\nabla_{\partial K} u,\nabla_{\partial K} u} dx  \right . \\
& \;\;\;\;\;\;\; \left . - 2 \int_{\partial K} \scalar{\nabla_{\partial K} u_\nu, \nabla_{\partial K} u} dx \right ) ~.
\end{align*}
Using $\II_{\partial K} > 0$ and applying the Cauchy--Schwarz inequality as in Theorem \ref{thm:sufficient}, we deduce (recalling that $u_\nu = \Psi$):
\begin{align}
\label{eq:logconvexReilly} & \int_{K} \scalar{\nabla^2 W \nabla u , \nabla u} e^W dx \geq \int_K \norm{\nabla^2 u}^2 e^W dx \\
\nonumber & + e^{w(1)} \brac{ \int_{\partial K} \brac{H_{\partial K} + \frac{1}{\scalar{x,\nu}}} \Psi^2 dx - \int_{\partial K} \scalar{\II_{\partial K}^{-1} \nabla_{\partial K} \Psi, \nabla_{\partial K} \Psi} dx } .
\end{align}
Comparing this with our desired inequality (\ref{eq:log-desired}), we deduce:

\begin{thm} \label{thm:dual-BL}
Let $w : [0,1] \rightarrow \Real$ denote a $C^2$ function with $w'(0) = 0$ and $w'(1) = 1$. Given $K \in \K^2_{+,e}$ denote $W(x) = w(\norm{x}_K)$, and
assume that:
\begin{align*}
\forall u \in \S_{N,e}(K) \; & \; \Delta u + \scalar{\nabla W,\nabla u} = 0 \text{ in int($K$)} \;\; \Rightarrow \;\; \\
& \int_{K} \scalar{\nabla^2 W \nabla u , \nabla u} e^W dx  \leq \int_K \norm{\nabla^2 u}^2 e^W dx .
\end{align*}
Then the local log-BM conjecture (\ref{eq:local-log-BM}) holds for $T(K)$ for all $T \in GL_n$. 
\end{thm}
\noindent
As usual, the application of the Cauchy--Schwarz inequality destroyed the linear invariance of the validity of the above sufficient condition, in contrast with the invariance of the local log-BM conjecture. 

\begin{rem} \label{rem:dual-BL}
Observe that when $w$ is convex, the sufficient condition in Theorem \ref{thm:dual-BL} is some sort of dual log-convex formulation of the classical Brascamp--Lieb inequality \cite{BrascampLiebPLandLambda1} (which in itself is known to be equivalent to the Pr\'ekopa--Leindler, and hence Brunn--Minkowski, inequality). 
\end{rem}

We can also obtain the following version of Theorem \ref{thm:dual-BL} for perturbations:
\begin{thm} \label{thm:dual-BL-pert}
With the same assumptions as in Theorem \ref{thm:dual-BL}, assume in addition the existence of $\eps > 0$ so that:
\begin{align}
\nonumber \forall u \in \S_{N,e}(K) \; & \; \Delta u + \scalar{\nabla W,\nabla u} = 0 \text{ in int($K$)} \;\;\; \Rightarrow\\
\label{eq:pert-assump}  &  \int_{K} \scalar{\nabla^2 W \nabla u , \nabla u} e^W dx  \leq (1-\eps) \int_K \norm{\nabla^2 u}^2 e^W dx .
\end{align}
Then there exists a $C^2$ neighborhood $N_{K}$ of $K$ in $\K^2_{+,e}$, so that the local log-BM conjecture (\ref{eq:local-log-BM}) holds for $T(K')$ for all $K' \in N_{K}$ and $T \in GL_n$. Equivalently, for all $T \in GL_n$ and $K_1,K_0 \in T(N_{K})$:
\[
V((1-\lambda) \cdot K_0 +_0 \lambda \cdot K_1) \geq V(K_0)^{1-\lambda} V(K_1)^{\lambda} \;\;\; \forall \lambda \in [0,1] . 
\]
\end{thm}
\begin{proof}
Plugging (\ref{eq:pert-assump}) into (\ref{eq:logconvexReilly}), we obtain for all $\Psi \in C^1_e(\partial K)$ with $\int_{\partial K} \Psi dx =0$:
\begin{equation} \label{eq:ellq-temp}
\int_{\partial K} \brac{H_{\partial K} + \frac{1}{\scalar{x,\nu}}} \Psi^2 dx - \int_{\partial K} \scalar{\II_{\partial K}^{-1} \nabla_{\partial K} \Psi, \nabla_{\partial K} \Psi} dx \leq -\delta \int_{K} \norm{\nabla^2 u}^2 dx  
\end{equation}
with $\delta = \eps e^{\min w - w(1)}$, where $u$ solves (\ref{eq:Laplacemu}). By definition:
\[
\int_{K} \norm{\nabla^2 u}^2 dx \geq \frac{1}{\BNH(K)} \int_{\partial K} \frac{u_\nu^2}{\scalar{x,\nu}} dx \;\;\; \forall u \in \S_{0,e}(K) ,
\]
and by Theorem \ref{thm:DK}, $\BNH(K) \leq \D(K) < \infty$. Consequently, we deduce from (\ref{eq:ellq-temp}) that:
\[
\int_{K} \scalar{\II_{\partial K}^{-1} \nabla_{\partial K} \Psi, \nabla_{\partial K} \Psi} dx  \geq \int_{\partial K} H_{\partial K} \Psi^2 dx + 
\brac{1 + \frac{\delta}{\D(K)}} \int_{\partial K} \frac{\Psi^2}{\scalar{x,\nu}} dx.
\]
In other words, (\ref{eq:pBMPsiNor}) holds with $p_K := - \frac{\delta}{\D(K)}$, and so the local $p_K$-BM conjecture (\ref{eq:local-p-BM}) holds for $K$. 
The assertion then follows by Proposition \ref{prop:local-equiv} (with $p_0 = p_K < 0$ and $p=0$) and the invariance under linear images.

\end{proof}

\begin{cor} \label{cor:dual-BL-pert}
The assumption and hence conclusion of Theorem \ref{thm:dual-BL-pert} hold if:
\[
Q_{K,w} := \max_{x \in K} \norm{\nabla^2 W(x)}_{op} e^{\max w - \min w} C_{Poin}^2(K) < 1 . 
\]
\end{cor}
\begin{proof}
For any $u \in \S_{N,e}(B_q^n)$:
\[
\int_{K} \scalar{\nabla^2 W \nabla u , \nabla u} e^W dx \leq \max_{x \in K} \norm{\nabla^2 W(x)}_{op} e^{\max w} \int_K \sum_{i=1}^n u_i^2(x) dx . 
\]
Since $u$ is even, $u_i$ is odd, and hence integrates to zero on $K$. Applying the Poincar\'e inequality on $K$ for each $u_i$ and summing, we obtain:
\[
\int_{K} \scalar{\nabla^2 W \nabla u , \nabla u} e^W dx \leq \max_{x \in K} \norm{\nabla^2 W(x)}_{op} e^{\max w - \min w} C_{Poin}^2(K) \int_K \norm{\nabla^2 u}^2 e^W dx .
\]
The assertion follows from Theorem \ref{thm:dual-BL-pert}. 
\end{proof}

\subsection{An alternative derivation via estimating $\B(K)$}

The approach of the previous subsection has the advantage of uncovering a certain duality between the sufficient condition of Theorem \ref{thm:dual-BL} and the Brascamp--Lieb inequality (see Remark \ref{rem:dual-BL}). In this subsection, we provide an alternative simpler derivation of an estimate very similar to that of Corollary \ref{cor:dual-BL-pert}, which is devoid of the former insight. On the other hand, it has the advantage of providing an upper estimate on $\B(K)$, so that even when the latter is strictly larger than $1$, Theorem \ref{thm:main-tool} may be used to deduce the local $p$-BM conjecture for $K$ for some $p \in (0,1)$. In addition, we do not need to assume that $K \in \K^2_+$.

\begin{thm} \label{thm:B-estimate}
Let $w : [0,1] \rightarrow \Real$ denote a $C^2$ function with $w'(0) = 0$ and $\max_{t \in [0,1]} \abs{w'(t)} = w'(1)= 1$. Given $K \in \K_e$ so that $\norm{\cdot}_K \in C^2(S^{n-1})$, denote $W(x) = w(\norm{x}_K)$, and assume that $K \supset r B_2^n$. Then:
\[
\B(K) \leq \frac{C_{Poin}(K)}{r} + C_{Poin}^2(K) \max_{x \in K} \norm{\nabla^2 W(x)}_{op} .
\]
\end{thm}
\begin{proof}
Let $u \in \S_{0,e}(K)$ be harmonic in $\text{int}(K)$. As usual:
\[
\nabla W(x) = w'(\norm{x}) \nabla \norm{x} = w'(\norm{x}) \frac{\norm{x}}{\scalar{x,\nu(x/\norm{x})}} \nu(x / \norm{x}) \;\; \forall  x\in K ,
\]
so $\nabla W|_{\partial K} = \frac{1}{\scalar{x,\nu}} \nu$ and $\abs{\nabla W} \leq \frac{\max_{t \in [0,1]} \abs{w'(t)}}{\max_{\nu \in S^{n-1}} h_K(\nu)} \leq \frac{1}{r}$ on $K$. 
Integrating by parts and utilizing the harmonicity of $u$:
\begin{align*}
 &\int_{\partial K} \frac{u^2_\nu}{\scalar{x,\nu}} dx = \int_{\partial K} u_{\nu} \scalar{\nabla u, \nabla W} dx =  \int_{K} {\rm div} \Bigl(  {\nabla u \langle \nabla u, \nabla W \rangle }  \Bigr) dx \\
 & = \int_{K} \Bigl( \langle \nabla^2 u \nabla u, \nabla W \rangle + \langle \nabla^2 W \nabla u, \nabla u \rangle  \Bigr)  dx .
\end{align*}
Applying Cauchy--Schwarz and the usual Poincar\'e inequality on each $u_i$, we have for any $\lambda > 0$:
\begin{align*}
& \leq \frac{\lambda}{2} \int_K \norm{\nabla^2 u}^2 dx + \frac{1}{2 \lambda} \int_K \abs{\nabla u}^2 \abs{\nabla W}^2 dx +  \max_{x \in K} \norm{\nabla^2 W(x)}_{op}  \int_K \abs{\nabla u}^2  dx  \\
&\leq \frac{\lambda}{2} \int_K \norm{\nabla^2 u}^2 dx + \brac{\frac{1}{2 \lambda r^2} +\max_{x \in K} \norm{\nabla^2 W(x)}_{op} } \int_K \abs{\nabla u}^2  dx \\
&\leq \brac{\frac{\lambda}{2} + \brac{\frac{1}{2 \lambda r^2} +\max_{x \in K} \norm{\nabla^2 W(x)}_{op} } C_{Poin}^2(K) } \int_K \norm{\nabla^2 u}^2 dx  .
\end{align*}
Setting $\lambda = \frac{C_{Poin}(K)}{r}$, the assertion follows. 
\end{proof}

It is particularly convenient to apply Theorem \ref{thm:B-estimate} to $B_q^n$, the unit-balls of $\ell_q^n$, for $q \in (2,\infty)$.
\begin{thm} \label{thm:BBqn}
For all $q \in (2,\infty)$, $\B(B_q^n) \leq C (n^{-1/q} + q n^{-2/q})$.
\end{thm}
\begin{proof}
Set $w(t) = \frac{1}{q} t^q$ and $W(x) = w(\norm{x}_{\ell_q^n}) = \frac{1}{q} \sum_{i=1}^n \abs{x_i}^q$. Observe that $\nabla^2 W = (q-1) \text{diag}(\abs{x_i}^{q-2})$, and hence $\max_{x \in B_q^n} \snorm{\nabla^2 W(x)}_{op} = q-1$ whenever $q \geq 2$. It remains to recall that $C_{Poin}(B_q^n)$ is of the order of $n^{-1/q}$ by Theorem \ref{thm:Poincare} and that $B_2^n \subset B_q^n$ when $q \geq 2$, and so Theorem \ref{thm:B-estimate} yields the assertion. 
\end{proof}

\bigskip

\section{Continuity of $\B$, $\BNH$, $\D$ with application to $B_q^n$} \label{sec:continuity}

\subsection{Continuity of $\B$, $\BNH$, $\D$ in $C$-topology}

\begin{prop} \label{prop:continuity}
Let $\set{\K_i} \subset \K_e$. If $K_i \rightarrow K$ in the $C$-topology then $\BNH(K_i) \rightarrow \BNH(K)$, $\B(K_i) \rightarrow \B(K)$  and $\D(K_i) \rightarrow \D(K)$. 
\end{prop}
\begin{proof}
As this is not a cardinal point in this work, let us only sketch the proof, as providing all details would be tedious.

It is easy to see that the mappings $\K_e \ni K \mapsto \B(K),\BNH(K),\D(K)$ are lower semi-continuous in the $C$ topology, being the suprema of a continuous family of functionals (parametrized by $u$). For instance, for $\B, \BNH$ and a fixed $u$, the functional is:
\[
\K_e \ni K \mapsto \frac{\int_{\partial K} \frac{u_\nu^2}{\scalar{x,\nu_{\partial K}(x)}} dx}{\int_K \norm{\nabla^2 u}^2 dx} ,
\]
which are continuous in $C$ since the vector valued measure $\frac{1}{\scalar{x,\nu_{\partial K}(x)}}\nu_{\partial K} \H^{n-1}|_{\partial K}$ weakly converges under $C$ convergence of convex bodies (for a more general statement regarding generalized curvature measures, also known as support measures, see \cite[Theorem 4.2.1]{Schneider-Book}).

The harder part is to show the upper semi-continuity. To see this for $\D(K)$, for instance, let $\set{K_i}$ denote a sequence on which $\limsup_{K_i \rightarrow K} \D(K_i)$ is attained. Since $K \mapsto \D(K)$ is invariant under homothety, we may assume w.l.o.g. that $rB_2^n \subset K_1 \subset K_2 \subset \ldots K \subset R B_2^n$. 
Denote by $u^i$ the test function for which:
\[
 \int_{K_i} \norm{\nabla^2 u^i}^2 dx = 1 \text{ and } \int_{\partial K_i} \frac{\abs{\nabla u^i}^2}{\scalar{x,\nu_{\partial K_i}(x)}} dx \geq \D(K_i) - \frac{1}{i}.
\]
As $\{u^i\}_{i \geq j}$ are bounded in $H^2(K_j)$, they have a weakly convergent subsequence in $H^2(K_j)$, and by a diagonalization argument, we may extract a subsequence (which we continue to denote $\{u^i\}$) weakly converging to $u \in H^2(K)$, so that:
\[
\lim_{i \rightarrow \infty} \int_{K_i} \scalar{\nabla^2 (u^i - u) , \varphi} dx = 0 
\]
for any smooth $2$-tensor $\varphi$ on $K$. The weak convergence implies that $\int_K \norm{\nabla^2 u} dx \leq 1$. By compactness  \cite[Corollary 7.4]{Biegert-SobolevTraces} of the trace embedding of Sobolev space on Lipschitz domains with upper Alfhors measures (such as $\mu_K := \frac{1}{\scalar{x,\nu_{\partial K}}} d\H^{n-1}|_{\partial K}$ in our setting), which in fact holds with a uniform constant for all $K_i$ (as they are uniformly Lipschitz and $\mu_K$ is uniformly upper Ahlfors, owing to convexity and $r,R$ being uniform), the weak convergence in $H^2$ implies strong convergence in the trace $H^1$ norm:
\[
\lim_{i \rightarrow \infty} \int_{\partial K_i} \abs{\nabla u^i - \nabla u}^2 d\mu_{K_i} = 0 .
\]
By weak convergence of $\mu_{K_i}$ to $\mu_K$ (as for the lower semi-continuity direction), there exists for any $\eps > 0$ a large enough $i_\eps$, so that for all $i \geq i_{\eps}$:
\begin{align*}
\D(K) & \geq \frac{\int_{\partial K} \frac{\abs{\nabla u}^2}{\scalar{x,\nu_{\partial K}(x)}} dx}{\int_K \norm{\nabla^2 u}^2 dx} \geq \int_{\partial K} \abs{\nabla u}^2 d\mu_{K} \geq \int_{\partial K_i} \abs{\nabla u}^2 d\mu_{K_i} - \eps \\
& \geq \int_{\partial K_i} \abs{\nabla u^i}^2 d\mu_{K_i} - \int_{\partial K_i} \abs{\nabla u^i - \nabla u}^2 d\mu_{K_i} - \eps \\
&\geq  \D(K_i) - \frac{1}{i} - \int_{\partial K_i} \abs{\nabla u^i - \nabla u}^2 d\mu_{K_i} - \eps . 
\end{align*}
Taking the limit as $i\rightarrow \infty$ and $\eps \rightarrow 0+$, the upper semi-continuity of $\D$ follows.  
Note that the limiting $u$ is not guaranteed to be in $\S_{N,e}(K)$, only in $H^2(K)$, but can be approximated in $H^2(K)$ by functions in $\S_{N,e}(K)$, and by the trace embedding theorem, also in $H^1(d\mu_K)$, and hence the above lower bound on $\D(K)$ is legitimate.

The proof is identical for $\BNH(K)$. For $\B(K)$, one just has to note that the limiting $u$ will be harmonic as the weak $H^2$ limit of the harmonic $u^i$. \end{proof}

\subsection{The Cube}

We can now extend Theorem \ref{thm:B-Cube} to a result on the even spectral-gap $\lambda_{1,e}(B_\infty^n)$ of the formal Hilbert--Brunn--Minkowski operator associated to $B_\infty^n$. Recalling the notation from Section \ref{sec:LK}, and in particular the definition (\ref{eq:lambda-nonsmooth}):
\[
\lambda_{1,e}(B_\infty^n) := \liminf_{\K^2_{+,e} \ni K \rightarrow B_\infty^n \text{ in $C$}}  \lambda_{1,e}(-L_K) ,
\]
we have:
\begin{thm} \label{thm:lambda1-cube}
$\lambda_{1,e}(B_\infty^n) = \frac{n}{n-1}$. 
\end{thm}
\begin{proof}
By Proposition \ref{prop:continuity} and Theorem \ref{thm:B-Cube}, we have:
\[
\lim_{\K^2_{+,e} \ni K_i \rightarrow B_\infty^n \text{ in $C$}} \B(K_i) = B(B_\infty^n) = 1.
\]
By Theorem \ref{thm:B2LK}, we know that  $\lambda_{1,e}(-L_K) \geq 1 + \frac{1}{(n-1) \B(K)}$ for any $K \in \K^2_{+,e}$. Consequently:
\[
\liminf_{\K^2_{+,e} \ni K \rightarrow B_\infty^n \text{ in $C$}}\lambda_{1,e}(-L_K) \geq \frac{n}{n-1} .
\]
To see that we actually have equality in the above inequality, it is enough to test the specific sequence $\set{B_q^n} \subset \K^2_{+,e}$ which converges to $B_\infty^n$ in $C$ as $q \rightarrow \infty$. Moreover, it is enough to show that the inequality $R_{B_q^n}(u^0) \geq 0$ we employed in Theorem \ref{thm:sufficient}, when transitioning from the equivalent condition for the local log-BM conjecture to the sufficient one, is not wasteful for the extremal even harmonic function $u^0 := \frac{x_1^2}{2} - \frac{x_2^2}{2}$ for $B_\infty^n$. 
Using Remark \ref{rem:RK}, we need to show:
\begin{equation} \label{eq:Bqn-goal}
R_{B_q^n}(u^0) = \int_{\partial B_q^n} \scalar{ \II^{-1}_{\partial B_q^n} P_{T_{\partial B_q^n}} \bigl[ \nabla^2 u^0 \cdot \nu \bigr], P_{T_{\partial B_q^n}} \bigl[\nabla^2 u^0 \cdot \nu \bigr]}  dx \rightarrow 0 \;\;\; \text{as $q \rightarrow \infty$} \; .
\end{equation}
A computation verifies that on the positive orthant:
\[
\nu(x) = \frac{\{x^{q-1}_i\}_{i=1}^n}{\sqrt{\sum_{i=1}^n x^{2(q-1)}_i}} ~,~ \II_{\partial B_q^n} = \Lambda^{1/2} U \Lambda^{1/2}|_{T_{\partial B_q^n}} ,
\]
where:
\[
\Lambda(x) = \frac{q-1}{\sqrt{\sum_{i=1}^n x^{2(q-1)}_i}}{\rm{diag}}(x^{q-2}_i) ~,~ U = \langle \Lambda \nu, \nu \rangle \langle \Lambda^{-1} \nu, \nu \rangle \hat{e}_2 \otimes \hat{e}_2 + \sum_{k >2}^n  \hat{e}_k \otimes \hat{e}_k ~ ,
\]
and $\set{\hat{e}_k}_{k=1}^n$ is an orthonormal frame with:
\[
\hat{e}_i = \frac{\tilde{e}_i}{\abs{\tilde{e}_i}} ~,~ \tilde{e}_1 = \Lambda^{1/2} \nu ~,~ \tilde{e}_2 = \Lambda^{-1/2} \nu - \frac{1}{\scalar{\Lambda \nu,\nu}} \Lambda^{1/2} \nu .
\]
Consequently:
\[
\II^{-1}_{\partial B_q^n} = \bigl( \Lambda^{-\frac{1}{2}}  \tilde{U} \Lambda^{-\frac{1}{2}} \Bigr)|_{T_{\partial K}} \text{ with } \;
\tilde{U} = \frac{1}{\langle \Lambda \nu, \nu \rangle \langle \Lambda^{-1} \nu, \nu \rangle} \hat{e}_2 \otimes \hat{e}_2 
 + \sum_{k >2}^n  \hat{e}_k \otimes \hat{e}_k ~ .
\]
It follows that the integrand in (\ref{eq:Bqn-goal}) is bounded by:
\begin{multline*} \scalar{ \II^{-1}_{\partial B_q^n} P_{T_{\partial K}} \bigl[ \nabla^2 u^0 \cdot \nu \bigr], P_{T_{\partial K}} \bigl[\nabla^2 u^0 \cdot \nu \bigr]} \leq \\
\max\brac{1,\frac{1}{\langle \Lambda \nu, \nu \rangle \langle \Lambda^{-1} \nu, \nu \rangle}} \abs{\Lambda^{-\frac{1}{2}} P_{T_{\partial K}} \bigl[ \nabla^2 u^0 \cdot \nu \bigr] }^2 . 
\end{multline*}
It is easy to check that:
\[
P_{T_{\partial K}} \bigl[ \nabla^2 u^0 \cdot \nu \bigr] = \frac{\{ a_i(x) x^{q-1}_i\}_{i=1}^n}{\sqrt{\sum_{i=1}^n x^{2(q-1)}_i}} ~,~ \abs{a_i(x)} \leq 2 . 
\]
Plugging in the above expression for $\Lambda$ and applying H\"{o}lder's inequality (using $\sum_{i=1}^n x_i^q = 1$), a straightforward calculation verifies that the integrand goes to zero uniformly in $x$ as $q \rightarrow \infty$, and the claim is established. 
\end{proof}

In view of Corollary \ref{cor:LK-SG} and Proposition \ref{prop:local-equiv}, Theorem \ref{thm:intro-cube} is a reformulation of Theorem \ref{thm:lambda1-cube}. 
Similarly, recalling definition (\ref{eq:lambda-unc}), we extend Theorem \ref{thm:unc} to non-smooth $K \in \K_{\unc}$:

\begin{cor} \label{cor:lambda1-unc}
For all $K \in \K_{\unc}$, $\lambda_{1,\unc}(K) \geq \frac{n}{n-1}$.
\end{cor}

\subsection{Unit-balls of $\ell_q^n$}

Observe that $B_q^n \notin \K^2_+$ whenever $q \neq 2$. Consequently, we will use Proposition \ref{prop:continuity} to obtain a neighborhood $N^C_{B_q^n}$ of $B_q^n$ in the $C_e$-topology, so that the results of the previous sections may be applied to its dense subset $N^C_{B_q^n} \cap \K^2_{+,e}$. This yields Theorems \ref{thm:intro-lq} and \ref{thm:intro-ellq1} from the Introduction, which we restate here as follows:

\begin{thm} \label{thm:ellq}
For all $q \in (2,\infty)$, there exists $n_q \geq 2$ so that for all $n \geq n_q$, there exists a neighborhood $N^C_{B_q^n}$ of $B_q^n$ in the $C_e$-topology, so that 
the local log-BM conjecture (\ref{eq:local-log-BM}) holds for $T(K)$ for all $K \in N^C_{B_q^n} \cap \K^2_{+,e}$ and $T \in GL_n$. In addition, for any $K \in N^C_{B_q^n} \cap \K^2_{+,e}$, there exists a $C^2$-neighborhood $N_K$ of $K$ in $\K^2_{+,e}$ so that for all $T \in GL_n$ and $K_1,K_0 \in T(N_{K})$:
\[
V((1-\lambda) \cdot K_0 +_0 \lambda \cdot K_1) \geq V(K_0)^{1-\lambda} V(K_1)^{\lambda} \;\;\; \forall \lambda \in [0,1] . 
\]
\end{thm}
\begin{proof}
Recall that by Theorem \ref{thm:BBqn}, $\B(B_q^n) \leq C (n^{-1/q} + q n^{-2/q})$. Setting $n_q = \exp(\frac{q}{2} \log(C' q))$, it follows that $\B(B_q^n) \leq \frac{1}{2}$ for all $n \geq n_q$. By Proposition \ref{prop:continuity}, there exists a neighborhood $N^C_{B_q^n}$ of $B_q^n$ in the $C_e$-topology so that $\B(K) \leq \frac{3}{4}$ for all $K \in N^C_{B_q^n}$. Consequently, Theorem \ref{thm:main-tool} implies that the local $p$-BM conjecture  (\ref{eq:local-p-BM}) holds with $p = -\frac{1}{3}$ for $T(K)$ for all $K \in N^C_{B_q^n} \cap \K^2_{+,e}$ and $T \in GL_n$, implying in particular the first assertion. The second assertion follows by invoking Proposition \ref{prop:local-equiv}. 
\end{proof}

\begin{thm} \label{thm:ellq1}
There exists a universal constant $c \in (0,1)$ so that for all $q \in [1,2)$, there exists a neighborhood $N^C_{B_q^n}$ of $B_q^n$ in the $C_e$-topology, so that 
the $p$-BM conjecture (\ref{eq:local-p-BM}) holds with $p=c$ for $T(K)$ for all $K \in N^C_{B_q^n} \cap \K^2_{+,e}$ and $T \in GL_n$. In addition, for any $K \in N^C_{B_q^n} \cap \K^2_{+,e}$, there exists a $C^2$-neighborhood $N_K$ of $K$ in $\K^2_{+,e}$ so that for all $T \in GL_n$ and $K_1,K_0 \in T(N_{K})$:
\[
V((1-\lambda) \cdot K_0 +_c \lambda \cdot K_1) \geq \brac{(1-\lambda) V(K_0)^{\frac{c}{n}} + \lambda V(K_1)^{\frac{c}{n}}}^{\frac{n}{c}} \;\;\; \forall \lambda \in [0,1] . 
\]
\end{thm}
\begin{proof}
By Lemma \ref{lem:D-Bqn}, there exists a universal constant $C>1$ so that for all $q \in [1,2)$, $\B(B_q^n) \leq \D(B_q^n) \leq C$. 
By Proposition \ref{prop:continuity}, there exists a neighborhood $N^C_{B_q^n}$ of $B_q^n$ in the $C_e$-topology so that $\B(K) \leq 2 C$ for all $K \in N^C_{B_q^n}$. Consequently, Theorem \ref{thm:main-tool} implies that the local $p$-BM conjecture  (\ref{eq:local-p-BM}) holds with $p = 1 - \frac{1}{2C}$ for $T(K)$ for all $K \in N^C_{B_q^n} \cap \K^2_{+,e}$ and $T \in GL_n$, implying in particular the first assertion. Setting $c = 1 - \frac{1}{3C}$, the second assertion follows by invoking Proposition \ref{prop:local-equiv}. 
\end{proof}

\bigskip

\section{Local Uniqueness for Even $L^p$-Minkowski Problem} \label{sec:Mink}

In this section, we describe an application (which is by now well-understood and standard -- see \cite{Lutwak-Firey-Sums,BLYZ-logBMInPlane}) of our local $p$-BM and log-BM inequalities to local 
uniqueness statements for the even $L^p$-Minkowski and log-Minkowski problems. 

\smallskip

The classical Minkowski problem (see \cite{Schneider-Book,LYZ-LpMinkowskiProblem} and the references therein) asks for necessary and sufficient conditions on a finite Borel measure $\mu$ on $S^{n-1}$, to guarantee the existence and uniqueness (up to translation) of a convex body $K \in \K$ so that its surface-area measure $dS_K$ coincides with $\mu$. 
It was shown by Minkowski for polytopes and by Aleksandrov for general convex bodies, that a necessary and sufficient condition is to require that the centroid of $\mu$ is at the origin and that its support is not contained in a great subsphere. In \cite{Lutwak-Firey-Sums}, Lutwak proposed to study the analogous $L^p$-Minkowski problem, where the role of the surface-area measure $dS_K$ is replaced by the $L^p$-surface-area measure:
\[
 dS_{K,p} := h_K^{1-p} dS_K .
 \]
For \emph{even} measures, Lutwak showed in \cite{Lutwak-Firey-Sums} that Minkowski's condition is again necessary and sufficient for existence and uniqueness (no translations required now) in the case $1 < p \neq n$  (see also Lutwak--Yang--Zhang \cite{LYZ-LpMinkowskiProblem} for the case $p=n$).

\smallskip

The same question may be extended to the range $p < 1$. Of particular interest is the the log-Minkowski problem, 
which pertains to the cone-measure $dV_K$ (corresponding to the case $p=0$).
 For \emph{even} measures $\mu$, a novel necessary and sufficient \emph{subspace concentration condition} ensuring the \emph{existence} question was obtained in \cite{BLYZ-logMinkowskiProblem}, and the \emph{uniqueness} question was settled in \cite{BLYZ-logBMInPlane} in dimension $n=2$; it remains open in full generality in dimension $n \geq 3$ (see also \cite{GageLogBMInPlane} for the planar uniqueness question for smooth convex bodies with strictly positive curvature, and \cite{StancuDiscreteLogBMInPlane} for existence and uniqueness in the even planar problem when $\mu$ is assumed discrete). 
   Various other partial results pertaining to the uniqueness question are known (see e.g. \cite{HuangLiuXu-UniquenessInLpMinkowski, MaLogBMInPlane, XiLeng-DarAndLogBMInPlane, ColesantiLivshyts-LocalpBMUniquenessForBall} and the references therein).
Without assuming evenness of $\mu$, both existence and uniqueness problems are more delicate, and there is a huge body of works on this topic which we do not attempt to survey here.
  Instead, let us mention the known intimate relation between the uniqueness question and the $p$-BM inequality with its equality conditions. 

\smallskip
Recall the definition (\ref{eq:Vp}) of the $L^p$-mixed-volume $V_p(K,L)$, introduced by Lutwak in \cite{Lutwak-Firey-Sums}. It was shown in \cite{BLYZ-logBMInPlane} (for $p \in (0,1)$, but the proof extends to all $p < 1$) that:
\[
V_p(K,L) = \frac{1}{n} \int_{S^{n-1}} h_L^{p} dS_{K,p} . 
\]

\begin{prop} \label{prop:Mink}
Let $K_0,K_1 \in \K_e$ and $p < 1$. Then each of the following statements implies the subsequent one:
\begin{enumerate}
\item The function $[0,1] \ni \lambda \mapsto g_p(\lambda) := \frac{1}{p} V((1-\lambda) \cdot K_0 +_p \lambda \cdot K_1)^{\frac{p}{n}}$ is concave, and it is affine 
 if and only if $K_0$ and $K_1$ are dilates. 
\item The first $L^p$-Minkowski inequality (\ref{eq:1stMink}) holds for the pair $K,L = K_0,K_1$ and for the pair $K,L = K_1,K_0$,
with equality in either of these cases if and only if $K_0$ and $K_1$ are dilates. 
\item $dS_{K_0,p} = dS_{K_1,p}$ implies $K_0 = K_1$. 
\end{enumerate}
\end{prop}
Slightly more is required for the converse implications to hold, as worked out in \cite{BLYZ-logBMInPlane}, but we do not require this here.
\begin{proof}
It was shown in \cite{BLYZ-logBMInPlane} (for $p \in [0,1)$, but the proof extends to all $p<1$) that:
\begin{align*} \left . \frac{d}{d\lambda} \right |_{\lambda=0+} V((1-\lambda) \cdot K_0 +_p \lambda \cdot K_1) & = \frac{1}{p} \int_{S^{n-1}} h_{K_0}^{1-p} (h_{K_1}^p - h_{K_0}^p) dS_{K_0}\\
&  = \frac{n}{p} \brac{V_p(K_0,K_1) - V(K_0)} . 
\end{align*}
Consequently, the chain rule yields: \[
\left . \frac{d}{d\lambda}\right|_{\lambda=0+} g_p(\lambda) = V(K_0)^{\frac{p}{n}-1} \frac{1}{p} \brac{V_p(K_0,K_1) - V(K_0)} ,
\]
with the case $p=0$ understood in the limiting sense. 
The concavity in statement (1) implies
$\frac{d}{d\lambda}|_{\lambda=0+} g_p(\lambda) \geq g_p(1) - g_p(0)$, which is precisely (\ref{eq:1stMink}). Reversing the roles of $K_0,K_1$ by the symmetry of (1), (\ref{eq:1stMink}) also holds in that case. Clearly we have equality in (\ref{eq:1stMink}) if $K_0$ and $K_1$ are dilates. Conversely, equality in (\ref{eq:1stMink}) translates to $\frac{d}{d\lambda}|_{\lambda=0+} g_p(\lambda) = g_p(1) - g_p(0)$, and since $g_p$ is assumed concave, it follows that it must be affine, and so the equality conditions in (1) imply those in (2). 

If $dS_{K_0,p} = dS_{K_1,p}$, then for any $Q \in \K$:
\[
V_p(K_0,Q) = \frac{1}{n} \int_{S^{n-1}} h_{Q}^p dS_{K_0,p} = \frac{1}{n} \int_{S^{n-1}} h_{Q}^p dS_{K_1,p}  = V_p(K_1,Q) .
\]
Assuming for simplicity that $p > 0$ (but an identical argument holds for general $p$), we have by  (\ref{eq:1stMink}) for both $i=0,1$ that:
\[
V(K_i) = V_p(K_i,K_i) = V_p(K_{1-i},K_i) \geq V(K_{1-i})^{1 - \frac{p}{n}} V(K_i)^{\frac{p}{n}} .
\]
It follows that $V(K_0) = V(K_1)$ and hence we have equality in (\ref{eq:1stMink}) for the pair $K_{1-i},K_{i}$. The equality conditions in (2) therefore imply that $K_0=K_1$. 
\end{proof}

\begin{defn*}
Given $p < 1$ and $K \in \K^2_{+,e}$, we will say that the even $L^p$-Minkowski problem has a locally unique solution in a neighborhood of $K$ if there exists a $C^2$-neighborhood $N_{K,p}$ of $K$ in $\K^2_{+,e}$, so that for all $T \in GL_n$, for all $K_0,K_1 \in T(N_{K,p})$, if $dS_{K_0,p} = dS_{K_1,p}$ then $K_0 = K_1$. 
\end{defn*}

\begin{thm} \label{thm:Mink-Uniq}
Assume that the local $p_0$-BM conjecture (\ref{eq:local-p-BM}) holds for $K \in \K^2_{+,e}$ and some $p_0 < 1$. Then for any $p \in (p_0,1)$, the even $L^p$-Minkowski problem has a locally unique solution in a neighborhood $N_{K,p}$ of $K$.
\end{thm}
\begin{proof}
Given $p \in (p_0,1)$, denote $p_1 = \frac{p + p_0}{2} \in (p_0,p)$.
Proposition \ref{prop:local-equiv} ensures the existence of a neighborhood $N_{K,p}$ so that for all $T \in GL_n$ and $K_0,K_1 \in T(N_{K,p})$, $K_0$ satisfies the local $p_1$-BM inequality (\ref{eq:local-p-BM}), and in addition, the function $g_p(\lambda)$ appearing in Proposition \ref{prop:Mink} (1) is concave. It remains to establish the equality conditions in Proposition \ref{prop:Mink}  (1) to deduce the local uniqueness statement in (3). Assume that $K_0,K_1 \in T(N_{K,p})$ are such that the function $g_p(\lambda)$ is affine. It follows by the argument in the proof of Lemmas \ref{lem:global2local} and \ref{lem:loc-to-global} that equality holds in the local $p$-BM inequality (\ref{eq:local-p-BM}) for the body $K_0$ and $\frac{1}{p} f^p_0 = \frac{1}{p} h_{K_1}^p - \frac{1}{p} h_{K_0}^p \in C_e(S^{n-1})$. Recalling the equivalent formulations of the local $p$-BM equality and $p_1$-BM inequality derived in Section \ref{sec:equivalent}, and denoting:
\[
z_0 := \begin{cases} \frac{1}{h_{K_0}^p} \frac{f^p_0}{p} = \frac{1}{p} \brac{\brac{\frac{h_{K_1}}{h_{K_0}}}^p - 1} & p \neq 0 \\ \log f_0 = \log \frac{h_{K_1}}{h_{K_0}} & p = 0 \end{cases} \\
\]
as in (\ref{eq:z}), we deduce (say, using the formulation of (\ref{eq:p-BM-mixed-vols})):
\begin{align*}
\frac{1}{V(K_0)} V(z_0 h_{K_0};1)^2 = \frac{n-1}{n-p} V(z_0 h_{K_0} ; 2) + \frac{1-p}{n-p} V(z^2_0 h_{K_0};1) ~, \\
\frac{1}{V(K_0)} V(z_0 h_{K_0};1)^2 \geq \frac{n-1}{n-p_1} V(z_0 h_{K_0} ; 2) + \frac{1-p_1}{n-p_1} V(z^2_0 h_{K_0};1) .
\end{align*}
Equating the $V(z_0 h_{K_0} ; 2)$ terms above and using that $p_1 < p < 1 \leq n$, it follows that:
\[
V(z^2_0 h_{K_0} ; 1) \leq \frac{1}{V(K_0)} V(z_0 h_{K_0} ; 1)^2 . 
\]
On the other hand, the reverse inequality is always satisfied by Cauchy--Schwarz (\ref{eq:CS}). By the equality conditions of Cauchy--Schwarz, it follows that $z_0$ must be a constant $dS_{K_0}$-a.e. on $S^{n-1}$. Using the fact that $dS_{K_0}$ and the Lebesgue measure are equivalent since $K_0 \in \K^2_{+}$, and as support functions are continuous, it follows that $h_{K_1} = C h_{K_0}$ identically on $S^{n-1}$ for some $C > 0$, and the equality case in Proposition \ref{prop:Mink} (1) is established. 
\end{proof}

It is now immediate to translate the results of this work into the following:
\begin{thm}
The even $L^p$-Minkowski problem has a locally unique solution in a neighborhood of $K$ for all $p \in (p_K,1)$, in the following cases:
\begin{enumerate}
\item For any $K \in \K^2_{+,e}$ and $p_K = 1 - \frac{c}{n^{3/2}}$. 
\item For $K = B_2^n$ and $p_K = -n$. 
\item For any $\eps > 0$, for all $K \in N^{C,\eps}_{B_\infty^n} \cap \K^2_{+,e}$ and $p_K = \eps$, where $N^{C,\eps}_{B_\infty^n}$ is an appropriate $C$-neighborhood of $B_\infty^n$ (depending on $\eps$).
\item If $q \in (2,\infty)$, for all $K \in N^{C}_{B_q^n} \cap \K^2_{+,e}$ and $p_K = 1 - \frac{c}{n^{-1/q} + q n^{-2/q}}$, where $N^{C}_{B_q^n}$ is an appropriate $C_e$-neighborhood of $B_q^n$. 
\item If $q \in [1,2)$, for all $K \in N^{C}_{B_q^n} \cap \K^2_{+,e}$ and $p_K = c \in (0,1)$, where $N^{C}_{B_q^n}$ is an appropriate $C_e$-neighborhood of $B_q^n$. 
\end{enumerate}
\end{thm}
\begin{proof}
(1) follows from Theorem \ref{thm:main}. (2) follows from Theorem \ref{thm:Ball}. (3) follows from Theorem \ref{thm:lambda1-cube} and Corollary \ref{cor:LK-SG}.
(4) follows from Theorem \ref{thm:ellq}. (5) follows from Theorem \ref{thm:ellq1}. 
\end{proof}

Case (2) should be compared with a result of Colesanti and Livshyts \cite{ColesantiLivshyts-LocalpBMUniquenessForBall}, who considered local uniqueness for the log-Minkowski problem (the case $p=0$), and showed the existence of a $C^2$-neighborhood $N_{B_2^n}$ of $B_2^n$, so that for all $K \in N_{B_2^n}$, if $dV_K = dV_{B_2^n}$ then necessarily $K = B_2^n$.

 \section{Stability Estimates for Brunn--Minkowski and Isoperimetric Inequalities} \label{sec:stability}

It is well-known that by differentiating the Brunn--Minkowski inequality, one can obtain isoperimetric inequalities. Recall that the sharp anisotropic isoperimetric inequality is the statement that:
\begin{equation} \label{eq:isop-inq}
 P_{L}(K) \geq n V(L)^{\frac{1}{n}} V(K)^{\frac{n-1}{n}} ,
\end{equation}
with equality when $K = L$, where:
  \begin{align}
  \label{eq:perimeter} P_{L}(K) & := \lim_{\varepsilon \to 0^{+}} \frac{V(K + \eps L) - V(K)}{\eps} = n V(L,K,\ldots,K) \\
  \nonumber & = \int_{S^{n-1}} h_L dS_K = \int_{\partial K} h_{L}(\nu_{\partial K}) \ dx ,
  \end{align}
  is the anisotropic perimeter of $K$ with respect to the convex body $L$. Note that in the isotropic case, when $L = B_2^n$, $P_{B_2^n}(K)$ boils down to the usual surface area $S(K)$ of $K$. 
  While the Brunn--Minkowski and isoperimetric inequalities hold for all Borel sets $K$, we restrict (as usual in this work) our discussion to convex bodies only. 
  
  \smallskip
  Since the (local) $p$-BM inequality for $p<1$ is a strengthening of the classical (local) BM inequality (in the class of origin-symmetric convex bodies), it is natural to expect that some stability results
 for isoperimetric inequalities could be derived from it.  In the first part of this section, we rewrite the equivalent formulations for the local $p$-BM inequality obtained in Section \ref{sec:equivalent} in a manner which is 
 more suitable for obtaining stability estimates. In particular, we obtain an interesting interpretation of the Hilbert--Brunn--Minkowski operator $L_K$ as the operator controlling the deficit in Minkowski's second inequality, and as a consequence, deduce new stability estimates for the anisotropic isoperimetric inequality and the (global) Brunn--Minkowski inequality for origin-symmetric convex bodies. 

 In the second part of this section, we show how to derive a strengthening of the best known stability estimates in the (global) Brunn--Minkowski and anisotropic isoperimetric inequalities for the class of all (not necessarily origin-symmetric) convex bodies.

\subsection{New stability estimates for origin-symmetric convex bodies with respect to variance}

Throughout this subsection, the body $K$ is assumed to be a convex body in $\K^2_{+,e}$. Recall that $L_K$ denotes the Hilbert--Brunn--Minkowski operator on $L^2(dV_K)$ associated to $K$ defined in Section \ref{sec:LK}, and that the local $p$-BM inequality for $K$ is equivalent by Proposition \ref{prop:LK-SG} to the inequality:
\begin{equation} \label{eq:pbm}
\int_{S^{n-1}} (-L_K z) z dV_K \geq \frac{n-p}{n-1} \Var_{dV_K}(z)  \;\;\; \forall z \in C^2_e(S^{n-1}) ,
\end{equation}
where:
\[
\Var_{dV_K}(z) := \int_{S^{n-1}} z^2 dV_K - \frac{1}{V(K)} \brac{\int_{S^{n-1}} z dV_K}^2 . 
\]
Also recall that one of our main results in this work is verifying the validity of (\ref{eq:pbm}) for $p = 1 - c n^{-3/2}$.

\begin{thm}[Stability estimate for Minkowski's second inequality] \label{thm:Mink-stability}
Assume that $K \in \K^2_{+,e}$ satisfies the local $p$-BM inequality (\ref{eq:pbm}) with $p \leq 1$. Then for all $L \in \K^2_{+,e}$: 
\begin{align*}
\frac{V(L,K,\ldots,K)^2}{V(K)} & \geq V(L,L,K,\ldots,K) + \frac{1-p}{n-p} R_K(L) \\
& \geq V(L,L,K,\ldots,K) + \frac{1-p}{n-1} \Var_{dV_K}\brac{\frac{h_L}{h_K}} ,
\end{align*}
where:
\begin{align}
\nonumber R_K(L) & :=  \int_{S^{n-1}} \brac{-L_K \frac{h_L}{h_K}} \frac{h_L}{h_K} dV_K \\
\label{eq:RKL} & = \frac{1}{n (n-1)} \int_{S^{n-1}} \scalar{(D^2 h_K)^{-1} \xi, \xi} dS_K \\
\label{eq:RKL2} & = \frac{1}{n (n-1)} \int_{\partial K} \scalar{\II_{\partial K} \zeta , \zeta } dx  ,
\end{align}
and:
\begin{align*}
\xi & := h_K \nabla_{S^{n-1}} \frac{h_L}{h_K} =  \nabla_{S^{n-1}} h_L - h_L \nabla_{S^{n-1}} \log h_K ~,~ \\
\zeta & := \xi(\nu_{\partial K}) = \nabla_{S^{n-1}} h_L(\nu_{\partial K}) - h_L(\nu_{\partial K}) \frac{x - \scalar{x,\nu_{\partial K}} \nu_{\partial K}}{\scalar{x,\nu_{\partial K}}} . 
\end{align*}
In particular, $R_K(L) = 0$ if and only if $L=C K$ for some constant $C > 0$. \end{thm}
\begin{rem}
Note that in the isotropic case (when $L = B_2^n$), $h_L(\theta) \equiv 1$ and $\nabla_{S^{n-1}} h_L = 0$, so the above expressions for $\xi$ and $\zeta$ simplify considerably. In addition, it will be clear from the proof that $L$ need not be a convex body, and that any $w \in C^2_e(S^{n-1})$ will work equally well. 
\end{rem}

\begin{proof}[Proof of Theorem \ref{thm:Mink-stability}]
Recall that by (\ref{eq:p-BM-mixed-vols}), the local $p$-BM inequality for $K$ is equivalent to:
\[
\forall z \in C^2_e(S^{n-1}) \;\;\; \frac{1}{V(K)} V(z h_K;1)^2 \geq \frac{n-1}{n-p} V(z h_K ; 2) + \frac{1-p}{n-p} V(z^2 h_K;1) ,
\]
where as usual $V(w; m)$ denotes the mixed volume of $w$ repeated $m$ times and $h_K$ repeated $n-m$ times. 
Applying this to $z := \frac{h_L}{h_K}$, we deduce:
\[
\frac{V(L,K,\ldots,K)^2}{V(K)} \geq V(L,L,K,\ldots,K) + \frac{1-p}{n-p} \brac{ V(z^2 h_K ; 1) - V(z h_K ; 2)} .
\]
On the other hand, recall by (\ref{eq:tilde-LK-as-mixedvol}) that $L_K = \tilde{L}_K - \text{Id}$ satisfies
\[
R_K(L) = \int_{S^{n-1}} (-L_K z) z dV_K = \int_{S^{n-1}} z^2 dV_K - \int_{S^{n-1}} (\tilde{L}_K z) z dV_K = V(z^2 h_K ; 1) - V(z h_K; 2) . 
\]
This verifies the first asserted inequality. The second asserted inequality follows immediately by (\ref{eq:pbm}). Note that $R_K(L) = 0$ if and only if $h_L = C h_K$ since $0$ is an eigenvalue of $-L_K$ with corresponding eigenspace spanned by the constant functions.

It remains to verify the alternative expressions derived for $R_K(L)$. Recall that by (\ref{eq:Dirichlet}):
\begin{align*}
\int_{S^{n-1}} (-L_K z) z dV_K & = \frac{1}{n-1} \int_{S^{n-1}} h_K ((D^2 h_K)^{-1})^{i,j} z_i z_j dV_K \\
& = \frac{1}{n (n-1)}\int_{S^{n-1}} ((D^2 h_K)^{-1})^{i,j} (h_K z_i) (h_K z_j) dS_K ,
\end{align*}
which yields the expression in (\ref{eq:RKL}). The one in (\ref{eq:RKL2}) follows by the change-of-variables $\theta = \nu_{\partial K}(x)$, as in Subsection \ref{subsec:FromSToPartialK}. Finally, recall that the inverse of the latter mapping is the Weingarten map $x = \nabla_{\Real^n} h_K(\theta)$, and since $\nabla_{S^{n-1}} h_K(\theta) = \nabla_{\Real^n} h_K(\theta) - \scalar{\nabla_{\Real^n} h_K, \theta} \theta$, we obtain the last expression for $\zeta$. 
\end{proof}

To deduce an isoperimetric inequality from Theorem \ref{thm:Mink-stability}, we employ the following well-known consequence of the Aleksandrov--Fenchel inequality (or simply Minkowski's inequality for mixed-volumes) \cite[p. 334]{Schneider-Book}:
\[
V(L,L,K,\ldots,K) \geq V(L)^{\frac{2}{n}} V(K)^{\frac{n-2}{n}} .
\]
It is also possible to employ the stronger inequality:
\[
  V(L,L,K,\ldots,K) \geq V(L)^{\frac{1}{n-1}} V(L,K,\ldots,K)^{\frac{n-2}{n-1}} ,
\]
but this would yield a slightly less aesthetically pleasing expression below, and so we leave this possibility to the interested reader. Together with (\ref{eq:perimeter}), we immediately obtain:

\begin{cor}[Stability estimate for anisotropic isoperimetric inequality] \label{cor:isop-stability} 
Assume that $K \in \K^2_{+,e}$ satisfies the local $p$-BM inequality (\ref{eq:pbm}) with $p \leq 1$. Then for all $L \in \K^2_{+,e}$: 
\begin{align*}
P_L^2(K) & \geq \brac{n V(L)^{\frac{1}{n}} V(K)^{\frac{n-1}{n}}}^2 + n^2 \frac{1-p}{n-p} V(K) R_K(L) \\
& \geq \brac{n V(L)^{\frac{1}{n}} V(K)^{\frac{n-1}{n}}}^2 + n^2 \frac{1-p}{n-1} V(K) \Var_{dV_K}\brac{\frac{h_L}{h_K}} .
\end{align*}
\end{cor}

Note that whenever $p<1$, the above inequalities constitute a strict improvement over the classical sharp anisotropic isoperimetric inequality (\ref{eq:isop-inq}); the latter is recovered by applying Corollary \ref{cor:isop-stability} with $p=1$ (which is always valid by the classical Brunn--Minkowski inequality, regardless of origin-symmetry). Consequently,  the terms on the right-hand-sides above may be viewed as some type of isoperimetric deficits. The term involving $\Var_{dV_K}(h_L/h_K)$ is a very intuitive way of quantifying how different $L$ is  from $K$. The term involving $R_K(L)$ is less intuitive, but nevertheless may be sometimes estimated from below by a simple geometric expression. 

\medskip

Let us demonstrate this for $n=2$. Denote by $r$ and $R$ the in and out radii of $K$ with respect to $L$, respectively, namely the best constants so that $r L \subset K \subset R L$, or equivalently:
\[
r := \min_{\theta \in S^{n-1}} \frac{h_K(\theta)}{h_L(\theta)} ~,~ R:= \max_{\theta \in S^{n-1}} \frac{h_K(\theta)}{h_L(\theta)}  .
\]

\begin{lem}
For all $K,L \in \K^2_{+,e}$ in $\Real^2$, we have:
\[
R_K(L) \geq 4 \max\brac{\frac{\brac{\frac{1}{r} - \frac{1}{R}}^2}{V(K^{\circ})} , \frac{\log^2\brac{\frac{R}{r}}}{V(L^\circ)}} ,
\]
where $Q^\circ$ denotes the polar-body to $Q$, i.e. the unit-ball for the norm $h_Q$. 
\end{lem}
\begin{proof}
When $n=2$, since $S^{n-1}$ is one dimensional and $dS_K = \det(D^2 h_K)$, we see from (\ref{eq:RKL}) that:
\[
R_K(L) = \frac{1}{2} \int_{S^1} (h_K (h_L/h_K)')^2 d\theta . 
\]
To obtain the first estimate above, we applying Cauchy--Schwarz in the following form:
\[
R_K(L) \geq \frac{1}{2} \frac{\brac{\int_{S^1} \abs{(h_L/h_K)'} d\theta}^2}{\int_{S^1} 1/h_K^2 d\theta} .
\]
By origin-symmetry, it is clear that the function $h_L / h_K$ changes its value from $1/r$ to $1/R$ to $1/r$ in half a revolution of $S^1$, and similarly for the second half, and so:
\[
\int_{S^1} \abs{(h_L/h_K)'} d\theta \geq 4 \brac{\frac{1}{r} - \frac{1}{R}} .
\]
On the other hand, since $h_K(\theta) = \norm{\theta}_{K^{\circ}}$, it follows by integration in polar coordinates that:
\[
\int_{S^1} 1/h_K^2 d\theta = 2 V(K^{\circ}) ,
\]
yielding the first estimate. The second estimate is established similarly by applying Cauchy--Schwarz in the following form:
\[
R_K(L) = \frac{1}{2} \int_{S^1} \brac{ h_L \log(h_L/h_K)' }^2 d\theta \geq \frac{1}{2} \frac{\brac{\int_{S^1} \abs{\log (h_L/h_K)'} d\theta}^2}{\int_{S^1} 1/h_L^2 d\theta} .
\]
It is clear that other variations in this spirit are also possible. 
\end{proof}

Plugging the latter estimate on $R_K(L)$ into Corollary \ref{cor:isop-stability}, we deduce:

 \begin{cor}
 For all $K,L \in \K^2_{+,e}$ in the plane, so that $K$ satisfies (\ref{eq:pbm}) for some $p \leq 0$, we have:
 \[  \frac{P_L(K)^2}{V(K)}\ge   4 V(L) + \frac{2 (1-p)}{2-p}\max \brac{\frac{8}{V(K^{\circ})} \brac{\frac{1}{r} - \frac{1}{R}}^2 , \frac{8}{V(L^{\circ})}  \log^2 \Bigl( \frac{R}{r}\Bigr) }
 \]  \end{cor}
 \noindent
  Recall that by the result of B\"{o}r\"{o}czky--Lutwak--Yang--Zhang \cite{BLYZ-logBMInPlane}, (\ref{eq:pbm}) always holds for any $K \in \K^2_{+,e}$ in the plane with $p=0$, and so by approximation, the above estimate clearly holds for all origin-symmetric $K,L$ in the plane with $p=0$. 
 
 \begin{rem}
  The above isoperimetric-deficit estimates should be compared with the classical estimate of Bonnesen (see \cite[p. 324]{Schneider-Book}): 
 \[
 \frac{P_L(K)^2}{V(K)} \ge 4 V(L) + \frac{V(L)^2}{V(K)} (R-r)^2. 
 \]
 Note that in the isotropic case, when $L = B_2^n$, Bonnesen showed that the constant $V(L)^2 = \pi^2$ improves to $4 \pi$. It is not hard to check that Bonnesen's numeric constant is better than the one in our estimates in the isotropic case when $K$ is itself close to a Euclidean ball. On the other hand, in the non-isotropic case, it is not clear to us whether the estimates are comparable. In addition, note that by Theorem \ref{thm:Ball} and Proposition \ref{prop:local-equiv} we know that when $K$ is $\eps$-close to a Euclidean-ball in $C^2$ then (\ref{eq:pbm}) in fact holds with $p \rightarrow -2$ as $\eps \rightarrow 0$, and so our constant $8$ improves to $ \frac{2 (1-p)}{2-p} 8 \rightarrow 12$, which is already better than Bonnesen's anisotropic constant $\pi^2$ (but worse than the isotropic constant $4 \pi$). 
  \end{rem}

To conclude this subsection, we use the well-known equivalence between the family of all anisotropic isoperimetric inequalities for all (say origin-symmetric) convex bodies $K,L$ and (global) Brunn-Minkowski inequalities for such bodies, to obtain stability estimates for the Brunn--Minkowski inequality. 

\begin{cor}[Stability estimate for Brunn--Minkowski inequality] \label{cor:BM-stable}
Let $K,L \in \K^2_{+,e}$ so that $K+L$ satisfies (\ref{eq:pbm}) for some $p \leq 1$. Then:
\[
\frac{V(K + L)^{\frac{1}{n}}}{V(K)^{\frac{1}{n}} + V(L)^{\frac{1}{n}}} - 1 \geq 2 \frac{1-p}{n-p} \frac{R_{K+L}(K)}{V(K+L)} \geq 2 \frac{1-p}{n-1} \frac{\Var_{dV_{K+L}}\brac{\frac{h_K}{h_K + h_L}}}{V(K+L)} .
\]
\end{cor}
\begin{rem}
Note that both expressions on the right-hand-side above are symmetric in $K$ and $L$ since $\frac{h_L}{h_K + h_L} = 1 - \frac{h_K}{h_K + h_L}$. 
\end{rem}
\begin{proof}[Proof of Corollary \ref{cor:BM-stable}]
By Corollary \ref{cor:isop-stability} applied to the pair $K+L$ and $Q$ for $Q \in \K^2_{+,e}$, we have:
\[
P_Q(K+L) \geq n V(Q)^{\frac{1}{n}} V(K+L)^{\frac{n-1}{n}} \sqrt{1 + B_{K+L}(Q) \brac{\frac{V(K+L)}{V(Q)}}^{\frac{2}{n}} }
\]
where we denote for brevity $B_{K+L}(Q) := \frac{1-p}{n-p} \frac{R_{K+L}(Q)}{V(K+L)}$. Applying the above inequality to $Q=K$, $Q=L$, and summing, since $P_{K}(K+L) + P_{L}(K+L) = P_{K+L}(K+L) = n V(K+L)$, we deduce:
\begin{align*}
\frac{V(K + L)^{\frac{1}{n}}}{V(K)^{\frac{1}{n}} + V(L)^{\frac{1}{n}}} & \geq  \alpha \sqrt{1 + B_{K+L}(K) \brac{\frac{V(K+L)}{V(K)}}^{\frac{2}{n}} } \\
& + (1-\alpha) \sqrt{1 + B_{K+L}(L) \brac{\frac{V(K+L)}{V(L)}}^{\frac{2}{n}} } ,
\end{align*}
with:
\[
\alpha := \frac{V(K)^{\frac{1}{n}}}{V(K)^{\frac{1}{n}} + V(L)^{\frac{1}{n}}} .
\]
Denote the expression on the left-hand-side above by $M$, and recall that $B_{K+L}(K) = B_{K+L}(L)$. By convexity of the function $x \mapsto \sqrt{1 + x^2}$, we conclude that:
\begin{align*}
M & \geq \sqrt{1 + B_{K+L}(K) \brac{\alpha \brac{\frac{V(K+L)}{V(K)}}^{\frac{1}{n}} + (1-\alpha) \brac{\frac{V(K+L)}{V(L)}}^{\frac{1}{n}}}^2} \\
& = \sqrt{1 + 4 B_{K+L}(K) M^2 } .
\end{align*}
Consequently, we see that $4 B_{K+L}(K) < 1$ and:
\[
M \geq \frac{1}{\sqrt{1 - 4 B_{K+L}(K)}} \geq 1 + 2 B_{K+L}(K) ,
\]
verifying the first asserted inequality. As usual, the second inequality follows immediately by (\ref{eq:pbm}). 
\end{proof}
 
  \subsection{Improved stability estimates for all convex bodies with respect to asymmetry}
  
  In this subsection, $K$ and $L$ are only assumed to be two convex bodies in $\Real^n$ (no origin-symmetry is required). 
The anisotropic isoperimetric deficit of $K$ with respect to $L$ is defined as:
\[
\delta(K,L) := \frac{P_{L}(K) }{n V(L)^{\frac{1}{n}} V(K)^{\frac{n-1}{n}}}  -1 .
\]  
The Brunn--Minkowski deficit of $K$ and $L$ is defined as:
\[
\beta(K,L) := \frac{V(K + L)^{\frac{1}{n}}}{V(K)^{\frac{1}{n}} + V(L)^{\frac{1}{n}}} - 1 . 
\]
The relative asymmetry of $K$ with respect to $L$ is defined as:
\[
A(K,L) := \inf_{x_0 \in \Real^n} \frac{V(K \triangle (x_0 + r L))}{V(K)},
\]
where $r > 0$ is such that $V(r L) = V(K)$. The scaling factor of $K$ and $L$ is then defined as:
\[
\sigma(K,L) := \max\brac{r , \frac{1}{r}} = \max \brac{ \frac{V(K)^{\frac{1}{n}}}{V(L)^{\frac{1}{n}}} ,  \frac{V(L)^{\frac{1}{n}}}{V(K)^{\frac{1}{n}}} } .
\]

The classical anisotropic isoperimetric inequality and Brunn--Minkowski inequality state that $\delta(K,L) \geq 0$ and $\beta(K,L) \geq 0$, respectively. It is also known that equality in both cases occurs if and only if $K = x_0 + r L$ for some $r  > 0$ and $x_0 \in \Real^n$, i.e. when $A(K,L) = 0$. A natural question is then whether one can obtain stable versions of these inequalities, lower bounding the deficits $\delta(K,L)$ and $\beta(K,L)$ by some function of $A(K,L)$.

In a groundbreaking work, it was shown by N.~Fusco, F.~Maggi and A.~Pratelli \cite{FuscoMaggiPratelli-SharpIsop} in the isotropic case (when $L=B_2^n$), that for all Borel sets $K$ in $\Real^n$ of finite perimeter:
\begin{equation} \label{eq:delta-A}
\delta(K,L) \geq \frac{1}{C(n)^2} A(K,L)^2 ,
\end{equation}
for some constant $C(n)$ depending on the dimension $n$. Various other estimates on $\delta(K,L)$ as a function of $A(K,L)$ had been previously known (see \cite[p. 326]{Schneider-Book}, \cite{FuscoMaggiPratelli-SharpIsop} and the references therein), but that was the first time the best-possible quadratic dependence on $A(K,L)$ was obtained. The approach in \cite{FuscoMaggiPratelli-SharpIsop} was based on a careful analysis of symmetrization of $K$, and thus confined to the isotropic case $L=B_2^n$. In another major milestone \cite{FMP-Inventiones}, it was shown by A.~Figalli, Maggi and Pratelli that (\ref{eq:delta-A}) holds for all convex bodies $L$ and sets of finite perimeter $K$ with $C(n)$ of the order of $n^{8.5}$. The approach in \cite{FMP-Inventiones} was based on M.~Gromov's optimal-transport proof of the anisotropic isoperimetric inequality \cite{Milman-Schechtman-Book}, which was itself inspired by Kn\"{o}the's proof of the Brunn--Minkowski inequality \cite{KnotheMap}. As already mentioned, it is well-known that the family of anisotropic isoperimetric inequalities for all convex $K,L$ is equivalent to the Brunn--Minkowski inequality for all convex $K,L$; consequently, Figalli--Maggi--Pratelli also obtained in \cite{FMP-Inventiones} the following stable version of the Brunn--Minkowski inequality for convex bodies $K,L$:
\begin{equation} \label{eq:beta-A}
\beta(K,L) \geq \frac{1}{C(n)^2} \frac{A(K,L)^2}{\sigma(K,L)} . 
\end{equation}
 Soon after, Figalli--Maggi--Pratelli gave a second simpler proof of (\ref{eq:beta-A}) (and hence of (\ref{eq:delta-A})) in \cite{FMP-StableBM} for all convex bodies $K,L$, which however yielded an exponential dependence of $C(n)$ on $n$. By carefully refining their argument, A.~Segal was able in \cite{Segal-StableBM} to improve their estimate (in both (\ref{eq:delta-A}) and (\ref{eq:beta-A}) for all convex bodies $K,L$) to $C(n)$ of the order of $n^{3.5}$ (improving further to $n^3$ in (\ref{eq:delta-A}) when $K$ is origin-symmetric, and in (\ref{eq:beta-A}) when both $K$ and $L$ are origin-symmetric). An improvement over Segal's estimates by a factor of $\sqrt{n}$ was claimed by Harutyunyan in \cite{Harutyunyan-Error}, but unfortunately the proof contains a mistake (formula (3.18) in \cite{Harutyunyan-Error} is off by a factor of $\sqrt{n}$), a correction of which precisely cancels out the claimed improvement. To the best of our knowledge, prior to our own results, the best known estimate on $C(n)$ was due to Segal. 

\medskip

In this subsection, we improve the best-known dependence on $C(n)$ further in both (\ref{eq:delta-A}) and (\ref{eq:beta-A}) for all convex bodies (regardless of origin-symmetry) as follows:
\begin{thm} \label{thm:improved-stability}
The stable versions of the anisotropic isoperimetric inequality (\ref{eq:delta-A}) and Brunn--Minkowski inequality (\ref{eq:beta-A}) hold with $C(n) \leq C n^{\frac{11}{4}}$ for all convex bodies $K,L$ in $\Real^n$. 
\end{thm}

All the proofs in \cite{FMP-StableBM,Segal-StableBM,Harutyunyan-Error} employ the following so-called ``Poincar\'e-type trace inequality" proved by Figalli--Maggi--Pratelli in \cite{FMP-StableBM}. We prefer referring to it as a ``Cheeger-type boundary inequality", since it involves $L^1$ rather than $L^2$ norms. 

\begin{lem}[Figalli--Maggi--Pratelli] \label{lem:FMP}
Let $K$ be a convex body in $\Real^n$ with:
\[
r(K) B_2^n \subset K \subset R(K) B_2^n .
\]
Then:
\[
\inf_{c \in \Real} \int_{\partial K} \abs{f(x) - c} dx \leq \frac{\sqrt{2}}{\log 2} \frac{n R(K)}{r(K)} \int_K \abs{\nabla f} dx \;\;\; \forall f \in C^1(K) . 
\]
\end{lem}

Our improved estimates in Theorem \ref{thm:improved-stability} are based on the following improvement of Lemma \ref{lem:FMP}. We denote by $C_{Che}(K)$ the (reciprocal of the) Cheeger constant of $K$, namely the best constant in the following inequality:
\[
\int_K \abs{f(x) - m_f} dx \leq C_{Che}(K) \int_K \abs{\nabla f} dx \;\;\; \forall f \in C^1(K) ,
\]
where $m_f$ denotes a median value of the law of $f$ under the Lebesgue measure on $K$. 

\begin{lem} \label{lem:improved-FMP}
With the same assumptions as in Lemma \ref{lem:FMP}:
\[
\int_{\partial K} \abs{f(x) - m_f} dx \leq \frac{n C_{Che}(K) + R(K)}{r(K)} \int_K \abs{\nabla f} dx \;\;\; \forall f \in C^1(K) . 
\]
\end{lem}

\begin{rem}
By a result of Kannan--Lov\'asz--Simonovits \cite{KLS} (see also Lov\'asz--Simonovits \cite{LSLocalizationLemma} for a slightly worse numeric constant), it is known that $C_{Che}(K) \leq \frac{2}{\log 2} R(K)$ for any convex body $K$, and so the estimate in Lemma \ref{lem:improved-FMP} is always better than the one in Lemma \ref{lem:FMP}, up to numeric constants. Our improvement is based on the fact that $C_{Che}(K)$ can be significantly smaller than $R(K)$ when an appropriate affine transformation is applied to $K$. 
\end{rem}

\noindent 
Our proof of Lemma \ref{lem:improved-FMP} is based on the same idea as in the proof of Theorem \ref{thm:DK}. 

\begin{proof}[Proof of Lemma \ref{lem:improved-FMP}]
Using convexity and integrating-by-parts:
\begin{align*}
\int_{\partial K} \abs{f - m_f} dx & \leq \frac{1}{r(K)} \int_{\partial K} \abs{f - m_f} \scalar{x,\nu_{\partial K}} dx \\
& =   \frac{1}{r(K)} \int_K \text{div}(\abs{f - m_f} x) dx \\
& = \frac{1}{r(K)} \brac{ n \int_K \abs{f - m_f} dx + \int_K \scalar{\nabla |f - m_f| , x} dx} \\
& \leq  \frac{n C_{Che}(K) + R(K)}{r(K)} \int_K \abs{\nabla f} dx . 
\end{align*}
Note that used that $\abs{\nabla |f - m_f|} = \abs{\nabla (f - m_f)} = \abs{\nabla f}$ on $\set{f \neq m_f}$ and the well-known $\int_{\set{f = c}} \abs{\nabla |f - c|} dx =0$, which is valid for all (say) locally Lipschitz functions $f$. 
\end{proof}

The improvement over the result of \cite{Segal-StableBM} is now achieved by substituting the application of Lemma \ref{lem:FMP} by that of Lemma \ref{lem:improved-FMP}. Observe that the desired estimates (\ref{eq:delta-A}) and (\ref{eq:beta-A}) are both invariant under a simultaneous application of a linear transformation to both $K$ and $L$, as well as individually translating $K$ or $L$ (note that $P_L(K)$ is indeed invariant under translation of $L$). Consequently, we may always assume that $K$ is in a desirable affine position. In \cite{FMP-StableBM,Segal-StableBM,Harutyunyan-Error}, the authors employed John's position, which ensures that $R(K)/r(K) \leq n$ for general convex bodies, and $R(K) / r(K) \leq \sqrt{n}$ for origin-symmetric ones. To obtain our improvement, we apply an affine transformation to $K$ so that $K$ is in isotropic position, which was already defined in Subsection \ref{subsec:general-DK}. As explained there, it is guaranteed that when $K$ is in isotropic position then:
\begin{equation} \label{eq:FMP-isotropic}
r(K) \geq \sqrt{\frac{n+2}{n}} ~,~ R(K) \leq \sqrt{(n+2)n} .
\end{equation}
As for the (reciprocal of the) Cheeger constant $C_{Che}(K)$, it is known by results of Maz'ya, Cheeger, Buser and Ledoux (see \cite{LedouxSpectralGapAndGeometry}) that for all convex bodies $K$, $C_{Che}(K)$ is equivalent (up to numeric constants, independent of the dimension) to the Poincar\'e constant $C_{Poin}(K)$, defined in Subsection \ref{subsec:general-DK}. Recall from that subsection that the conjecture of Kannan, Lov\'asz and Simonovits \cite{KLS} predicts that if $K$ is convex and in isotropic position then $C_{Che}(K) \leq C$ for some universal numeric constant $C>0$, independent of the dimension $n$; also recall that the presently best-known estimate for the KLS conjecture is due to Lee and Vempala \cite{LeeVempala-KLS}, who showed that:
\begin{equation} \label{eq:FMP-LV}
C_{Che}(K) \leq C \sqrt[4]{n} 
\end{equation}
for all isotropic convex bodies $K$ in $\Real^n$. This results in a strict improvement in Lemma \ref{lem:improved-FMP} over the estimates in \cite{FMP-StableBM,Segal-StableBM,Harutyunyan-Error}, which employed Lemma \ref{lem:FMP} -- an improvement by an order of $n^{\frac{3}{4}}$ for general convex bodies, and by an order of $n^{\frac{1}{4}}$ for origin-symmetric ones -- thereby  
yielding Theorem \ref{thm:improved-stability}. Note that confirmation of the KLS conjecture would further improve our estimate on $C(n)$ down to an order of $n^{\frac{10}{4}} = n^{2.5}$. 

\medskip

While all other arguments remain unchanged, and the reader can easily reconstruct all the remaining details by following Segal's proof from \cite{Segal-StableBM}, for the reader's convenience, we sketch the most essential steps in the proof.

\begin{proof}[Sketch of proof of Theorem \ref{thm:improved-stability}]
By scaling, one can assume that $K, L$ have equal volumes. Since $A(K,L) \leq 2$, we may assume that $\delta(K,L) \leq 1$ (say), otherwise there is nothing to prove. 
Let $$T(x) = \nabla \Phi(x)$$ be the Brenier optimal-transport map pushing forward the Lebesgue measure on $K$ onto the Lebesgue measure on  $L$, with $\Phi : K \rightarrow \Real$ convex. The classical argument of Kn\"othe and Gromov (see \cite[p. 385]{Segal-StableBM}) then yields:
\[
V(L) \delta(K,L) \geq \int_K \brac{\frac{\text{tr} \; d T(x)}{n} - (\text{det} \; d T(x))^{1/n}} dx  . \]
The usual arithmetic-geometric means (AM-GM) inequality applied to the (positive) eigenvalues of $d T(x) = \nabla^2 \Phi(x) \geq 0$ already shows that the right-hand-side is non-negative, and the goal is to obtain a lower bound depending on $A(K,L)$.

Applying a stable version of the AM-GM inequality and Cauchy--Schwarz, Segal obtains (see \cite[formula (10)]{Segal-StableBM}):
\begin{equation} \label{eq:Segal1}
2n  V(L) \sqrt{\delta(K,L) } \geq \int_{K} \norm{d T(x) - {\rm Id} } dx , 
\end{equation}
where as usual $\norm{\cdot}$ denotes the Hilbert-Schmidt norm. By translating $L$ we may always assume that the median of $T_i$ satisfies $m_{T_i} = 0$ for each $i=1,\ldots,n$. Now applying Lemma \ref{lem:improved-FMP} instead of Lemma \ref{lem:FMP} with $f(x) = T_i(x)$ to each coordinate $i=1,\ldots,n$, and passing back and forth between the $\ell_2^n$ to the $\ell_1^n$ norm (the latter step was not accounted for in \cite{Harutyunyan-Error}), we obtain:
\begin{align}
\nonumber & \int_{K} \norm{d T(x) - {\rm Id} } dx \geq \frac{1}{\sqrt{n}} \int_K \sum_{i=1}^n \abs{\nabla T_i(x) - e_i} dx \\
\label{eq:Segal2} & \geq \frac{1}{\sqrt{n} C_{Che,\partial}(K)} \int_{\partial K} \sum_{i=1}^n \abs{T_i(x) - x_i} dx \geq \frac{1}{\sqrt{n} C_{Che,\partial}(K)} \int_{\partial K} \abs{T(x) - x} dx ,
\end{align}
where:
\[
C_{Che,\partial}(K) := \frac{n C_{Che}(K) + R(K)}{r(K)} . 
\]
While the passage between the $\ell_2^n$ to the $\ell_1^n$ norms seems wasteful, we have not been able to make this step more efficient. 

Now, a beautiful geometric argument of Figalli--Maggi--Pratelli (see \cite[p. 387]{Segal-StableBM}) verifies that (recall $V(K) = V(L)$):
\begin{equation} \label{eq:Segal3}
\frac{1}{2} V(K) A(K,L) \leq \frac{1}{2} V(K \triangle L) =  V(K \setminus L)  \le \int_{\partial K} \abs{T(x) - x} dx .
\end{equation}
Combining (\ref{eq:Segal1}), (\ref{eq:Segal2}) and (\ref{eq:Segal3}), we deduce:
\[
 A(K,L) \leq   4 n^{\frac{3}{2}}  C_{Che, \partial}(K)  \sqrt{\delta(K,L) }.
\]
As already explained above, it remains to translate $K$ and apply a simultaneous linear transformation to both $K$ and $L$, so that $K$ is in isotropic position. The estimates (\ref{eq:FMP-isotropic}) and (\ref{eq:FMP-LV}) then yield:
\[
C_{Che,\partial}(K) \leq C n^{\frac{5}{4}} ,
\]
concluding the proof of Theorem \ref{thm:improved-stability} for the isoperimetric deficit (\ref{eq:delta-A}). The estimate on the Brunn--Minkowski deficit (\ref{eq:beta-A}) follows by the standard equivalence between the family of anisotropic isoperimetric inequalities for all convex bodies $K,L$ and the Brunn--Minkowski inequality for all such pairs, as in the previous subsection (see \cite[p. 203]{FMP-Inventiones}). 
\end{proof}

\bibliographystyle{plain}
\bibliography{../../../ConvexBib}

\end{document}